\documentclass[aos,preprint]{imsart}

\RequirePackage[OT1]{fontenc}
\RequirePackage{amsthm,amsmath}
\RequirePackage[numbers]{natbib}
\RequirePackage[colorlinks,citecolor=blue,urlcolor=blue]{hyperref}

\startlocaldefs
\usepackage{float}
\usepackage{amssymb}
\usepackage{stmaryrd}
\usepackage{graphicx}
\usepackage{mathtools}
\usepackage{color}
\makeatletter
\@ifundefined{showcaptionsetup}{}{%
 \PassOptionsToPackage{caption=false}{subfig}}
\usepackage{subfig}
\usepackage[normalem]{ulem}

\numberwithin{equation}{section}
\numberwithin{figure}{section}
\theoremstyle{plain}
\newtheorem{thm}{\protect\theoremname}
  \theoremstyle{plain}
	\numberwithin{thm}{subsection} 
\newtheorem{cor}[thm]{\protect\corollaryname}
  \theoremstyle{plain}
\newtheorem{lem}[thm]{\protect\lemmaname}
\newtheorem{defn}[thm]{Definition}

\providecommand{\corollaryname}{Corollary}
\providecommand{\lemmaname}{Lemma}
\providecommand{\theoremname}{Theorem}
\endlocaldefs

\usepackage[colorinlistoftodos]{todonotes}
\begin{document}

\global\long\def\R{\mathbb{R}}
\global\long\def\E{\mathbb{E}}
\newcommand{\cN}{{\mathcal{N}}}
\newcommand{\bS}{{\bf{S}}}
\newcommand{\bR}{{\bf{R}}}
\newcommand{\ellOneStar}{\ell_1^*}
\newcommand{\ellMStar}{\ell_m^*}
\newcommand{\lamOneStar}{\lambda_1^*}
\newcommand{\lamMStar}{\lambda_m^*}
\newcommand{\goto}{\rightarrow}
\newcommand{\DelAsy}{{\Delta^a}}
\newcommand{\KAsy}{{K^a}}
\newcommand{\KEmp}{{K^e}}
\newcommand{\DelEmp}{{\Delta^e}}

\newcommand{\etammnl}{ \eta_{\mbox{\footnotesize \sc mm}}}
\newcommand{\etammst}{ \eta_{\mbox{\footnotesize \sc mmst}}}
\newcommand{\etapnl}{ \eta_{\mbox{\footnotesize \sc pnl}}}
\newcommand{\etagst}{ \eta_{\mbox{\footnotesize \sc gst}}}
\newcommand{\etauasn}{\eta_{\mbox{\sc \footnotesize pf}}^*}
\newcommand{\elld}{\dot{\ell}}
\newcommand{\beq}{\begin{equation}}
\newcommand{\eeq}{\end{equation}}
\newcommand{\Dell}{\partial_{\ell}}
\newcommand{\Deta}{\partial_\eta}
\newcommand{\etaa}{\eta^a}
\newcommand{\bitem}{\begin{itemize}}
\newcommand{\eitem}{\end{itemize}}
\newcommand{\Sep}{\SEP}
\newcommand{\eps}{\epsilon}
\newcommand{\cI}{{\mathcal{I}}}
\newcommand{\cL}{{\mathcal{L}}}
\newcommand{\cS}{{\mathbb{S}}}
\newcommand{\semidefp}{{SD(p)}}
\newcommand{\HS}{\hat{\Sigma}}
\newcommand{\dell}{\dot{\ell}}
\newcommand{\elldot}{\dot{\ell}}
\newcommand{\hS}{\hat{\Sigma}}
\newcommand{\tD}{\tilde{\Delta}}
\newcommand{\tnu}{\tilde{\nu}}
\newcommand{\bW}{\textbf{W}}
\newcommand{\tu}{\tilde{u}}
\newcommand{\tA}{\tilde{A}}
\newcommand{\teta}{\tilde{\eta}}
\newcommand{\tv}{\tilde{v}}
\newcommand{\tc}{\tilde{c}}
\newcommand{\ts}{\tilde{s}}
\newcommand{\tb}{\tilde{b}}
\newcommand{\cW}{{\cal W}}
\newcommand{\cWnr}{\cW_{n,r}}
\newcommand{\SR}{{ \mbox{\sc sr}}}
\newcommand{\SEP}{{ \mbox{\sc sep}}}
\newcommand{\RSRG}{{ \mbox{\sc rsrg}}}
\newcommand{\RSEPG}{{ \mbox{\sc rsepg}}}
\newcommand{\LDA}{{ \mbox{\sc lda}}}
\newcommand{\MV}{{ \mbox{\sc mv}}}
\newcommand{\gammamval}{{ 0.618033...}}
\newcommand{\gammamvals}{{ (\sqrt{5}-1)/2}}
\newcommand{\EtaRNP}{{\boldsymbol{\eta}_{r,p_n}}}
\newcommand{\EtaRNPa}{{\boldsymbol{\eta}_{r,p}}}
\newcommand{\EtaRNPf}{\boldsymbol{\eta}_{r,p}}

\begin{frontmatter}
\title{Optimal Covariance Estimation for \\ 
 Condition Number Loss \\
 in the Spiked  model}
\runtitle{Optimal Eigenvalue Shrinkage For $\kappa-$Loss}

\begin{aug}
\author{\fnms{David} \snm{L. Donoho}\thanksref{t1}
\ead[label=e2]{donoho@stanford.edu}}
\and
\author{\fnms{Behrooz} \snm{Ghorbani} \thanksref{t2}
\ead[label=e1]{ghorbani@stanford.edu}}


\thankstext{t1}{Supported by NSF DMS 1418362 and 1407813.}
\thankstext{t2}{Supported by Stanford's Caroline and Fabian Pease Graduate Fellowship.}
\runauthor{D.L.Donoho and B.Ghorbani}

\affiliation{Stanford University}

\address{B.Ghorbani \\ 
Department of Electrical Engineering\\
David Packard Building\\
350 Serra Mall\\
Stanford, California 94305 \\
\printead{e1}\\
\phantom{E-mail:\ }}
\address{D.L.Donoho \\
Department of Statistics\\
Sequoia Hall\\
390 Serra Mall\\
Stanford, California 94305\\
\printead{e2}\\
}
\end{aug}

\begin{abstract}
We study estimation of the covariance matrix
under relative condition number loss $\kappa(\Sigma^{-1/2} \hat{\Sigma} \Sigma^{-1/2})$,
where $\kappa(\Delta)$ is the condition number of matrix $\Delta$, and $\hat{\Sigma}$
and $\Sigma$ are the estimated and theoretical covariance matrices.
Optimality in $\kappa$-loss provides optimal guarantees
in two stylized applications: Multi-User Covariance Estimation 
and Multi-Task Linear Discriminant Analysis.

We assume the  so-called {\it spiked covariance model} for  $\Sigma$, 
and exploit recent advances in understanding that model, to derive  a  nonlinear shrinker  
which is asymptotically optimal among orthogonally-equivariant procedures. 
In our asymptotic study, the number of variables $p$
is comparable to the number of observations $n$. 

The form of the optimal nonlinearity 
depends on the aspect ratio $\gamma=p/n$
of the data matrix and on the top eigenvalue of $\Sigma$. 
For $\gamma > \gammamval$, 
even dependence on the top eigenvalue can  be avoided. 

The optimal shrinker has two notable properties.
First, when $p/n \rightarrow \gamma \gg 1$ is large, 
it shrinks even very large eigenvalues substantially, by a factor $1/(1+\gamma)$. 
Second, even for moderate $\gamma$, 
certain highly statistically significant eigencomponents
will be completely suppressed. We show that when $\gamma \gg 1 $ is large, purely
diagonal covariance matrices can be optimal, despite the
top eigenvalues being large and the empirical eigenvalues being highly
statistically significant. This aligns with practitioner 
experience.

We identify intuitively reasonable 
procedures with small worst-case relative regret \-- the simplest
being generalized soft thresholding
having threshold at the bulk edge and  
slope $(1+\gamma)^{-1}$ above the bulk. For $\gamma < 2$
it has at most a few percent relative regret.

\end{abstract}



\end{frontmatter}

\maketitle
\section{Introduction} \label{sec:intro}
Sixty years ago, Charles Stein \cite{stein1956} made the surprising observation
that, when estimating the covariance matrix underlying a dataset $X = (X_{i,j})$ 
with $p$ variables and $n$ observations, 
if $p$  and $n$ are both large, and comparable in size,
 one should shrink the eigenvalues of the empirical covariance matrix away from their raw empirical values.

In the decades since, shrinkage estimates of covariance have been studied by many researchers,
and Stein's insight has emerged as pervasive and profound. {Dozens of citations relevant through 2013 are given in \cite{donoho2013optimalShrinkage}; more recent work includes \cite{ledoit2014nonlinear, wang2015shrinkage, lam2016nonparametric, bodnar2016direct, kubokawa2014estimation}. 
Broadly speaking, the literature has found 
that the appropriate way to shrink the eigenvalues depends both on the use to be 
made of the estimated covariance matrix and on the properties of the underlying covariance matrix;
for procedures optimal under various assumptions,
see among others \cite{ledoit2014optimal,Ledoit2012nonlinear,donoho2013optimalShrinkage}.

\subsection{Our Focus}
This paper studies eigenvalue shrinkage under five specific assumptions:
\begin{description}
\item[{\bf [$\kappa$-Loss]}]  We consider {\it relative condition number loss}:
\[
L(\hat{\Sigma},\Sigma) = \kappa(\Sigma^{-1/2} \hat{\Sigma} \Sigma^{-1/2}),
\] 
where  $\hat{\Sigma}$ denotes the estimated
and $\Sigma$ the true underlying theoretical covariance, and $\kappa(\Delta)$ denotes the condition number of matrix $\Delta$.
\item[{\bf [Equivariance]}] We consider  
{\it orthogonally equivariant procedures,} 
defined as follows. Let $S = \frac{1}{n}X'X$ denote
the empirical covariance, and $\hat{\Sigma} = \hat{\Sigma}(S)$ denote
an estimator of interest. 
We say that $\hat{\Sigma}$ is orthogonally-equivariant (OE) 
if $\hat{\Sigma}(U' S U) = U' \hat{\Sigma}(S)U$ for all $U \in O(p)$.  
Such procedures are in a certain sense coordinate-free.


\item[{\bf [Spike]}] {\it Spiked covariance models}, where $\Sigma$ is the 
identity outside an $r$-dimensional subspace.  Here $r$ is fixed, 
and the  $r$ top eigenvalues of $\Sigma$, $\ell_1 , \ell_2, \dots \ell_r$, 
say, exceed $1$, with all later eigenvalues equalling $1$.
\item [{\bf [PGA]}] We focus on {\it proportional growth asymptotics},  
where the data matrix sizes grow proportionally large \-- $p, n \goto \infty$
with asymptotic aspect ratio $p/n \goto \gamma \in (0,\infty)$.
\item [{[Distribution]}] 
In this paper we assume the classical model, where the rows of $X$ are 
i.i.d $N(0,\Sigma)$\footnote{We conducted simulation studies
of $\kappa$-loss performance of this paper's shrinkers 
across a broader collection of situations 
with the same bulk spectrum. Our observed performance was consistent with
our theorems for the Normal case when we used a linear generative model $X = \Sigma^{1/2} Z$ with
i.i.d. $Z$ having variance $1$ and finite fourth moments. Accordingly, we believe that work
of \cite{bai2012sample, Benaych-Georges2011, dobriban2016pca} 
can extend our results unchanged across a range of non-normal assumptions.
We are mainly interested here in the structure of optimal procedures and the phenomena they exhibit.}. 
\end{description}

\subsection{Optimality}

Our assumptions allow us to evaluate the optimal 
asymptotic loss among orthogonally equivariant
procedures.  

\begin{thm} \label{thm:asymptotic_optimal_loss}
{\bf (Optimal Asymptotic Loss)}
The following limit exists almost surely:
\[
     \lim_{n \goto \infty} \inf_{\hat{\Sigma} \in OE} L(\hat{\Sigma},\Sigma) =_{a.s.}  L^*(\ell_1,\dots,\ell_r; \gamma);
\]
say. Here the infimum is over orthogonally equivariant procedures.
Definition \ref{def:kappa1star} in
(\ref{eq:kappa1defn}) below  
specifies a function $\kappa_1^*$  depending only on the
aspect ratio $\gamma$ and on the top spike eigenvalue $\ell_1$, for which:
\[
 L^*(\ell_1,\dots,\ell_r; \gamma) =  \kappa_1^*(\ell_1; \gamma).
\]
\end{thm}

Let $\textbf{S}=\frac{1}{n}X'X$  be the usual empirical second-moment matrix and let $\textbf{S} = V \Lambda V'$ be its
usual eigendecomposition, where $V$ is orthogonal and $\Lambda$ is diagonal with the 
ordered eigenvalues $\lambda_1 \geq \lambda_2 \dots \geq \lambda_p$ along the diagonal.

We derive in Section \ref{sec:ShrinkMultiSpike} below, in 
Theorem \ref{thm:OptimalShrinkage} et seq., 
a closed-form expression for an  asymptotically optimal
nonlinearity $\eta^* : \bR^+ \mapsto \bR^+$.
That theorem defines
a nonlinearity $\eta^*(\cdot) \equiv \eta_m^*(\cdot; \lambda_{1},\gamma)$ 
having two tuning parameters: 
$\gamma = p/n$, the aspect ratio; and
$\lambda_1$, the limiting value of the top
empirical eigenvalue under our asymptotic model.
A fully data-driven nonlinearity $\eta^e$
is obtained by using for the (unknown)  
limit eigenvalue $\lambda_1$ tuning parameter 
simply the top empirical eigenvalue
$\lambda_{1,n}$, which is observable. 
This nonlinearity is applied separately
to each of the empirical eigenvalues $(\lambda_{i,n})$,
 producing the diagonal matrix $\eta^e(\Lambda) \equiv diag(\eta^e(\lambda_{i,n}))$
which sits at the core of the orthogonally-equivariant 
covariance estimator $\hat{\Sigma}^e = V \eta^e(\Lambda) V'$ .

\begin{thm} \label{thm:asymptotic_optimal_nonlinearity}
{\bf (Asymptotically Optimal Nonlinearity)}
The estimator $\hat{\Sigma}^e$ is asymptotically optimal 
among orthogonally equivariant procedures
under relative condition number loss: 
\[
     \lim_{n \goto \infty}L(  \hat{\Sigma}^e ,\Sigma) =_{a.s.}  \kappa_1^*(\ell_1; \gamma).
\]
\end{thm}

We also show -- see Theorem \ref{thm:MultiSingle} -- 
that in case $ \lim_n p_n/n = \gamma  > \gamma_m^* \equiv \gammamvals = \gammamval,$ 
there is even an asymptotically optimal nonlinearity  that does not
depend on the top eigenvalue $\lambda_{1}$, 
but instead only on $\gamma$ (which of course is known);
the nonlinearity is denoted
$\eta_1^* \equiv \eta_1^*(\cdot; \gamma)$ 
and derived in Theorem \ref{thm:spike_1}.
Moreover even for $\gamma < \gamma_m^*$,
$\eta_1^*(\cdot; \gamma)$ is asymptotically optimal in the single-spike situation, $r=1$; and is within
several percent of optimal for the multi-spike case $r > 1$.

A different option also avoids tuning by the top eigenvalue,
and offers theoretical performance guarantees (of a weaker sort)
even for $\gamma < \gamma_m^*$. This option intentionally mis-tunes the tuning parameter $\lambda_1$ to infinity, producing a still well-defined 
nonlinearity $\etammnl (\cdot) \equiv \eta^*_m( \cdot ; \infty, \gamma)$. Section \ref{sec:MinimaxInterp}
shows that this tuning parameter-free 
nonlinearity minimaxes the $\kappa$-loss across all spike models
of all orders $r$ and all configurations $(\ell_i)$.



\subsection{Insights}

Figure \ref{fig:NonlinMultiSpike} depicts the optimal nonlinearity
for various choices of parameters. Several properties seem surprising:

\begin{figure}
\centering
\includegraphics[height=3in]{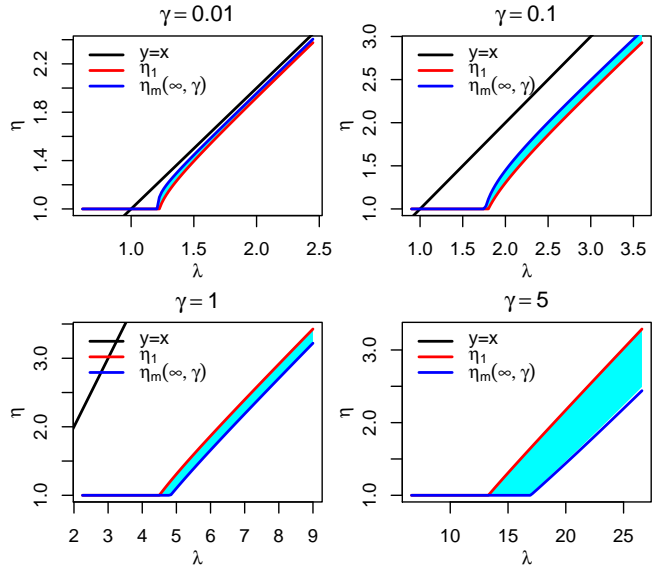}
\caption{{\bf Optimal shrinkage nonlinearities.} 
Different panels present different aspect ratios $\gamma=p/n$. 
Each panel's black curve gives the identity line $y=x$.
Each red curve depicts the {\it optimal single-spike nonlinearity} $\eta_1^*$. 
Each aqua-shaded area depicts the range of the {\it optimal multi-spike nonlinearity}
$\eta_m^*(\lambda; \lambda_1, \gamma)$ as $\lambda_1$ varies from $\lambda$ to $\infty$. 
The {\it minimax nonlinearity} $\etammnl$ is shown in each 
panel as the blue curve. In each panel, all depicted shrinkers lie significantly below the identity line $y=x$, in the case $\gamma= 5$, so much so that the identity line is no longer visible.
}
\label{fig:NonlinMultiSpike}
\end{figure}

\begin{enumerate}

\item {\it Asymptotic slope $1/(1+\gamma)$}. A `natural' psychological expectation for optimal shrinkage
is that it `ought to' minimally displace any very large empirical eigenvalues, 
producing `shrunken' outputs that are still
relatively close to inputs $\eta^*(\lambda) \approx \lambda$.  
Contrary to this belief, the optimal shrinker has asymptotic slope 
$1/(1+\gamma)$: $\eta^*(\lambda) \sim \lambda/(1+\gamma)$ as $\lambda \goto \infty$.
This entails very substantial displacement of very large eigenvalues \-- in fact shrinkage by more than 50\%
when $\gamma > 1$: $\eta^*(\lambda) < \lambda/2$ as $\lambda \goto \infty$.

\item {\it Dead-Zone 1.} 
The nonlinearity $\eta^*$  has a {\it dead zone}, 
an interval throughout which it collapses all empirical eigenvalues
 to output value $\eta=1$.
The {\it dead-zone threshold} $\lambda_{1}^+(\gamma)$ gives the upper edge of this interval for the largest spike: 
\[
  \eta^*(\lambda_1) = 1, \qquad \lambda_1 < \lambda_{1}^+(\gamma).
\]

A `natural' psychological expectation for the dead zone is that it `ought to'
agree with the so called {\it bulk edge} $\lambda_+(\gamma) = ( 1+ \sqrt{\gamma})^2$.
Indeed traditional statistical hypothesis tests for the number of eigencomponents in a spiked model 
consider a null hypothesis where $\Sigma =I$ (i.e. all eigenvalues are one), and
spiked alternatives where $\ell_1 > \dots > \ell_r > 1$. Such tests
declare $\lambda_1$ to be statistically significant evidence against the
null whenever it noticeably exceeds the bulk edge
\cite{johnstone2001distribution,kritchman2009non,perry2010minimax,onatski2013asymptotic,passemier2012determining}.
This leads to the `natural' presumption   that the `correct' place to terminate the dead zone is
at the bulk edge. 

Contradicting this, the dead-zone threshold $\lambda_1^+(\gamma)$
is noticeably larger than $\lambda_+(\gamma)$; see Figure \ref{fig:DeadZone}.
This gap implies that we can have $\eta^*(\lambda_1)=1$ even though $\lambda_1$
is large enough to provide conclusive statistical 
evidence for existence of a non-null spike.  
Optimal shrinkage does not at all agree with statistical
significance in such cases.

For an example, imagine that there are $\gamma=50$ times more variables than
observations (not unusual in some genomics applications). 
Then $\lambda_+(\gamma) = 65.14$ while $\lambda_1^+(\gamma) = 103.887$.
Unless the top theoretical eigenvalue 
$\ell_1 > { 52.92}$ \-- 
i.e. more than {50} times larger than the average eigenvalue \--
the eigenvalues {\it all} fall in the dead zone.

\begin{figure}
\centering
\includegraphics[height=3in]{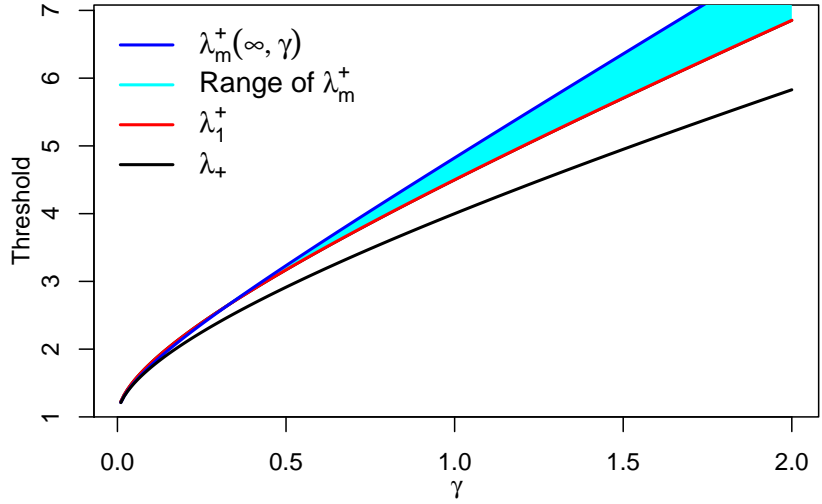}
\caption{{\bf Thresholding behavior of optimal shrinkers; comparison to bulk edge.} 
The dead zone thresholds $\lambda_1^+(\gamma)$ (red), $\lambda_{m}^+(\infty,\gamma)$ (blue) 
and (shaded aqua) the full range of $\lambda_{m}^+(\lambda_1,\gamma)$. All empirical eigenvalues
smaller than these thresholds are collapsed to $1$, under the single-spike nonlinearity
$\eta_1^*$, the multi-spike nonlinearity $\eta_m^*$, and under the minimax nonlinearity,
respectively.
For comparison,  the bulk edge
$\lambda_+(\gamma)$  is  also shown, in black.
Note the ordering $\lambda_+(\gamma) < \lambda_1^+(\gamma) \leq \lambda_{m}^+(\lambda_1;\gamma) < \lambda_{m}^+(\infty;\gamma)$.
}
\label{fig:DeadZone}
\end{figure}



\item {\it Dead-Zone 2.} When the top eigenvalue escapes the dead zone, it can happen that 
some noticeably large secondary eigenvalues  will 
still be collapsed to $1$.
For secondary eigenvalues, we denote by $\lambda^+_m \equiv \lambda^+_{m}(\lambda_1, \gamma)$
the upper edge of the set where $\eta^*=1$:
\[
  \forall i >1, \;\eta^*(\lambda_i) = 1, \qquad \lambda_i < \lambda_{m}^+.
\]
Figure \ref{fig:DeadZone} shows that for moderately large $\gamma$, $\lambda_m^+$ can be 
substantially larger than $\lambda_1^+(\gamma)$. It were as if the existence of at least one
included eigencomponent makes the optimal shrinker be even more demanding of subsidiary
eigenvalues.
Any eigencomponents between $\lambda_+$ and $\lambda^+_m$ are, for large $n$,
overwhelmingly statistically significant according to  well-designed tests  \cite{johnstone2001distribution}
\--  yet the optimal rule $\eta^*(\lambda_i) = 1$ effectively views them as useless,
and ignores them.
Figure \ref{fig:DeadZone} shows that,
for $\gamma$ fixed,
 $\lambda_m^+$ varies  with $\lambda_1$. 
\end{enumerate}

The severe shrinkage induced by $\eta^*(\cdot)$ has precedent: 
the asymptotic slope $1/(1+\gamma)$ of this paper's optimal shrinker $\eta_{m}^*$   
occurred previously in the  eigenvalue shrinkage literature \--
most immediately, \cite{donoho2013optimalShrinkage} 
showed\footnote{See Lemma 7.1, Table 3 and Table 4 in \cite{donoho2013optimalShrinkage}} that, under the spiked model,
the optimal nonlinearities
also have asymptotic slopes $1/(1+\gamma)$ in three important cases:
under Stein's loss; under Frobenius norm loss on the precision discrepancy $\hat{\Sigma}^{-1} - \Sigma^{-1}$;
and under operator norm loss on the relative discrepancy $\Sigma^{-1/2} \hat{\Sigma} \Sigma^{-1/2} - I$.

 \subsection{Other Performance Comparisons}

 In Section \ref{sec:PerformanceComparisons} we  compare and
 contrast  performance of the optimal shrinker 
 with some popular and well-known approaches. 
 In the process, we learn
 surprising and revealing things about the well-known approaches.

We first consider worst case analysis
across all spike models.
We identify two explicit  nonlinearities delivering the same worst-case guarantees as
 the optimal nonlinear shrinker:
 \bitem
 \item {\it Minimax soft thresholding}:  
The {\it generalized soft thresholding nonlinearity} 
with slope $b$ and soft threshold $\lambda_0$ has the form
\[
  \etagst(\lambda; b , \lambda_0) = 1 + b \cdot (\lambda-\lambda_0)_+.
\]
Tuning this rule  with slope $b = b^*(\gamma) = \frac{1}{1+\gamma}$, (i.e.  the 
asymptotic slope of the optimal shrinker),  and threshold at the bulk edge  $\lambda_0 = \lambda_+(\gamma)$, 
produces a minimax estimator\footnote{i.e. this rule minimaxes 
the asymptotic loss -- min across rules and max across spike configurations.
For space reasons we do not prove the minimaxity property in this paper.}. 
 \item {\it Precision Nonlinearity}: 
 the nonlinearity   $\etapnl$ derived in \cite{donoho2013optimalShrinkage}
 for asy. optimal estimation of the so-called precision matrix 
 $\Sigma^{-1}$ under Frobenius norm loss; 
 see (\ref{eq:optPrecision}) below.  $\etapnl$ also turns out to be asy. minimax for 
 relative condition number loss.
 \eitem
We next consider the relative regret, i.e. the percentage extra loss suffered by a given rule; and we
show that both of these rules suffer minimally \--
from a few to several percent additional loss, as compared to the optimal rule.

\subsection{Motivation for Relative Condition Number Loss}

Our interest in relative condition number loss estimation
has been driven both by the mathematical analysis 
and by the potential applications.

\subsubsection{Mathematical Structure}

Optimal shrinkers under the spiked model were derived  
in \cite{donoho2013optimalShrinkage} for 26 different 
loss functions, including Frobenius loss, Stein loss, and
Operator loss and Nuclear norm loss.
In that paper the discrepancy between the estimated and 
underlying covariance was shown to have a certain asymptotic {\it block structure}.
This could be exploited when the loss measure
exhibited a certain {\it separability} property. 
For such separable losses, the optimal shrinker
could be obtained by a simple spike-by-spike analysis.

In this paper, the block structure persists, as we explain in Section 2 below;
however separability is lost. Thus, the general perspective that was used 
successfully 26 times in \cite{donoho2013optimalShrinkage} is not applicable. 
Instead, the optimal shrinker is no longer separable; 
as depicted in Figure \ref{fig:NonlinMultiSpike} the shrinker's
output at a given sub-principal eigenvalue $\lambda_i$, $i > 2$,
changes according to the value of the top eigenvalue $\lambda_1$.
Also, study of properties using spike-by-spike analysis no longer works.
Our proofs of optimality show that the non-separable optimal rule 
$\eta_m^*$ manages a delicate cross-spike compromise \--
typically choosing a shrinker to balance distortions at the top spike and the bottom spike.   So, while we
build on the earlier ideas, several new ideas are needed and developed in this paper.

\subsubsection{Multi-User Covariance Estimation} 
Covariance matrices play an essential 
role in modern portfolio allocation in empirical finance \cite{grinold2000active}.
An extensive statistical literature discusses the benefits of shrunken
covariance estimates in that setting; 
we mention work by Bai and co-authors \cite{bai2009enhancement}, 
El Karoui \cite{el2010high},
Fan and co-authors \cite{fan2011high, fan2013large}, 
Lai, Xing and Chen \cite{lai2011mean},
Ledoit and Wolf \cite{ledoit2014nonlinear}, 
McKay and co-authors: \cite{yang2015minimum, couillet2014large}
and Onatski \cite{onatski2012asymptotics}.  

From the shrinkage literature across 6 decades
and also individual papers like \cite{donoho2013optimalShrinkage}, we know
that shrinkage is often beneficial, but optimal shrinkage is very task-dependent. 
{\it Multi-user Covariance Estimation} (MUCE) posits 
an interesting variation on the traditional portfolio allocation task,
in which optimal $\kappa$-loss shrinkage plays a key role.

Suppose a central authority 
supplies a covariance matrix  
$\hat{\Sigma}$  to many users, 
who  each privately use the supplied matrix $\hat{\Sigma}$ to perform mean-variance
portfolio optimization. The $u$-th user is assumed to have
a vector  $\mu_u$ of forecasted returns, and to allocate one unit of investor capital 
across a portfolio (vector) of holdings $h_u$ by solving the mean-variance portfolio allocation problem \cite{grinold2000active} 
\begin{equation}
\begin{array}{cc}
(\MV(\mu_u,\hat{\Sigma})) \qquad & \mbox{Minimize over }h: h'\hat{\Sigma}h\\
& \mbox{subject to } \qquad  h'\mu_u =1.
\end{array}\label{MRK_Prob}
\end{equation}
The  forecast $\mu_u$ is presumably private to user $u$, and is not known
to the suppliers of the covariance matrix (and presumably each user's forecast is
kept private from other users)\footnote{Problem (\ref{MRK_Prob}) suppresses 
the leverage and short sale constraints which often complicate practical allocation;
we prefer to focus on 
risk-return trade-offs in our discussion.}.

MUCE models a common situation
in empirical finance, where risk measurement enterprises
Axioma, MSCI-Barra, and RiskMetrics,
do, in fact, disseminate risk models (i.e covariance matrices)
to their customers on a periodic basis. 
Many of those customers then privately 
use the supplied covariance matrices in portfolio allocation.
 
MUCE is also rumored to arise in
large trading organizations composed of decentralized 
 {\it forecasting} teams working independently across the 
organization to each develop their own mean forecasts,  where 
 the organization has a central {\it risk-measurement}
 team that supplies a returns covariance matrix
 for all forecasting teams to use in {\it decentralized} allocation. Namely,
 each team in the organization 
mean-variance allocates  its own individual capital endowment
separately from other teams, using the organization's common covariance matrix
and that team's private forecast.

In the MUCE setting, the producer of the covariance matrix estimate does not
know the forecasts that the matrix' users are applying, and in particular cannot make use
of empirical cross-user compromises such as `produce the matrix that historically works best on average across
all existing users'. 

Relative condition number loss optimality, explored in this paper,
 offers instead mathematical guarantees that apply across
{\it all} users.

Suppose that the returns $X_i \sim_{iid} N(\mu,\Sigma)$, $i=1,\dots,n$, and 
let $\SR(\hat{\Sigma};\mu,\Sigma)$ denote the expected out-of-sample 
risk-adjusted return, where the estimated covariance matrix $\hat{\Sigma}$ and (true) 
forecast $\mu$ are supplied to (\ref{MRK_Prob});  and where, in computing the
risk-adjusted return, we evaluate the risk
of the portfolio using the correct risk model $\Sigma$.
Similarly, let $\SR({\Sigma};\mu,\Sigma)$ denote the Sharpe ratio when the 
true underlying covariance matrix is used both in Markowitz allocation
and in evaluation of the portfolio's risk-adjusted return.
We always have
\[
 {\SR(\hat{\Sigma};\mu,\Sigma)} \leq  {\SR(\Sigma;\mu,\Sigma)};
\]
but in the high-dimensional setting where $p \sim \gamma \cdot n$,
the shortfall is typically substantial. 
In words, using $\hat{\Sigma}$ rather than $\Sigma$ for portfolio allocation
causes a substantial shortfall of realized Sharpe ratio \cite{el2010high,ledoit2014nonlinear,lai2011mean,ledoitWolf2004AWell}.

Performance shortfalls are pervasive in real-world empirical finance
and one would like to limit them where possible.
A {\it Relative Sharpe Ratio Guarantee} (RSRG) $\rho$ for $\rho > 1$ is a statement of the form 
\[
  \SR({\Sigma};\mu,\Sigma)/\rho \leq  \SR(\hat{\Sigma};\mu,\Sigma), \qquad \forall \mu.
\]
It  guarantees that the Sharpe ratio one experiences using $\hat{\Sigma}$ never
suffers more than the given factor $\rho$ of deterioration,
no matter which underlying $\mu$ is in force.
In this way the performance shortfall is curtailed.

Let $\RSRG(\hat{\Sigma},\Sigma)$ denote the smallest positive $\rho$ 
which allows such a guarantee; this has an intimate connection with $\kappa$-loss.
\begin{lem} ({\bf RSRG in terms of Condition Number.}) \label{lem:ReduceConditionNumber}
Define the pivot  $\Delta = {\Sigma}^{-\frac{1}{2}}\hat{\Sigma}{\Sigma}^{-\frac{1}{2}}$.
\begin{equation}  \label{eq:KappaIsometry}
\sup_{\mu}  \frac{\SR({\Sigma};\mu, \Sigma)}{\SR(\hat{\Sigma};\mu, \Sigma)}= \frac{1}{2}\sqrt{\kappa(\Delta)+\frac{1}{\kappa(\Delta)}+2} 
\end{equation}
More specifically, let $\Delta(\eta) = {\Sigma}^{-\frac{1}{2}}\hat{\Sigma}(\eta){\Sigma}^{-\frac{1}{2}}$; then
\[
   \RSRG(\hat{\Sigma}(\eta)) =  \frac{1}{2}\sqrt{\kappa(\Delta(\eta))+\frac{1}{\kappa(\Delta(\eta))}+2}.
\]
\end{lem}
The correspondence  $\kappa \mapsto \kappa + 1/\kappa$ 
is one-to-one on $\{ \kappa : \kappa \geq 1 \}$.
So this lemma sets up an isomorphism 
between questions concerning Sharpe ratio guarantees
and questions concerning condition numbers.
Since $\kappa \mapsto \kappa + 1/\kappa$ is increasing in $\kappa$, 
minimizing relative condition-number loss is equivalent to minimizing $\RSRG$.

\subsubsection{Multi-Task Discriminant Analysis} 

Ever since RA Fisher's pioneering work in discriminant analysis \cite{fisher1936use},
the covariance matrix has been an essential component in
designing linear classifiers. Fisher of course showed how to use
the underlying theoretical covariance in LDA, 
but that covariance matrix would not truly be available to us in applications.
The dominant approach over the ensuing eight decades has certainly been to `plug-in'
the naive empirical covariance as if it were the best available approximation to the
theoretical covariance. Relatively little attention has been paid to the
idea that in high-dimensional cases the covariance might 
need to be estimated with application to LDA in mind.

Bickel and Levina \cite{bickel2004some} discuss
some of the difficulties of  traditional empirical covariance matrices
as plug-in inputs to  Fisher's discriminant analysis.
Methodological work aiming to surmount some of these difficulties
includes Jerome Friedman's regularized discriminant analysis (RDA) 
\cite{friedman1989regularized} and several methods developed by
Jianqing Fan and co-authors \cite{fan2008high,fan2012road},
including ROAD. 
Theoretical work documenting the difficulties of plug in rules under
high-dimensional asymptotics includes pioneering work of
Serdobolski \cite{serdobolskii2007multiparametric}.
Recent work on RDA by Dobriban and Wager 
\cite{dobriban2015high} studied regularized discriminant analysis
under Serdobolski-flavored asymptotics.  

This paper's study of
estimation under $\kappa$-loss is directly relevant to
Multi-Task discriminant analysis (MTDA), 
which envisions the use of LDA across many different 
classification tasks, but always with the same
underlying features driving the classification
and always the same covariance estimate defined on those features.

The setting for MTDA is increasingly important: it applies to large 
shared databases which many researchers can mine for different purposes.
There are conceptually two distinct databases: first, a database `X'
in which a sample of individuals has been measured
on $p$ `predictive', or easy-to-measure, feature values. Separately,
other 'outcome' or `hard-to-measure' 
properties about those same individuals are recorded in a
database `Y'. Researchers are interested in 
using the outcome `Y' data to define interesting dichotomies 
of individuals, and then in developing
rules to predict class membership 
in such dichotomies, based on the `X' data. 

In one stylized application of MTDA, 
`X' would be gene expression data 
and `Y' phenotypic data.  
Dichotomies based on `Y' split individuals into classes 
with and without certain phenotypes. The goal of one individual
researcher is to identify an interesting phenotypic split and
then use  gene expression data to predict those phenotypes.

Historically, researchers worked in an uncoordinated way on
data of this kind,  each one developing a
bespoke classifier. Such a practice \-- developing a different
workflow for every such project \-- 
is considered harmful by many serious scientists
\cite{carp2012secret}, as
the resulting analysis variability and method instability adds 
uncertainty to the interpretation of researcher claims. 
For this reason, and also because it's good research
hygiene to have a good baseline procedure, one might consider 
developing an LDA-based procedure appropriate for use across
many different dichotomies.

A common situation in such large shared databases
involves the number $p$ of `X' features being comparable to or larger than
the number $n$ of individuals in the database.
In that setting, high-dimensional asymptotics become relevant, and
we should carefully pay attention to the inaccuracies in the covariance
matrix.

Recall that in traditional single-task linear discriminant analysis (LDA) \cite{friedman2001elements}
the feature measurements $X$ in the two classes
of a dichotomy are assumed to have distributions $N(\xi_c,\Sigma)$, $c=0,1$
where the $\xi_c$ denote the class-conditional means. For
simplicity, we assume that
there are a priori the same number of individuals in each class, 
and define the mean feature value $\xi = (\xi_0+\xi_1)/2$, 
and the mean interclass contrast $\mu = (\xi_1-\xi_0)/2$.
Under traditional LDA,
we attempt to classify a future observed $X$ into one of the two groups based on
linear classifier scores:
\[
       w'(X-\xi)  <> 0.
\]
Fisher showed that, when $\Sigma$ is known,
the ideal weights $w^*$ are the solution to 
\begin{equation}
\begin{array}{cc}
(\LDA(\mu,\Sigma)) \qquad & \mbox{Maximize over }w: \dfrac{w'\mu}{\sqrt{w'\Sigma w}}\\
\end{array}\label{prob:SNR}
\end{equation}
so that $w^* \propto \Sigma^{-1} \mu$.
The misclassification rate of this ideal classifier is $\Phi(-\Sep)$, where $\Phi$ denotes the standard $N(0,1)$ CDF,
and $\Sep$ denotes Fisher's measure of {\it interclass separation}:
\[
\Sep(\Sigma; \mu,\Sigma) \equiv \sqrt{\mu'\Sigma ^{-1}\mu}.
\]
In practice, $\Sigma$ is not available, and $\Sep(\Sigma; \mu,\Sigma)$ represents 
an ideal performance achievable only with the aid of an oracle. Plugging  $\hat{\Sigma}$ into (\ref{prob:SNR}) gives 
the achievable interclass separation 
\[
\Sep(\hat{\Sigma};\mu,{\Sigma}) \equiv \frac{\mu' \hat{\Sigma}^{-1} \mu}{\sqrt{\mu' \hat{\Sigma}^{-1} \Sigma \hat{\Sigma}^{-1} \mu}}.
\]

When the number of individuals $n$ in the database is comparable to the number of `X' feature measurements $p$ per individual,
the concerns of this paper become relevant. 
{We always have $ \Sep(\hat{\Sigma};\mu,{\Sigma}) < \Sep(\Sigma; \mu,\Sigma)$,} but when $p/n \goto \gamma > 0$,
the achievable 
classifier performance can fall well short of ideal oracle performance. 
A {\it relative separation guarantee } is an inequality of the form
\[
 { \Sep(\Sigma;\mu,{\Sigma})}/\rho <  \Sep(\hat{\Sigma}; \mu,\Sigma), \quad \forall \mu .
\] 
Specifically, we are guaranteed that the asy. misclassification rate of any dichotomy will
not exceed $\Phi(- \Sep/\rho)$ where $\Sep$ here denotes the ideal separation $\Sep(\Sigma; \mu,\Sigma)$.
The best performance guarantee relative to  oracle performance involves 
\[
\RSEPG \equiv \sup_{\mu} { \frac{\Sep(\Sigma; \mu,\Sigma) }{\Sep(\hat{\Sigma};\mu,{\Sigma})}}.
\]

There is a formal resemblance of 
Fisher separation $\Sep$ and Sharpe Ratio $\SR$
which implies a formal isomorphism between the MTDA and MUCE settings.
Indeed, Lemma \ref{lem:ReduceConditionNumber} yields that
\[
\rho = \frac{1}{2}\sqrt{\kappa(\Delta)+\frac{1}{\kappa(\Delta)}+2} ,
\]
where, as earlier, $\Delta = {\Sigma}^{-\frac{1}{2}}\hat{\Sigma}(\eta){\Sigma}^{-\frac{1}{2}}$. 
When applied in the MTDA setting,
the optimal shrinker of this paper offers a covariance 
estimate which optimally limits the maximal cross-dichotomy performance shortfall relative to an oracle.

\subsubsection{Implications}

In these stylized applications, the optimal shrinker
behaves quite differently  than today's common practices. Consider first 
the {\it treatment of large eigenvalues}.

\bitem
\item Today, many practical applications of mean-variance optimization and Fisher LDA
undoubtedly involve no eigenvalue shrinkage whatever.
\item Where shrinkage has been applied  in practice \-- for example by Barra \cite{menchero2011eigen}--
the goal has been
to ensure that the variance in top eigendirections is accurately estimated.
Such {\it debiasing rules} shrink large eigenvalues relatively little.
In contrast this paper's  shrinker 
discounts the top eigenvalues quite severely, by a factor  $1/(1+\gamma)$.
\item Shrinkage of the top eigenvalue by nonlinearity $\eta^*$ combats
a marked deficiency of  the $(\MV(\mu,\hat{\Sigma})$ and $(\LDA(\mu,\hat{\Sigma}))$
optimization problems. Those optimization procedures are not aware
of the types of inaccuracy suffered by the covariance estimate $\hat{\Sigma}$  
in the high-dimensional case. They `clamp down' severely on exposures to the top 
empirical eigenvector. However, when the top eigenvector is not perfectly accurate,
there is a limited risk benefit from such clamping down.
 $\eta^*$ is aware of the inaccuracy of eigenvectors and guards 
 against clamping down too severely.
\eitem
Consider next the {\it treatment of small or moderate eigenvalues}. 
\bitem
\item It is common  in applied statistics to decide the rank of a factor model
by null hypothesis significance testing. This leads to accepting 
eigencomponents as part of the covariance model whenever the
corresponding empirical eigenvalue exceeds the bulk edge noticeably. The optimal
shrinker derived here sets a noticeably higher standard for eigencomponent inclusion, 
i.e a threshold well outside the bulk, thereby
demanding much more than mere `statistical significance'.
\item In particular, there are spike configurations 
where the empirical eigenvalues are unquestionably 
non-null, yet the optimally shrunken covariance model is simply the identity matrix.
In such situations, no-one would dispute the existence of correlations between variables, 
however, it could be understood that the corresponding eigencomponents are too inaccurate to
be of value.  

For an empirical finance scenario, suppose that we have $p=3000$ stocks (as in the Russell 3000)
and $n=252$ observations (1 year of daily returns). Then $\gamma=11.9$, and
 $\lambda_+(11.9) \approx { 19.8}$ while $\lambda_m^+({11.9}) \approx { 27.42}$.
The threshold for inclusion in the model by the optimal shrinker 
 is asymptotically {\sl {38\%} larger than the
usual  threshold based on null hypothesis significance testing.}
\eitem
Finally, consider the worst-case situation for MUCE and MTDA. Practitioners often suspect that
theoretical guarantees are misleading, because they may guard against unrealistically
pessimistic situations. And yet:
\bitem
\item  In Section \ref{ssec:LFUF} below, we show that the least-favorable forecasts (in MUCE) /
least-favorable dichotomies in (MTDA) involve 
$\mu$ a linear combination of the top empirical eigenvector and the top
theoretical eigenvector. 
\eitem
In empirical finance, these two directions translate into the
past market direction and the future market direction, respectively. 
Undoubtedly, many investors face heavy exposures of just this type.
In genomic analysis, those two directions concern the directions of strongest
underlying genetic relatedness in the population and in the sample, respectively. Undoubtedly,
some of the most provocative dichotomies are well approximated using just these two directions.

We briefly mention implications specific to Multi-Task Discriminant Analysis.
Because of R.A. Fisher's prestige, 
it might today be considered 
{\it de rigueur} to use the standard empirical covariance matrix in 
developing an empirical linear classifier 
({\it pace} Bickel and Levina \cite{bickel2004some}).
Nevertheless,  practitioners have long successfully 
used naive diagonal rules in linear classification and 
they have sometimes taken pains to document the fact that such
naive rules outperform covariance-aware rules; see
references in Hand and Yu \cite{hand2001idiot},
the empirical studies of Zhao and co-authors on 
the very simple Mas-O-Menos classifier \cite{zhao2014menos}, 
and of Donoho and Jin on the closely related
Higher-Criticism feature selector/classifier \cite{donoho2008higher}.

The theory developed here aligns with earlier research 
-- both empirical and theoretical -- affirming the use of
`naive' diagonal covariance estimates in LDA, but going much further
in showing that such naive rules can actually be optimal \-- in the sense discussed here.
For well-cited theoretical work considering diagonal covariances
in place of standard covariance estimates when $\gamma$ is large, 
see for example \cite{bickel2004some,dudoit2002comparison,fan2008high}. 
Under our spiked model assumption, where $\gamma = p/n$ is large 
and the top eigenvalue is not very large, the optimal equivariant shrinker 
literally reduces to the naive identity rule.  
In short, the {\it identity   covariance estimate
offers optimal $\kappa$-loss guarantees in a range of cases where there 
are relatively large and statistically significant eigencomponents.}

An example given earlier
considered $\gamma = 50$ as 
being very plausible for genomic data analysis.
In that setting, unless
 $\lambda_1$ exceeds 100, {\it the best achievable performance 
guarantee comes by using the identity covariance matrix.}

The theory developed here also provides an additional setting 
-- to add to the previously known ones \-- where
singular values and eigenvalues can sometimes profitably be ignored even when they
are noticeably outside the bulk; compare \cite{gavish2014optimal,donoho2013optimalShrinkage}.

\section{Basic Tools\label{sec:Basic Tools}}

\subsection{Spiked Model Asymptotics}
\label{ssec:SpikeModelAsymp}

We remind the reader of two central assumptions:
\begin{description}
\item[{[PGA]}] {\it Proportional-Growth Asymptotics.}
We consider a sequence of problem sizes $n,p\rightarrow\infty$ 
while $\frac{p}{n}\rightarrow\gamma>0$.
\item[{[SPIKE]}] {\it Spiked Covariance Model.}
We assume that the population covariance $\Sigma$ is spiked,
with $r$ fixed spikes and base variance $1$.
In other words, the eigenvalues of $\Sigma$ are $l_{1} > l_{2} >...> l_{r}> 1 = l_{r+1}=...=l_{p}$.
\end{description}

In this setting, the empirical eigenvalues $\lambda_{i,n}$  of $\bS = \frac{1}{n} X'X$
are random, but have a very simple asymptotic description. 

\begin{description}
\item [Inside the ``bulk''\,] All but at most $r$
lie inside a bulk distribution extending through the interval 
$[\lambda_-(\gamma), \lambda_+(\gamma)]$ where
\[
   \lambda_\pm(\gamma) = (1 \pm \sqrt{\gamma})^2.
\]
\item [Outside the ``bulk'' \,] 
Each theoretical eigenvalue $\ell_i$ among the first $r$
which exceeds $\ell_+(\gamma) = 1 + \sqrt{\gamma}$,
generates a corresponding empirical eigenvalue $\lambda_{i,n}$ of $\bS$
outside the bulk $[\lambda_-,\lambda_+]$, and converges to a limiting position $\lambda_i > \lambda_+$ (see (\ref{eq:eigendisp}) below)
as $n \goto \infty$. Asymptotically, at most $r$ eigenvalues lie outside the vicinity of the
bulk.
\end{description}

Moreover, in the spiked setting, the following fundamental result
tells us a great deal about the
behavior of the eigenvalues outside the bulk and their eigenvectors.
Variants and extensions have been developed by 
Baik, Ben Arous, and P\'{e}ch\'{e} \cite{baik2005phase},
Baik and Silverstein \cite{baik2006eigenvalues},
Paul \cite{paul2007asymptotics}, Nadler \cite{nadler2008finite}, Benaych-Georges and Rao-Nadakuditi
\cite{Benaych-Georges2011}.

\begin{thm} \label{lem:spiked} {\bf (Spiked Covariance Asymptotics, \cite{baik2005phase,baik2006eigenvalues,paul2007asymptotics,nadler2008finite, Benaych-Georges2011})}
In the spiked model [SPIKE] under the proportional growth asymptotic [PGA],
we have  {\bf eigenvalue displacement}, such that for each 
spike eigenvalue $\ell_i > \ell_+(\gamma)$, a corresponding empirical eigenvalue
obeys:
\[
   \lambda_{i,n} =  \lambda(\ell_i; \gamma) \cdot (1 + o_P(1)), \qquad n \goto \infty,
\]
where
\beq \label{eq:eigendisp}
 \lambda(\ell; \gamma)  \equiv  \left \{ \begin{array} {ll} 
      \ell \cdot ( 1 + \frac{\gamma}{\ell-1}) & \ell > \ell_+(\gamma)\\
      \lambda_+(\gamma) & \ell \leq \ell_+(\gamma) \\
      \end{array} \right . .
\eeq
Suppose that the spike values are distinct.
We also have {\bf  eigenvector rotation}, such that the theoretical eigenvector $U_i$ and
the corresponding empirical eigenvector $V_i$ obey:
\[
|\langle U_{i},V_{i}\rangle|\xrightarrow{a.s}c(\ell_{i}; \gamma) , \qquad n \goto \infty,
\]
where 
\[
c(\ell ; \gamma ) \equiv \left \{ 
\begin{array} {ll}  \sqrt{\frac{1-\frac{\gamma}{(\ell-1)^{2}}}{1+\frac{\gamma}{\ell-1}}} & \ell > \ell_+(\gamma) \\
      0  & \ell \leq \ell_+(\gamma) \\
      \end{array} \right . .
\]
In addition, for $\ell_i, \ell_j > 1 + \sqrt{\gamma}$, $\ell_i \not = \ell_j$ we have
\[
|\langle U_{i},V_{j}\rangle|\xrightarrow{a.s} 0  \mbox{ as } \qquad n \goto \infty.
\]
\end{thm}

Earlier, we pointed out that at most the first $r$ eigenvalues emerge from the bulk. 
The following lemma formalizes this notion.
\begin{lem} \label{lem:max_bulk} \cite{baik2006eigenvalues,paul2007asymptotics}
Let $\lambda_{i,n}$ be the (decreasingly arranged \-- $\lambda_{1,n} \geq \lambda_{2,n} \geq \dots \geq \lambda_{p_n}$) empirical eigenvalues of the empirical covariance matrix $\bS_{n,p_n}$.  Under the assumptions [SPIKE] and [PGA], suppose there are $0 \leq k \leq r$ spikes $\ell_i$ exceeding $\ell_+(\gamma)$. Then
\[
\lambda_{k+1,n} \xrightarrow{a.s} \lambda_+(\gamma), \quad n \goto \infty.
\]
\end{lem}

All results of this paper can be viewed as consequences of
Theorem \ref{lem:spiked} and Lemma \ref{lem:max_bulk}.
We believe the {\it conclusions} of Theorem \ref{lem:spiked} and Lemma \ref{lem:max_bulk} 
extend to a wider range of assumptions,  
where the bulk distribution
of empirical eigenvalues still follows the Mar\v{c}enko-Pastur law.
See also the footnote in Section 1.

\subsection{Two Parametrizations of Eigenvalues}
\label{sec:twoparam}

Theorem \ref{lem:spiked} gives us 
two ways to describe the top eigenvalues.
We can equivalently describe either:
\bitem
\item Underlying theoretical eigenvalues $\ell_i$ associated to $\Sigma$; these are not directly observable.
\item Observable empirical eigenvalues  
$\lambda_{i,n}$ associated to empirical covariance matrices $\bS_{n,p_n}$,
with large-$n$ limits $\lambda_i$. The $(\lambda_i)$ 
correspond to, but are
displaced from, their theoretical counterparts $(\ell_i)$.  
\eitem
As long as we are speaking about a theoretical spike eigenvalue 
$\ell > \ell_+(\gamma) = (1+\sqrt{\gamma})$, 
the formula for the limit empirical eigenvalue
$\lambda = \lambda(\ell;\gamma) = \ell \cdot (1+\frac{\gamma}{\ell-1})$ applies
and gives a one-one-correspondence between the spike $\ell$ and limit empirical $\lambda(\ell)$.  
We can invert the relation $\ell \mapsto \lambda(\ell)$ to obtain:
\beq \label{eq:ellOfLambda}
  \ell(\lambda; \gamma) = \frac{\lambda+1-\gamma + \sqrt{(\lambda + 1 - \gamma)^2 - 4 \lambda}}{2},
\eeq
provided $\lambda$ lies above the bulk edge $\lambda_+(\gamma) = (1 + \sqrt{\gamma})^2$.

The one-one correspondence establishes a kind of interchangeability between
writing out key expressions in terms of spike parameters $\ell$ or in terms of limiting empirical
eigenvalues $\lambda$. It turns out to be very convenient to have understood conventions
whereby we can without comment write expressions in terms of $\ell$, in terms of $\lambda$,
or even in terms of both. Accordingly, we adopt three conventions.

\vspace*{2em}
\noindent
{\bf Convention 1.}
{\sl Provided $\ell > \ell_+$ \-- or equivalently $\lambda(\ell) > \lambda_+$ \--  we are permitted to write
expressions
\[
     F(\lambda) = G(\ell)
\]
and the meaning will be  unambiguous. If we are in a pedantic mood,
we can rewrite such expressions entirely in terms of $\lambda$ 
(hence as  $F(\lambda) = G(\ell(\lambda))$) or entirely in terms of $\ell$ (hence as $F(\lambda(\ell)) = G(\ell)$).
We  adopt without comment the habit of writing 
expressions in terms of a mixture of $\ell$'s and $\lambda$'s
when it gives simpler or more memorable formulas.}

\vspace*{1em}
\noindent
{\bf Convention 2.}
{\sl Provided $\ell > \ell_+$ \-- or equivalently $\lambda(\ell) > \lambda_+$ \--  there are defined quantities
$c$ and $s$ that may be written in terms either of $\ell$ or $\lambda$ and we may write 
expressions
\[
     F(\ell,\lambda,c,s)
\]
and the meaning will be  unambiguous. In a pedantic mood we could write everything out in terms of $\ell$ or in terms of $\lambda$.}

\vspace*{1em}
\noindent
{\bf Convention 3.}
{\sl In several cases below it will be convenient to
consider some  function $F$ as {\em in certain passages}  a function of $\ell$ and
{\em other passages} a function of $\lambda$. In such cases what we really mean is that 
we are assuming $\ell > \ell_+$ and that there is a function $F_\lambda(\cdot)$ and a function $F_\ell()$,
linked by $F_\lambda(\lambda(\ell)) = F_\ell(\ell)$ and $F_\ell(\ell(\lambda)) = F_\lambda(\lambda)$ over the relevant ranges
of both $\ell$ and $\lambda$,
and that when we write $F(\lambda)$, we really mean $F_\lambda(\lambda)$, while 
when we write $F(\ell)$, we mean $F_\ell(\ell)$. We believe there is little risk of confusion in these instances.
\footnote{Note: similar conventions are followed in object-oriented programming languages; different formulas
apply depending on the `type' of the argument. Here the type is `empirical eigenvalue (or its limit)' or `spike value'.}
}

\subsection{The Empirical Pivot and Asymptotic Pivot}

Let the empirical covariance matrix $S$ have spectral decomposition
$S = V \cdot \Lambda \cdot V'$. Let $\eta \in \R_+^{p}$ 
be some nonnegative sequence $(\eta_i)$; 
then $\hat{\Sigma}_n(\eta) = V \cdot diag(\eta) \cdot V'$ is a matrix 
where we retain the empirical eigenvectors but
replace the eigenvalues  $(\lambda_i)$ by $(\eta_i)$.

We consider a matrix-valued  discrepancy $\Delta$ 
between such an estimator $\hat{\Sigma}_n (\eta)$ and 
the underlying covariance $\Sigma$ which we call the {\it empirical pivot}.
Specifically, $\Delta \equiv \Delta({\Sigma},\hat{\Sigma}_n) \equiv\Sigma^{-1/2}  \hat{\Sigma}_n(\eta)\Sigma^{-1/2}$.

We always represent the empirical pivot  
in the so-called $W$-basis:
\[ 
\Delta_n = W'\Sigma^{-1/2}  \hat{\Sigma}_n(\eta)\Sigma^{-1/2}W.
\]
As in \cite{donoho2013optimalShrinkage},
the columns of $W$ are  basis vectors obtained, for $n,p > 2r$, by 
applying Gram-Schmidt orthogonalization
to the vectors 
$\{U_{1},V_{1,n},...,U_{r},V_{r,n}\}$, 
producing a sequence of orthonormal vectors $(W_i)_{1,\dots,2r}$,
which we then complete out to an orthogonal basis of $\R^p$. 
From now on, we always assume that we have transformed coordinates to this $W$-basis
and we simply write $\Delta_n =\Sigma^{-1/2}  \hat{\Sigma}_n(\eta)\Sigma^{-1/2}$.

\begin{lem} \label{lem:blocks} {\bf \cite{donoho2013optimalShrinkage}}
({\bf Convergence to the Asymptotic Pivot}).
Suppose that $\eta \in \R^{p}$ has the special form
\beq \label{eq:padbyone}
(\eta_i) = (\eta_1,\eta_2,\dots,\eta_r,1,\dots,1),\;\eta_i \geq 1.
\eeq
where $\eta_1$,\dots $\eta_r$ are fixed independently of $n,p > r$.
Let $\Delta^a( (\ell_i), (\eta_i)) = \Delta^a( (\ell_i), (\eta_i); \gamma)$ denote the
deterministic
block-diagonal $p \times p$ matrix  
\[
\Delta^a = \left[\begin{array}{ccccc}
A(\ell_{1},\eta_1)\\
 & A(\ell_{2},\eta_2)\\
 &  & A(\ell_{3},\eta_3)\\
 &  &  & ...\\
 &  &  &  & A(\ell_{r},\eta_{r})
\end{array}\right]\varoplus I_{p-2r} ,
\]
where the 2-by-2 matrices $A(\ell,\eta) \equiv A(\ell,\eta;\gamma)$ are defined via:
\[
A(\ell,\eta):=\left[\begin{array}{cc}
\frac{\eta c^{2}+s^{2}}{\ell} & \frac{(\eta-1)cs}{\sqrt{\ell}}\\
\frac{(\eta-1)cs}{\sqrt{\ell}} & c^{2}+\eta s^{2}
\end{array}\right] ,
\]
with $c=c(\ell;\gamma)$ and $s =\sqrt{1-c^{2}}$. We have almost surely and in probability
the convergence in Frobenius norm:
\begin{equation} \label{eq:FroConv}
   \| \Delta_n - \Delta^a( (\ell_i) , (\eta_i)) \|_F \goto 0,
\end{equation}
as $n \to \infty$.
\end{lem}

Below, we will call $\Delta^a$ the {\it asymptotic pivot}; it 
is actually a sequence of pivots, one for each pair $(n,p_n)$  visited along the
way to our proportional growth limit.
Convergence in Frobenius norm is not happening in one common space,
but instead, the Frobenius norm of the pivot difference 
is tending to zero along this sequence.

The following trivial but still useful corollary follows from
the continuity of $\eta \mapsto A(\ell,\eta)$ 
and the triangle inequality for the Frobenius norm.
\begin{cor}
\label{cor:AsyPivWPerturbations}
Consider a sequence of vectors $(\eta_{i,n})_{i=1}^r$ 
such that $\eta_{i,n} \goto \eta_i$, $n \goto \infty$, $i=1,\dots,r$.
Then for 
\[
\Delta_n((\eta_{i,n})) = W'\Sigma^{-1/2}  \hat{\Sigma}_n((\eta_{i,n}))\Sigma^{-1/2}W,
\]
we have both
\begin{equation} \label{eq:FroConvA}
   \| \Delta_n - \Delta^a( (\ell_i) , (\eta_i)) \|_F \goto 0,
\end{equation}
and
\begin{equation} \label{eq:FroConvB}
   \| \Delta_n - \Delta^a( (\ell_i) , (\eta_{i,n})) \|_F \goto 0,
\end{equation}
\end{cor}


\subsection{Pivot Optimization}
\label{sec:StudyAsyPiv}

Consider a spike configuration $(\ell_i)_{i=1}^r$  
fixed independently of $n, p$, $p > 2r$. 
Let $\EtaRNPa$ denote the collection of vectors $(\eta_{i})_{i=1}^{p}$
obeying $Ave_{i > r} \eta_i =1$.
The condition $\eta \in \EtaRNPa$ is a normalization condition
which is also obeyed by the underlying 
theoretical eigenvalues $(\ell_i)_{i=1}^{p} \in \EtaRNPa$.
Denote the empirical pivot by
\[
\DelEmp((\ell_i); (\eta_i)_{i=1}^p; n,p) =  \Sigma^{-1/2} \hat{\Sigma}(\eta) \Sigma^{-1/2},
\]
where $\hat{\Sigma}(\eta) = V diag(\eta_{i}) V'$.
Consider the {\bf empirical pivot optimization problem}
\[
   (\KEmp((\ell_i);n,p)) \qquad  \min_{\eta \in \EtaRNPa} \kappa(\DelEmp((\ell_i);(\eta_i); n , p)).
\]
The value of this finite-$n$  problem is random (because the $(V_{i})$ are random).

Lemma \ref{lem:blocks} suggests that this seemingly difficult problem can be
replaced by the seemingly easier problem of determining $r$ scalars 
$(\eta_i)_{i=1}^r$ that optimize the asymptotic pivot matrix.}
This section justifies such replacement.

\begin{lem} ({\bf Asymptotic Pivot Optimization}). \label{lem:AsyOptProb}
Consider a fixed configuration $(\ell_i)_{i=1}^r$ and a 
formal variable $\eta = (\eta_i)$ with $\eta_i$ free to vary in $(0,\infty)$ for $1=1,\dots,r$
and padded by ones $\eta_i =1$ for $i > r$ as in (\ref{eq:padbyone}). Define 
\[
\DelAsy((\ell_i); (\eta_i); n,p,\gamma) = \oplus_{i=1}^r A(\ell_i,\eta_{i}; \gamma)  \oplus I_{p-2r} .
\]
Consider the asymptotic pivot optimization problem
\beq \label{eq:asymp-pivot}
   (\KAsy((\ell_i);n,p_n,\gamma)) \qquad  \min_{(\eta_i)_{i=1}^r} \kappa(\DelAsy((\ell_i);(\eta_i); n , p_n, \gamma)) .
\eeq
\begin{enumerate}
\item Setting $\gamma$ fixed independently of $n$,
this problem  has the same optimal value for every $p > 2r$.  
There exists a configuration $(\eta_i^*)_{i=1}^r$ which, 
when padded by ones as in (\ref{eq:padbyone}), 
can achieve this optimal value, simultaneously for every $p > 2r$.  
Denote this configuration by  
$\eta^* = (\eta_i^*((\ell_i);\gamma))$. 

\item The configuration $\eta^*$ 
can be taken to satisfy the
constraint $\eta_i^* \geq 1$, $i=1,\dots,r$. Thus 
 we may equivalently write (\ref{eq:asymp-pivot})
as a constrained optimization problem:
\beq \label{eq:asymp-pivot-lbndone}
   (\KAsy((\ell_i);n,p_n,\gamma)) \qquad  \min_{\substack{(\eta_i)_{i=1}^r \\ \eta_i \geq 1}} \kappa(\DelAsy((\ell_i);(\eta_i); n , p_n, \gamma)) .
\eeq
\end{enumerate}
\end{lem}

\begin{thm} \label{lem:AsyShrinkOpt} ({\bf Asymptotics of empirical pivot optimization}).
The value of the empirical pivot optimization problem 
tends both almost surely and in probability to the
same value as the asymptotic pivot optimization problem:
\[
  val({K}^e((\ell_i);n,p_n)) \goto   val(\KAsy((\ell_i);n,p_n,\gamma)), \qquad n \goto \infty.
\]
Finally,  let
$ \eta^* = (\eta_i^*((\ell_i);\gamma))_{i=1}^r$  
denote an optimizing configuration 
of the asymptotic pivot optimization 
problem $\KAsy((\ell_i);n,p_n,\gamma)$. 
Padding this optimum configuration with $1$'s
asymptotically almost surely achieves the optimum of the 
empirical pivot 
\[
|\kappa(\DelEmp((\ell_i),(\eta^*_1,\dots \eta_r^*,1,1,\dots,1); n ,p_n))  -  val(\KEmp((\ell_i);n,p_n)) | \goto 0, \qquad n \goto \infty.
\]

\end{thm}

Summarizing: solving the
 asymptotic pivot optimization problem gives us a
 shrinker which approximately solves each associated large, finite-$n$
 empirical pivot optimization problem.

Of course, asymptotically small perturbations of $\eta^*$
are also asymptotically optimal. Arguing as for Corollary \ref{cor:AsyPivWPerturbations} gives
\begin{cor}
\label{cor:OptimalityWTuning}
Consider a sequence of vectors $(\eta_{i,n})_{i=1}^r$ 
such that $\eta_{i,n} \goto \eta_i^*$, $n \goto \infty$, $i=1,\dots,r$.
Then as $ n \goto \infty$, 
\[
|\kappa(\DelEmp((\ell_i),(\eta_{1,n},\dots \eta_{r,n},1,1,\dots,1); n ,p_n))  -  val(\KEmp((\ell_i);n,p_n)) | \goto 0 .
\]
\end{cor}

\subsection{Optimality over all Orthogonally Invariant Procedures}
\label{sec:OrthoInvar}

{  The pivot optimization explicitly optimizes over 
diagonals in the spectral decomposition 
$\hat{\Sigma} = V diag((\eta_i)) V'$
of the estimator. Actually 
this optimization covers {\it all} orthogonally-equivariant
procedures.

\begin{lem}
{\bf Diagonal Representation of Orthogonal Equivariance:} 
\label{lem:diagrep}
Let $\semidefp$ denote the collection of symmetric
nonnegative semidefinite matrices.
{\sl Any OE procedure $\hat{\Sigma}(S)$ is of the form 
\begin{equation} \label{eq:OE_characterization}
\hat{\Sigma}(S) = V diag(\eta_1,\dots, \eta_p)V'
\end{equation}
where $S=V\Lambda V'$ is a spectral decomposition of $S$ and
\begin{enumerate}
\item for $i \leq m = \min(n,p)$, $\eta_i : {SD(p)} \mapsto \bR^+$ is a nonnegative function of $S$.
\item {for some $\eta_0: SD(p) \mapsto \bR^+$ and all $i > m$, $\eta_i = \eta_0$.}
\end{enumerate}
}
\end{lem}

The diagonal representation shows that if, 
for a given realization $S = V \Lambda V'$, 
we optimize across all possible diagonals in 
$\hat{\Sigma}(\eta) = V \cdot diag(\eta) \cdot V'$ with variable $\eta$, 
we have thereby optimized over the {\it range of matrices possible
from orthogonally-equivariant estimators}. 


Empirical pivot optimization performs 
such optimization, however
under a normalization condition, which normalizes the estimated 
bulk eigenvalues to have the same average as the bulk theoretical eigenvalues.
The value of the normalized optimization problem
is actually the same as that of the un-normalized problem. 
Indeed, the objective function is scale invariant; namely, for $b > 0$
\[
\kappa(\Delta^e( (\ell_i), ( \eta_i  )_{i=1}^{p_n} ; n , p)) =
\kappa(\Delta^e( (\ell_i), (  b \cdot \eta_i  )_{i=1}^{p} ; n , p)).
\]
It follows that
\[
\kappa(\Delta^e( (\ell_i), ( \eta_i  ))) =
\kappa(\Delta^e( (\ell_i), (   \eta_i / (Ave_{i>r} \eta_i) ))).
\]
The constrained optimization problem is thus effectively unconstrained.  
We prefer the normalized problem because
the asymptotic analysis identifying its limit behavior
is then more transparent and better connected to the 
Asymptotic Pivot Optimization problem.

\subsection{Condition Number of the Asymptotic Pivot} \label{ssec:condNumPivot}

We turn our attention to evaluating and optimizing  the asymptotic pivot,
which we henceforth typically denote simply by
$\Delta((\ell_i); (\eta_i))$ \-- the value of $\gamma$ being suppressed in notation.

Using the block diagonal representation of Lemma \ref{lem:blocks},
the argument for Lemma \ref{lem:AsyOptProb} shows that if $\eta_i \geq 1$,
the asymptotic pivot has condition number
\beq  \label{eq:condnoformula}
 \kappa(\Delta((\ell_i), (\eta_i))) = \frac{\max_i \nu_+(A(\ell_i,\eta_i))}{\min_i \nu_-(A(\ell_i,\eta_i))} ,
\eeq
where $\nu_\pm(A)$ denote the maximum and minimum characteristic values of
the two-by-two matrix  $A$, and the maximum and minimum range over $1 \leq i \leq {r+1}$,
where we formally set $\ell_{r+1}=1$ and $\eta_{r+1}=1$.

\begin{lem} \label{lem:defTD} {\bf \cite{donoho2013optimalShrinkage}}
For the 2-by-2 matrix $A(\ell,\eta),$ 
we have
\[
  D \equiv D(\ell,\eta) =  det(A(\ell,\eta)) = \frac{\eta}{\ell},
\]
\[
  T \equiv T(\ell,\eta) =  tr(A(\ell,\eta)) = (\frac{1 + \dot{\eta} c^2}{\ell}+1+ \dot{\eta} s^2),
\]
where $\dot{\eta} \equiv \eta-1$. Consequently:
\[
   \nu_\pm(A(\ell,\eta)) = T/2  \pm \sqrt{T^2/4-D}.
\]
\end{lem}

For later use, we note that $D(\ell,1) =1/\ell$ and $T(\ell,1) = \frac{1}{\ell} + 1$. Therefore,
\[
  \nu_+( A(\ell,1)) = 1 \qquad \nu_-(A(\ell,1)) = \ell^{-1}.
\]

\section{The Optimal Shrinker in the Single-Spike Case}
\label{sec:OptShrinkSingleSpike}
We initially consider single-spike configurations $r=1$, and derive
an optimal shrinkage nonlinearity,  $\lambda \mapsto \eta_1^*(\lambda)$. 

\subsection{The Closed-Form Expression}

The shrinker $\eta_1^*(\lambda)$ optimize the
asymptotic pivot will equivalently optimize, 
for $\ell = \ell(\lambda;\gamma) > \ell_+(\gamma)$,
this specialization of (\ref{eq:condnoformula}):
\begin{equation} \label{def:Kone}
 K_1(\eta) \equiv  \min_\eta \frac{ \max\left[ 1, \nu_+(A(\ell,\eta))\right] }{\min \left[  1,  \nu_-(A(\ell,\eta)) \right]} .
\end{equation}
\begin{thm} \label{thm:spike_1} 
In the single-spike case, $r=1$, 
the $K_1$-optimal shrinker has the following form
\begin{equation}
\label{eq:optSingleSpikeShrink}
\eta_1^*(\lambda) = \left\{ \begin{array}{ll}  \frac{\ell}{1 + \gamma + \frac{2\gamma}{\ell-1}} & \ell > \ell_1^+(\gamma) \\
                                                1 & \ell \leq \ell_1^+(\gamma) ,
                                                \end{array}   \right.
\end{equation}
where $\ell_1^+(\gamma) =  1+(\gamma + \sqrt{\gamma^2 + 8\gamma})/2$. 
In fact, $\ell_1^+(\gamma) > \ell_+(\gamma)$,
and so $\ell = \ell(\lambda;\gamma)$ is also expressible as a function of $\lambda$ 
throughout the domain $\ell > \ell_1^+(\gamma)$. 
\end{thm}

\subsection{Properties of \texorpdfstring{$\eta_1^*$}{optimal single spike shrinker}}
Let $\eta_+(\ell)$ denote the formula
occurring in the upper branch 
of the case statement defining $\eta_1^*$ in (\ref{eq:optSingleSpikeShrink}).
Then $\eta_1^*(\lambda(\ell)) = \eta_+(\ell)$
throughout the interval $\ell > \ell_1^+$. Note that $\ell_1^+(\gamma)$ is precisely the value of $\ell$ at which  $\eta_+(\ell) = 1$ and that $\ell > \ell_1^+(\gamma)$ gives
the interval of those $\ell$ where $\eta_+(\ell)  > 1$. The branching by cases
in (\ref{eq:optSingleSpikeShrink})  serves to keep $\eta_1^* \geq 1$, which is appropriate
in the spiked model, as every model eigenvalue 
is at least $1$, and  it therefore makes no sense to shrink to  a value less $1$.

Corresponding to $\ell_1^+$ on the $\ell$-scale, we have a threshold 
$\lambda_{\eta_1^*}^+(\gamma) = \lambda(\ell_1^+(\gamma); \gamma)$ on the $\lambda$-scale.
 This threshold exceeds the upper bulk edge $\lambda_+(\gamma)$. We have
\[
\eta_1^*(\lambda)  \searrow 1 \qquad \lambda \searrow \lambda_{\eta_1^*}^+,
\]
while
\[
   \eta_1^*(\lambda) = 1 ,\qquad \lambda \leq \lambda_{\eta_1^*}^+(\gamma).
\]
So the shrinker $\eta_1^*$ behaves like  a form of
`soft' thresholding for smaller values of $\lambda$, with dead-zone 
 $\{ \lambda \leq \lambda_{\eta_1}^+(\gamma) \}$, throughout which all
 eigenvalues are collapsed to $1$. Such collapse is a feature we will see
 again below.
\begin{defn}
The scalar nonlinearity $\eta()$ is said to {\bf collapse the vicinity of the bulk to 1}
provided that for some $\eps > 0$ and $i=1,\dots,min(n,p_n)$,
\[
    \lambda_{i,n} < \lambda_+(\gamma) + \eps  \mbox{ implies } \eta(\lambda_{i,n}) = 1.
\]
\end{defn}

Using this terminology, we say that the 
optimal single-spike shrinker $\eta_1^*$ collapses 
the vicinity of the bulk to $1$.

\begin{defn} \label{def:kappa1star}
{\bf Optimal $\kappa$-loss.}
$\kappa_1^{*}(\ell) \equiv \kappa_1^{*}(\ell;\gamma)$ denotes the
optimal condition number achieved by $\eta_1^{*}(\ell)$. In detail, for $ \ell \geq 1$ set
\begin{eqnarray}
\kappa_1^{*}(\ell) &\equiv& \kappa(A(\ell,\eta_1^*(\ell))). \nonumber 
\end{eqnarray}
\end{defn}

Formulas derived in Lemmas \ref{lem:increasingwrtELL} and \ref{lem:decreasingRwrtELL}
lead to a closed-form expression for $\kappa_1^*$:
\begin{cor}
\begin{eqnarray}
\kappa_1^{*}(\ell) &=&  \frac{1+ \sqrt{\delta^*(\ell)}}{1-\sqrt{\delta^*(\ell)}}, \label{eq:kappa1defn}
\end{eqnarray}
where
\[
\delta^*(\ell;\gamma) = \left\{ \begin{array}{ll}
 \frac{\gamma((\ell-1)^2-\gamma)}{(\ell-1) ((1 + \gamma) (\ell-1) + 2 \gamma)} & \ell > \ell_1^+(\gamma) \\
\frac{(1 - 1/{\ell})^2}{(1 + 1/{\ell})^2} &  1 \leq \ell < \ell_1^+(\gamma)
\end{array} \right . .
\]
\end{cor}

{For very large values $\lambda \gg 1$,
  $\eta_1^*(\lambda)$ resembles
linear shrinkage, with asymptotic slope $\frac{1}{1+\gamma}$.
This resemblance yields the following:}

\begin{cor} ({\bf Large-$\ell$ asymptotics.}) \label{cor:kappaLargeEll} As $\ell \goto \infty$, 
\begin{eqnarray}
\eta_1^{*}(\lambda) &\sim& \frac{\ell}{1+\gamma},  \nonumber \\
\kappa_1^{*}(\ell) &\goto& \frac{1+\sqrt{\frac{\gamma}{\gamma+1}}}{1-\sqrt{\frac{\gamma}{\gamma+1}}}. \nonumber
\end{eqnarray}

\label{cor:asymp_sol}
\end{cor}
Corollary \ref{cor:asymp_sol} documents the approximate linearity of the optimal shrinker
 at large eigenvalues.
For large underlying spike eigenvalues, the optimal loss
becomes a function only of the aspect ratio $\gamma$.

To prove Theorem \ref{thm:spike_1} we use a formula
 equivalent to (\ref{eq:optSingleSpikeShrink}) that
 provides different insights in later sections.

\begin{lem} \label{lem:spikeFormulaEquivalent}
Throughout the interval
$\ell > \ell_1^+(\gamma)$,
\begin{equation}
\label{eq:altSingleSpikeShrink}
\eta_1^{*}(\lambda) =\frac{\ell c^{2}+s^{2}}{\ell s^{2}+c^{2}} .
\end{equation}
\end{lem}
Of course, in (\ref{eq:altSingleSpikeShrink}) 
$c^2 = c^2(\ell(\lambda;\gamma);\gamma)$, $s^2=1-c^2$,
and $\ell$ can all be expressed in terms of $\lambda$
via $\ell = \ell(\lambda; \gamma)$.
Lemma \ref{lem:spikeFormulaEquivalent} is proven in the Appendix, 
alongside the proof of Theorem \ref{thm:spike_1}.

\subsection{Rigorous Optimality}

The phrase `$\kappa$-loss-optimal' in Theorem \ref{thm:spike_1}
refers to the asymptotic pivot optimization problem defined in equation (\ref{eq:asymp-pivot}),
in which we optimize over $r$ variables $(\eta_i)_{i=1}^r$ (at the moment, with $r=1$). Our optimality result, means that the fixed function $\eta_1^*$, evaluated at $\lambda_1 = \lambda(\ell_1)$ gives the optimum value.
 Lemma \ref{lem:AsyOptProb} yields
\begin{lem} \label{lem:SingleSpikeAsympOpt}
For each $\ell_1 \geq 1$ and for each $(n,p_n)$ pair where $p_n > 2$,
\beq \label{eq:optSingSpike}
   \kappa(\DelAsy(\ell_1; \eta_1^*(\lambda(\ell_1)); n , p_n, \gamma)) =    \min_{(\eta_i)_{i=1}^r} \kappa(\DelAsy(\ell_1;(\eta_i); n , p_n, \gamma)).
\eeq
\end{lem}

Let $\kappa_1^{*}(\ell)$  denote the value of [both sides of] equation (\ref{eq:optSingSpike}).
We can show that the empirical performance tends to this theoretical value.
Because  $\eta_1^*$ collapses the vicinity of the bulk to $1$, the spiked model
asymptotics cause, with high probability for large $n$, the full sequence
$(\eta_1^*(\lambda_{i,n}))_{i=1}^{p_n}$ 
to be identical to the simpler sequence $(\eta_1^*(\lambda_{1,n}),1,1,\dots,1)$.
This in turn is an asymptotically small perturbation of 
the deterministic sequence $(\eta_1^*(\lambda_1),1,1,\dots,1)$,
whereby $\lambda_{1}$ replaces the stochastic $\lambda_{1,n}$.
Corollary \ref{cor:OptimalityWTuning} and Theorem \ref{lem:AsyShrinkOpt} yield:
\begin{cor}  \label{cor:formalOptOneSpike}
Always assuming  {\bf [PGA]} and {\bf [SPIKE]},
we have the following limit in probability and almost surely:
\[
\kappa(\DelEmp(\ell_1,(\eta^*_1(\lambda_{i,n})); n ,p_n)) \goto  \kappa_1^{*}(\ell_1;\gamma) , \qquad n \goto \infty,
\]
as well as optimality of that limit:
\[
val(\KEmp(\ell_1;n,p_n)) \goto  \kappa_1^{*}(\ell_1;\gamma) , \qquad n \goto \infty.
\]
\end{cor}


\section{Multi-spike Case}
\label{sec:ShrinkMultiSpike}

We now turn to the multi-spiked case $r > 1$. 
Suppose that a hypothetical optimal shrinker $\eta()$ collapses the vicinity of the bulk, so that 
$\eta(\lambda_{i,n})=1$ for eigenvalues inside the asymptotic support of the bulk.
If there are  $k$ limit eigenvalues $\lambda_i = \lambda(\ell_i; \gamma ) > \lambda_+(\gamma)$ 
emerging from the bulk, we can restrict
the range of the maximum and minimum somewhat:
\begin{equation} \label{eq:optpivot}
 \kappa(\Delta) = \frac{ \max \left[ 1,\max_{1 \leq i \leq k}  \nu_+(A(\ell_i,\eta_i)) \right] }{\min \left[  1/\ell_{k+1}, \min_{1\leq i \leq k}  \nu_-(A(\ell_i,\eta_i)) \right]} .
\end{equation}

\subsection{Using the single-spike nonlinearity in the multi-spike case}

Previous work in \cite{donoho2013optimalShrinkage} 
may lead one to expect the optimal single-spike shrinker 
to be applicable to the multi-spike case
and to be also optimal in that case.  
Such hopes founder in general, for the following reason.
A performance analysis of the single-spike-optimal 
nonlinearity {\it in the multi-spike case} shows
that the single-spike performance measures
$\nu_{\pm}(A(\ell,\eta_1^*))$ {\it evaluated
at a sub-principal eigenvalue $ \ell \in (\ell^+, \ell_1]$}
can `stick out' beyond the eigenvalue range 
\[
[ \nu_{-}(A(\ell_1,\eta_1^*)), \nu_{+}(A(\ell_1,\eta_1^*))]
\]
which defines $\kappa^*_1$. Consequently, the single-spike
optimal nonlinearity \-- when used in  
the multi-spike case \--
may exhibit worse $\kappa$-loss performance
than it does in the single-spike case.
As it turns out, this 'sticking-out' phenomenon
can't happen whenever 
$\gamma$ is sufficiently large $\gamma$. 

\begin{thm} \label{thm:MultiSingle}
Let $\gamma_m^* = \gammamvals \approx \gammamval$ 
denote the unique solution for $\gamma > 0$ of the equation
\[
     \gamma - \sqrt{\gamma}\sqrt{1+\gamma} = \frac{- \gamma}{1+\gamma}.
\]
For $\gamma > \gamma_m^*$, 
the single-spike optimal nonlinearity, $\eta^*_1$, 
if applied in the multi-spike case, $r > 1$,
is also optimal. 
\end{thm}
}

\subsection{Optimal procedure in the multi-spike case}

We now construct a multi-spike optimal nonlinearity $\eta_m^*$ to
exhibit, for a given top theoretical eigenvalue $\ell_1$, 
the same performance in the multi-spike case 
as does $\eta_1^*$ in the single-spike case.  
To achieve this, the construction guarantees 
inequalities 
\[
\nu_{-}(A(\ell_1,\eta_1^*)) \leq \nu_{-}(A(\ell,\eta_m^*)), \qquad 
  1 \leq \ell \leq \ell_1.
\]
\[
\nu_{+}(A(\ell_1,\eta_1^*)) \geq \nu_{+}(A(\ell,\eta_m^*)), \qquad 
  1 \leq \ell  \leq \ell_1.
\]
ensuring  the following {\it no-sticking-out} condition for each $\ell \in [1,\ell_1]$ :
\[
[ \nu_{-}(A(\ell,\eta_m^*)), \nu_{+}(A(\ell,\eta_m^*))] \subset
[ \nu_{-}(A(\ell_1,\eta_1^*)), \nu_{+}(A(\ell_1,\eta_1^*))].
\]

\begin{thm} \label{thm:OptimalShrinkage} 
Consider a multi-spike configuration 
$(\ell_i)_{i=1}^r$ with $1 \leq \ell_r <  \ell_{r-1} < \dots < \ell_1$.
We provide an explicit optimal nonlinear shrinker, 
$\eta_m^*(\lambda; \lambda_1, \gamma)$, in (\ref{eq:optmultispike}) below; this depends on $\gamma$ and 
on the tuning constant $\lambda_1$ 
(i.e. the limiting top empirical eigenvalue
$\lambda(\ell_1)$).

\begin{enumerate}

\item  {\bf (Optimal Asymptotic Pivot).}
Here {\it optimal} means that the condition number of the asymptotic pivot 
is the smallest achievable by any nonlinearity $\eta(\cdot)$  at the 
given spike configuration:   
\[
    \kappa( \Delta(\eta^*_m)) = \min_{\eta}   \kappa( \Delta(\eta)).  
\]
More explicitly, as soon as $p > r+1$:
\[
  \kappa(\DelAsy((\ell_i); (\eta^*_m(\lambda(\ell_i))); n , p, \gamma)) =    \min_{(\eta_i)_{i=1}^r} \kappa(\DelAsy((\ell_i);(\eta_i); n , p, \gamma)).
\]
 \item {\bf (Optimal Asymptotic Loss).} The optimal condition number obeys:
\beq \label{eq:OptAsyLoss}
    \kappa( \Delta(\eta^*_m)) = \kappa_1^*(\ell_1).
\eeq
 That is, the optimal $\kappa$-loss is the same in the multispike model as in that instance of the single-spike model 
 having the same value $\ell_1$ for the  largest spike. 

\item {\bf (Optimal Shrinkage Formula).} 
Suppose $\ell_1 > \ell_1^+(\gamma)$,  and
let $\nu_-^{1,*} = \nu_-(\ell_1;\eta_1^*(\lambda(\ell_1),\gamma))$.
Define 
\[
   \ell_{\eta_m^*}^+(\ell_1,\gamma) \equiv \frac{1}{\nu_-^{1,*}}.
\]
Also set
\[
    a \equiv c^2/\ell+s^2, \qquad b \equiv s^2/\ell+c^2.
\]
Then $\nu_-^{1,*} \neq 1/(a\ell)$ and
the optimal multi-spike shrinker obeys
\beq \label{eq:optmultispike}
    \eta_m^*(\lambda(\ell))  = \left \{ \begin{array}{ll}
                 \frac{(\nu_-^{1,*} - b)}{a -1/(\ell\cdot\nu_-^{1,*})}  & \ell >  \ell_{\eta_m^*}^+ \\
                 1           &  \ell \leq  \ell_{\eta_m^*}^+
                 \end{array}
                 \right . .
\eeq
\item {\bf (Collapse of the Bulk.)} 
Define $\lambda_{\eta_m^*}^+ = \lambda(\ell^+_{\eta_m^*}(\gamma);\gamma)$. Then
$\lambda_{\eta_m^*}^+ > \lambda_+(\gamma)$, so $\eta_m^*( \cdot ; \lambda_1) $ collapses the vicinity of the bulk to $1$,
for each $\lambda_1 > \lambda_+(\gamma)$.
\end{enumerate}
\end{thm}

\noindent
{\bf Extension to  clustered spikes.} 
 {\sl The theorem statement and proof specifically exclude the case of clustered spike
eigenvalues \-- i.e. $\ell_i = \ell_{i+1}$ for some $1 \leq i < r$.
It is possible to show that the same conclusions
hold in the clustered case, meaning that the same
nonlinearity is optimal and  the same formulas for
the optimal $\kappa$-loss hold. 
However, in the clustered case, the underlying
asymptotic pivot and the underlying proofs are different.
 } 

At the top eigenvalue, we have
 \[
      \eta_m^*(\lambda_1; \lambda_1) = \eta_1^*(\lambda_1). 
 \]
 Consequently, 
 \[
  \eta_m^*(\lambda_1; \lambda_1)  \sim \frac{1}{1+\gamma} \lambda_1, \qquad \lambda_1 \goto \infty.
 \]

\subsection{Rigorous Optimality}
Theorem \ref{thm:OptimalShrinkage} offers a certain formal optimality result;
it shows  that applying
$\eta^* = (\eta^*_m(\cdot;\lambda_1))$ to the limit eigenvalues $\lambda_i$
produces a sequence that minimizes the condition number of the asymptotic pivot. 
This is two steps removed from a `full' optimality result. We now take those steps. 
Our first step shows that applying the oracle procedure 
$\eta^*_m(\cdot; \lambda_1)$ to the empirical eigenvalues $\lambda_{i,n}$ 
also gives the (asymptotically) optimal $\kappa$-loss. 
Now the nonlinearity $\eta^*_m(\cdot; \lambda_1)$ collapses the bulk,
and we have the convergence $\eta^*_m(\lambda_{i,n}; \lambda_1) \goto \eta^*_m(\lambda_i; \lambda_1)$
as $n \goto \infty$ for $i=1,\dots,r$. 
Lemma \ref{lem:blocks} and Theorem \ref{lem:AsyShrinkOpt} 
and the immediately following corollaries apply here, and we obtain:

\begin{cor} \label{cor:AsympOracle}
Always assuming [SPIKE] and [PGA],
the asymptotic performance of the oracle procedure has the following limit:
\[
\kappa(\DelEmp((\ell_i),(\eta^*_m(\lambda_{i,n};\lambda_1)); n ,p_n)) \goto_{a.s.}  \kappa_1^{*}(\ell_1;\gamma) , \qquad n \goto \infty,
\]
which is asymptotically optimal:
\[
val(\KEmp((\ell_i);n,p_n)) \goto_{a.s.}  \kappa_1^{*}(\ell_1;\gamma) , \qquad n \goto \infty.
\]
\end{cor}

Our second step is more delicate. The oracle procedure 
$\eta^*_m(\cdot; \lambda_1)$ is tuned using  the top {\it limit} eigenvalue $\lambda_1$. 
However we only observe the empirical eigenvalue $\lambda_{1,n}$, 
not the limit $\lambda_1$. 
In practice, we would tune using  $\lambda_{1,n}$, giving the fully empirical procedure 
$\eta_{i,n}^e \equiv \eta_{m}^*(\lambda_{i,n}; \lambda_{1,n},\frac{p_n}{n})$. Lemma \ref{lem:AsympMShrink} shows that $(\eta^e)$ 
is also asymptotically optimal. 

\begin{lem} \label{lem:AsympMShrink}
Always assuming the spike model [SPIKE] and the [PGA] asymptotic $p_n/n \goto \gamma$,
we have both almost surely and in probability
\[
      \max_{1\leq i \leq \min(n,p_n)} | \eta_{i,n}^e  - \eta^*_m(\lambda_{i,n}; \lambda_1) |  \goto 0, \qquad n \goto \infty.
\] 
Consequently, the empirical shrinker is asymptotically optimal:
\[
\kappa(\DelEmp((\ell_i),(\eta_{i,n}^e); n ,p_n)) \goto_{a.s.}  \kappa_1^{*}(\ell_1;\gamma) , \qquad n \goto \infty.
\]
\end{lem}

\section{A Parameter-Free Minimax Nonlinearity}
\label{sec:MinimaxInterp}
  
In general, achieving optimal performance requires tuning
the shape of the nonlinearity $\eta(\lambda; (\ell_i))$ based 
on the underlying configuration of spike
eigenvalues $(\ell_i)$. 
\footnote{For $\gamma > \gamma_m^* = \gammamvals$, 
Theorem \ref{thm:MultiSingle}
observed that we can even achieve optimal performance
with a  {\it tuning-parameter-free} nonlinearity,
i.e. applying the single-spike nonlinearity $\eta_1^*$
even in the multi-spike setting.
For $\gamma \leq \gamma_m^*$ Theorem \ref{thm:OptimalShrinkage} 
shows that tuning can be reduced to {\it one real parameter}, 
the top spike value $\ell_1$,
yielding optimal performance by $\eta_m^*(\lambda; \ell_1)$.}

This section presents the best available guarantees 
which can be offered by a  tuning-parameter-free nonlinearity
to be applied eigenvalue by eigenvalue, and where the nonlinearity
is not tuned in any way
using the configuration of the eigenvalues \-- in particular
the shape of the nonlinearity does not depend on $\ell_1$
(as  $\eta_m(\cdot; \ell_1)$ does).

\subsection{Minimax \texorpdfstring{$\kappa$}{k}-Loss}

We focus on the class $\cL_r$ of {\it all} $r$-spike models with 
$\ell_1 > \ell_2 > \dots > \ell_r > 1$,
and study the minimax $\kappa$-loss
\[
	\min_\eta \max_{(\ell_i) \in \cL_r} \kappa((\ell_i),\eta).
\]
This provides the best possible guarantee for performance
of a single nonlinearity used across a whole collection 
of spike configurations.

For some (possibly joint) nonlinearity $\eta$ 
we measure the global worst-case performance:
\[
   K_{r}^*(\eta)  = \max_{\cL_{r}} \kappa( \Delta( (\ell_i) ; \eta(\cdot) ) ).
\]
Minimizing this quantity across all shrinkage choices $\eta$ yields the
minimax $\kappa$-loss across {\it all } spike models \-- 
i.e. we can obtain a {\it globally} minimax nonlinearity.
By Theorem \ref{thm:MultiSingle}, 
for $\gamma > \gamma_m^*  = \gammamval$, 
the single-spike nonlinearity is optimal at every spike configuration, 
and hence it is globally minimax over that range of $\gamma$. 
A different approach
intentionally `mis-tunes' the optimal rule,
setting the assumed $\ell_1$ to $ \infty$ independently of
the actual situation;  the result, $\eta_m^*(\cdot ; \infty, \gamma)$, 
is a minimax procedure for each $\gamma > 0$. 
This procedure is tuning-parameter-free, as $\gamma=p/n$ is known
and properly not a tunable parameter.

\newcommand{\mmbp}{\mbox{\sc \footnotesize mm}}
\begin{thm} \label{thm:minimaxEta}
({\bf Minimax Loss, and a Minimax Procedure}) 

\bitem 
\item {\bf Minimax $\kappa$-loss. }
For each $r \geq 1$, the optimum $\kappa$-loss which can be guaranteed uniformly 
over the class $\cL_r$  of all $r$-spike models  $\ell_i \geq 1$, $i=1,\dots, r$
is given by:
\[
   \inf_{\eta} \max_{(\ell_i) \in \cL_{r}} \kappa( \Delta( (\ell_i) ; \eta(\cdot) ) ) = \kappa_1^*(\infty,\gamma) =  \frac{1+\sqrt{\frac{\gamma}{\gamma+1}}}{1-\sqrt{\frac{\gamma}{\gamma+1}}}.
\]

\item {\bf Minimax Shrinker $\etammnl$. }
Define \[
\etammnl(\cdot; \gamma) \equiv \eta_m^*(\cdot ; \infty,\gamma) \equiv \lim_{\ell_1 \goto \infty} \eta_m^*(\cdot; \ell_1,\gamma).
\] 
The limit exists and can be characterised using functions $a(\ell)$ and $b(\ell)$ 
from part 3 of Theorem \ref{thm:OptimalShrinkage}, as follows.
For $\nu_-^{\mmbp} \equiv \nu_-^{\mmbp}(\gamma) = 1 - \sqrt{\frac{\gamma}{\gamma+1}}$, 
note that $a(\ell) \ell \neq 1/\nu_-^{\mmbp}$ on $\ell > 1/\nu_-^{\mmbp} $ and set 
\beq \label{eq:minimaxdef}
    \etammnl(\lambda(\ell))  = \left \{ \begin{array}{ll}
                 \frac{(\nu_-^{\mmbp} - b)}{a -1/(\ell\cdot\nu_-^{\mmbp})}  & \ell >  1/\nu_-^{\mmbp} \\
                 1           &  \ell \leq  1/\nu_-^{\mmbp}
                 \end{array}
                 \right . .
\eeq
The nonlinearity so defined obeys:

\beq \label{eq:asyinequality}
      \nu_-(A(\ell, \etammnl)) \geq \nu_-^{\mmbp}(\gamma), \qquad \forall \ell > 1.
\eeq
\item {\bf Minimaxity of $\etammnl$}
The nonlinearity $\etammnl$ delivers the optimal guarantee: 
\[
K_{r}^*( \etammnl  ) =  \max_{(\ell_i) \in \cL_{r}} \kappa( \Delta( (\ell_i) ; \etammnl (\cdot) ) )  =  \frac{1+\sqrt{\frac{\gamma}{\gamma+1}}}{1-\sqrt{\frac{\gamma}{\gamma+1}}}.
\]
\eitem

\end{thm}

\section{Performance Comparisons}
\label{sec:PerformanceComparisons}

We now consider scalar nonlinearities that approximate the optimal joint nonlinearity, $\eta_m^*$. 
We consider the minimax nonlinearity $\etammnl(\cdot ; \gamma)$, 
just introduced in the last section, alongside the following:
\bitem
\item {\it Generalized Soft Thresholding $\etammst$.}  
Apply the (generalized) soft thresholding nonlinearity specifically tuned with
 threshold  at the bulk edge $\lambda_0 = \lambda_+(\gamma)$ 
and slope $b = b(\gamma) = \frac{1}{1+\gamma}$. 
\beq \label{eq:defGST}
 \etammst(\lambda) \equiv  1 + \frac{1}{1+\gamma} (\lambda - \lambda_+(\gamma))_+   .
\eeq
Like the optimal nonlinearity, $\etammst$  has  slope $1/(1+\gamma)$ for large values of $\lambda$,
and a dead zone; 
in this case the dead zone ends at the bulk edge $\lambda_+(\gamma)$.
We use the label {\sc mmst}, because tuning the nonlinearity in this specific way makes it minimax.
\item {\it Optimal Nonlinearity for Precision Estimation.} 
\cite{donoho2013optimalShrinkage} found the optimal nonlinearity for estimation 
of the so-called precision matrix $\Sigma^{-1}$ with respect to Frobenius norm loss, under the spiked model,
and gave explicit expressions; using the notations of this paper we can write:
\begin{equation}
\label{eq:optPrecision}
\etapnl(\lambda) = \left\{ \begin{array}{ll}  
\frac{\ell }{\ell s^{2}+c^{2}} & \ell > \ell_+(\gamma),\\
                                                1 & \ell \leq \ell_+(\gamma) .
                                                \end{array}   \right.
\end{equation}
Compare this formula  with the alternate formula (\ref{eq:altSingleSpikeShrink})  for the 
single-spike-optimal shrinker in Lemma \ref{lem:spikeFormulaEquivalent}.
One sees immediately that $\eta_{NPL}$ has the same asymptotic slope \--  $\frac{1}{1+\gamma}$ \--
and that there is again a dead zone; this time extending out only to the bulk edge $\lambda_+(\gamma)$.
\eitem

\begin{table}
\footnotesize
\begin{tabular}{| l | l | l | c|}
\hline
Rule & Notation & Source & Optimality Properties\\
\hline
Optimal 1-spike Nonlinearity    & $\eta_1^*$ & Theorem \ref{thm:spike_1} & Individually Optimal, single-spike case  \\
Optimal Nonlinearity                  & $\eta_m^*$ & Theorem \ref{thm:OptimalShrinkage} & Individually Optimal and  Globally Minimax  \\
Minimax Nonlinearity                & $\etammnl$    & Theorem \ref{thm:minimaxEta} & Globally minimax\\
Generalized Soft Threshold     & $\etammst $ & (\ref{eq:defGST})& Globally minimax \\
Precision Nonlinearity               & $\etapnl$ &  (\ref{eq:optPrecision}) & Optimal Frobenius loss for estimation of $\Sigma^{-1}$ \\
\hline
\end{tabular}
\end{table}

The performance of these rules depends 
on the number of underlying spikes, the spike amplitudes, and
$\gamma$. There is an endless variety of possible combinations 
we could study for performance evaluations.

For performance comparisons, we consider both $\kappa$-loss and RSRG. Our results
above yield a formula for the optimal RSRG.

\begin{cor}
{\bf Optimal Relative Sharpe Ratio Guarantees.}
Again assuming [SPIKE] and [PGA], define:
\[
\RSRG^*(\ell_1,\dots,\ell_r) = \lim_{n \goto \infty} \inf_{\eta} \RSRG(\hat{\Sigma}(\eta), \Sigma).
\]
This is the smallest RSRG achievable 
asymptotically by an orthogonally invariant procedure.
Then, 
\[
\RSRG^*(\ell_1,\dots,\ell_r)  \goto {\sqrt{1 + \gamma}}, \qquad{ as  } \; \ell_1 \goto \infty, 
\]
where either $r=1$, or where, in the limit process, we keep $\ell_2,\dots,\ell_r$ fixed.
\end{cor}

The {\it relative regret} is the performance deficit relative to best achievable performance, expressed in percentage terms:
\[
    \kappa \mbox{-loss Reg}[\eta,(\ell_i)] = 100 \cdot  \left ( 1- \frac{\kappa[\eta^*,(\ell_i)]}{\kappa[\eta,(\ell_i)]} \right ) .
\]
\[
    \mbox{\RSRG Reg}[\eta,(\ell_i)] = 100 \cdot  \left ( 1- \frac{\RSRG[\eta^*,(\ell_i)]}{\RSRG[\eta,(\ell_i)]} \right ) .
\]
(Dependence of both sides on $\gamma$ is suppressed).

This measure depends -- possibly sensitively -- on the spike configuration $(\ell_i)$,
and its interaction with $\eta$. It varies from zero up to some maximum value,
which depends on $\gamma$ only:
\[
      \mbox{MaxReg}(\eta,\gamma) \equiv \max_r \max_{(\ell_i) \in \cL_r}   \mbox{Reg}[\eta,(\ell_i)] .
\]

Figure \ref{fig:MaximumRegretVGammaSoftThresh} 
shows  maximal regret for three rules  which 
are globally minimax, but not individually optimal.
Perhaps surprisingly,
these three rules, though simpler than the optimal nonlinearity,
never suffer  much regret, at least for  $\gamma \leq 2$.
In particular, 
\bitem
\item {\it Minimax soft thresholding }  $\etammst$
 is always within a few percent of optimal
 over the range $0 \leq \gamma \leq 2$;
\item {  The {\it single-spike optimal nonlinearity} is   {\it within one percent} of optimal -- across all $\gamma$; while of course we know it is
optimal for $\gamma > \gamma_m^* =\gammamval$.}
 \item The nonlinearity $\etapnl$, optimal for estimation of precision matrix $\Sigma^{-1}$
in Frobenius loss  under the spiked model, is always within several percent
of  RSRG-regret optimal over the full range $0 \leq \gamma \leq 2$.
\item Finally, the {\it minimax nonlinearity}  $\etammnl \equiv \eta_m^*(\cdot; \infty)$
 is always within a few percent of  optimal, over the same range; while over the important range $\gamma < 1/2$,
 it has even lower maximal RSRG-regret than the
single-spike nonlinearity.
 \eitem
 Figure \ref{fig:MaximumKappaLossVGammaSoftThresh} shows that the
 maximal regrets in $\kappa$-loss are slightly larger, yet still
 in the range of several percent.
 
 
\begin{figure}
\centering
\includegraphics[width=1\textwidth]{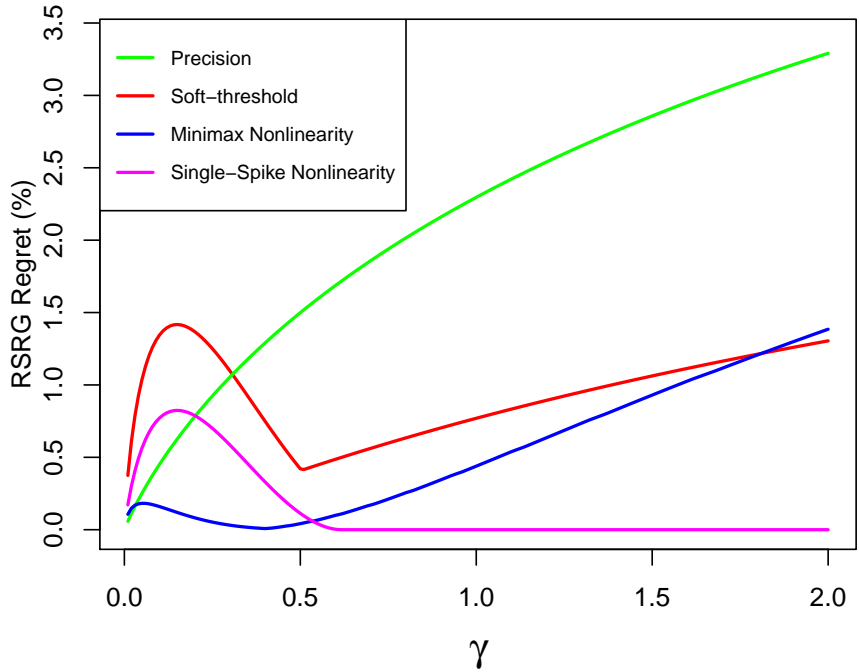}
\caption{Worst-case RSRG regrets for MMST, Precision,  Minimax and Single-spike nonlinearities.
These rules are within a few percent of optimal for $\gamma \leq 2$. The single-spike nonlinearity is within one percent of optimal for all $\gamma$ and is precisely optimal for $\gamma > \gamma_m^* =\gammamval $. The minimax nonlinearity has still better worst-case regret than the single-spike nonlinearity over the important range $0 \leq \gamma < 1/2$.}
\label{fig:MaximumRegretVGammaSoftThresh}
\end{figure}


\begin{figure}
\centering
\includegraphics[width=1\textwidth]{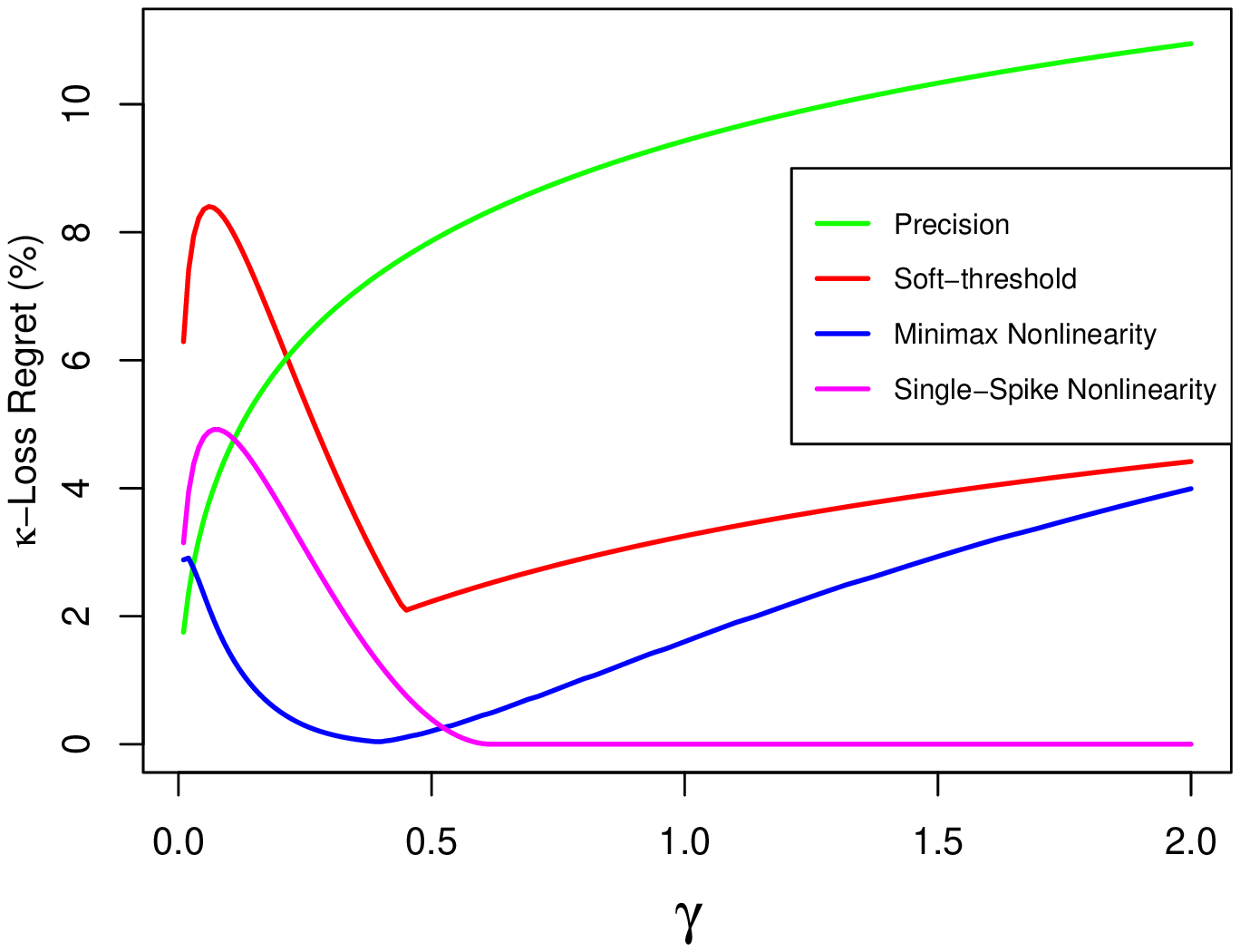}
\caption{Worst-case $\kappa-$loss regrets for MMST, Precision,  Minimax, and single-spike nonlinearities.
These rules are always within several percent of 
optimal for $\gamma \leq 2$.The single-spike nonlinearity is within five percent of optimal for all $\gamma$ and is precisely optimal for $\gamma > \gamma_m^* = \gammamval$. The minimax nonlinearity has still better worst-case regret than the single-spike nonlinearity over the important range $0 \leq \gamma < 1/2$.}
\label{fig:MaximumKappaLossVGammaSoftThresh}
\end{figure}
\section{Least-Favorable \texorpdfstring{$\mu$}{mean vector}}
\label{ssec:LFUF}

Recall the interpretation of $\kappa$-loss given in the introduction for the tasks of
MUCE and MTDA. In each task the covariance matrix estimator $\hat{\Sigma}$ is
used along with many possible vectors $\mu$ to solve 
a range of individual allocation/discrimination problems.
There is (at least one) least-favorable  $\mu$ vector underlying the RSRG approach, i.e.
a forecast vector $\mu$ at which the RSRG guarantees are still effective, but only just.
%
%
\begin{thm} \label{thm:worst-case-forecast}
In either the single-spike case $r=1$ or the multiple-spike case $r>1$,
let $U_{1,n}$ denote the top eigenvector of $\Sigma$, $V_{1,n}$
the top eigenvector of $\bS$, and $W_{i,n}$ the result of Gram-Schmidt
orthogonalization of $(U_{1,n},V_{1,n})$. 
There exist constants $\alpha_j = \alpha_j((\ell_i),\gamma)$, $j =1,2$ such that 
\[
\mu_n \propto \alpha_1 U_{1,n} 
      + \alpha_2 W_{2,n}
\]
asymptotically attains the worst-RSRG.
\end{thm}

This worst-case forecast lies entirely in the span of $U_{1,n}$ and $V_{1,n}$.
In the empirical finance setting, these vectors are the `model market direction' and the `recent market direction' respectively. The theorem tells us that portfolios which hold an appropriate compromise between recent and model market directions are the ones that the covariance estimator is most `concerned about'. Namely, their future Sharpe ratios
are the ones yielding the largest disappointment, 
relative to their recent Sharpe ratios. 
The optimal shrinkage described here is 
designed to optimally limit such disappointment.
In modern trading, massive amounts of capital are being invested
along the recent market direction; such investments are made by momentum or returns-chasing investors.
{\it Such behavior was shown here to expose investors to future Sharpe ratio disappointments.}
The optimal shrinker proposed here {\it in effect} anticipates the need to protect such investors.

\section{Conclusion}
The optimal equivariant shrinker for relative condition number loss
has been derived under the standard spiked model.
It significantly shrinks  even very substantial eigenvalues  \--
by a factor of roughly $\frac{1}{1 + \gamma}$.
It also imposes a surprisingly large dead zone, where all eigenvalues are collapsed to 1.
It performs very much like generalized soft thresholding $\etammst$
having  deadzone extending to the bulk edge and slope $\frac{1}{1 + \gamma}$ above it. It also performs
very much like the optimal rule  $\etapnl$ derived in \cite{donoho2013optimalShrinkage}
for estimating the precision matrix $\Sigma^{-1}$
under Frobenius loss: $\Vert \Sigma^{-1} - \hat{\Sigma}^{-1}\Vert_F$.
Both $\etapnl$ and $\etammst$ have  
small regret compared to the optimal rule.

The shrinker is at odds with certain `sophisticated' conventional
practices regarding eigenvalues. 
A vast literature discusses the choice of the number of factors in a factor model; many practitioners
propose to use {\it any} significant factors, 
even weak \-- barely detectable \-- ones \cite{kritchman2009non}\cite{passemier2012determining}\cite{owen2016bi}. 
This paper's shrinker would in many cases not include factors in a model {\it even when} they
are highly statistically significant.
Indeed, from the viewpoint of this paper, eigenestimates of certain 
factors can be very noisy -- {\it too noisy to be helpful} -- 
despite statistical significance. 

The paper aligns well with other, `naive' conventional practices.
In high-dimensional discriminant analysis the `naive Bayes' 
practitioner tradition would suggest that naive diagonal covariance estimates can be better than 
more seemingly accurate empirical covariance estimates, 
{\it even when there is undeniable correlation} \cite{hand2001idiot},\cite{bickel2004some}.
In mean-variance portfolio allocation, some practitioners advocate for
diagonal covariance matrices \cite{stivers2016mitigating} even when correlations exist.
This paper's conclusions lend additional theoretical support to this common practice.

\appendix
\section{Proofs Deferred from Main Text}
\label{sec:Proofs}

\subsection{Proofs for Section \ref{sec:intro}}
\subsubsection*{Proof of Theorem \ref{thm:asymptotic_optimal_loss}}

\begin{proof}

The trivial constant estimator 
$\hat{\Sigma} \equiv \hat{\Sigma}(S) =I_{p \times p}$ $\forall S$ 
is orthogonally equivariant.
It has loss $ \kappa(\Delta(\Sigma, I_{p \times p})) = \kappa(\Sigma)$. 
In the spiked model $\Sigma \gtrsim I$, so is always positive definite, so $\kappa(\Sigma) < \infty$.
Hence there are OE procedures with finite loss.

Consider some such OE procedure $\hat{\Sigma}$.
The condition number $\kappa(\Delta(\Sigma, \hat{\Sigma}))= \infty $ 
in cases where $\hat{\Sigma}$
has a null space. As $\hat{\Sigma}$
has finite loss, it must also be positive definite.

OE estimators have a diagonal representation, 
exposed in Section \ref{sec:OrthoInvar}, Lemma \ref{lem:diagrep}.
In this representation,  a positive definite matrix 
will have strictly positive entries: $\eta_i > 0$, $1 \leq i \leq p$. Hence
$Ave_{i>r} \eta_i > 0$,
and the rescaled procedure 
$ \tilde{\Sigma}^{(r)} = \left( Ave_{i>r} \eta_i \right)^{-1} \hat{\Sigma}$ is well-defined.
Since $\kappa-$loss is scale invariant, 
the loss of this rescaled estimator is the same as that of $\hat{\Sigma}$.
Note that such a rescaled procedure $\tilde{\Sigma}$ is also orthogonally equivariant.

Let $OEU^{(r)}_{n,p_n}$ denote the collection of orthogonally-equivariant procedures
$\tilde{\Sigma}^{(r)}$ whose diagonal representations
obey the constraint $\tilde{\eta} \in \EtaRNP$,
i.e. obey ${Ave}_{i>r} \tilde{\eta}^{(r)}_i = 1$. The rescaling argument of the
last few paragraphs shows that
the range of losses achievable by procedures  $\hat{\Sigma} \in OE$ is the same as
those achievable by procedures $\tilde{\Sigma} \in OEU$. In particular:
\begin{equation}\label{eq:normalization}
\inf_{OE_{n,p_n}} \kappa(\Delta(\Sigma,\hat{\Sigma}_{n,p_n})) = \inf_{OEU^{(r)}_{n,p_n}} \kappa(\Delta(\Sigma,\tilde{\Sigma}_{n,p_n})).
\end{equation}
Section \ref{sec:StudyAsyPiv} studies the optimization problem on the right-hand side of \eqref{eq:normalization}. It defines $val({K}^e((\ell_i);n,p_n))$ to be the value of this optimization problem (which is a random quantity, as $\tilde{\Sigma}_{n,p_n}$ depends on $S$). Theorem \ref{lem:AsyShrinkOpt} shows that $val({K}^e((\ell_i);n,p_n))$ has an almost sure limit as $p_n/n \goto \gamma$. 
This almost sure limit is equal to the value provided by Theorem 
\ref{lem:AsyOptProb}, which is the common value of the (deterministic) problem
$K^a((\ell_i); n, p_n,\gamma)$
for all large enough $n$ and $p_n > 2r$. 
By Theorem \ref{thm:OptimalShrinkage}, this value is equal to
$ \kappa^*(\ell_1; \gamma)$. 

Combining the above yields this almost-sure and in-probability limit:
\begin{eqnarray*}
\kappa^*(\ell_1;\gamma) &=& \liminf_{n \goto \infty} \inf_{OE_{n,p_n}} \kappa(\Delta(\Sigma,\hat{\Sigma})). 
\end{eqnarray*}

\end{proof}
\subsubsection*{Proof of Theorem \ref{thm:asymptotic_optimal_nonlinearity}}


\begin{proof}
The explicit form of the optimal nonlinearity 
is given in Theorem \ref{thm:OptimalShrinkage},
initially tuned by the (unobservable) underlying true
eigenvalue $\ell_1$.
For the specific empirical tuning 
mentioned in the statement of this Theorem,
Corollary \ref{cor:AsympOracle} and Lemma \ref{lem:AsympMShrink} prove 
\begin{equation*}
\lim _{n \rightarrow \infty} L(\hat{\Sigma}(\eta^*), \Sigma) =_{a.s} \kappa_1^*(\ell_1; \gamma).
\end{equation*}
\end{proof}

\subsection{Proofs for Section \ref{sec:Basic Tools}}
\subsubsection*{Proof of Lemma \ref{lem:ReduceConditionNumber}}
\begin{proof}
Suppose $X_t \sim (\mu, \Sigma)$ be the vector of returns at period $t$. Let $\HS$ be the estimated covariance of the returns. 
For a portfolio with holdings $h\in\R^{p}$, 
we define the Sharpe ratio $\SR(h) = \SR(h; \mu, \Sigma)$ by\footnote{Traditionally the Sharpe ratio 
involved the notion of risk-free rate which was typically positive.
In recent times, the risk-free rate has often been close to zero, and here
we simply take the risk-free-rate as zero.}
\[
\SR(h)  = \frac{E h'X_t}{SD(h'X_t)} = \frac{h'\mu}{\sqrt{h' \Sigma h}} .
\]
Given $\mu$ and $\Sigma$, in problem (\ref{MRK_Prob}) the maximum
SR is achieved by any portfolio proportional to
\[
h^*= \Sigma^{-1}\mu .
\]
The Sharpe ratio achieved by an optimal portfolio is 
\[
\SR^*(\mu, \Sigma) \equiv \sqrt{\mu' \Sigma^{-1}\mu} ,
\]
A user with forecast $\mu$ who
allocates a portfolio using the estimated covariance matrix
$\HS$ {\it as if} it were the true covariance $\Sigma$
will obtain the portfolio 
\[
\hat{h}=\frac{1}{\mu'\HS^{-1}\mu} \cdot \HS^{-1} \mu,
\]
achieving the Sharpe ratio
\[
\SR(\HS;\mu, \Sigma)=\frac{\mu' \HS ^{-1}\mu}{\sqrt{\mu'\HS^{-1}\Sigma{\HS}^{-1}\mu}}.
\]
We have 
\begin{eqnarray}\label{eq:raw_form}
\frac{\SR({\HS};\mu, \Sigma)}{\SR(\Sigma;\mu, \Sigma)} &=& \frac{\mu' {\HS}^{-1}\mu}{\sqrt{\mu' {\HS}^{-1} \Sigma {\HS}^{-1}\mu}\sqrt{\mu' \Sigma^{-1}\mu}} \nonumber \\
&=& \frac{\Vert x\Vert_{2}^{2}}{\sqrt{x'\hat{\Sigma}^{-\frac{1}{2}}\Sigma\hat{\Sigma}^{-\frac{1}{2}}x}\sqrt{x'\hat{\Sigma}^{\frac{1}{2}}\Sigma^{-1}\hat{\Sigma}^{\frac{1}{2}}x}},
\end{eqnarray}
where  $x \equiv \hat{\Sigma}^{-\frac{1}{2}}\mu$.
WLOG, assume $\Vert x\Vert=1$. We can
use the matrix version of Kantorovich's inequality \cite{marshall1990matrix} to bound the denominator. 
For a positive semi-definite matrix $A$, and a unit vector $x$, $\|x\|_2=1$,
\[
(x'Ax)(x'A^{-1}x)\leq\frac{1}{4}(\kappa(A)+\frac{1}{\kappa(A)}+2), 
\]
where $\kappa(A)$ denotes the condition number of $A$. In particular,
if  $x$ is in
the span of the most dominant and the least dominant eigenvectors
of $A$ and $x'Ax=\frac{\lambda_{1}(A)+\lambda_{min}(A)}{2}$, the expression holds with equality. 
Applying this to our
problem, and noting that for two full-rank conformable square matrices $A$ and $B$, $\kappa(AB) = \kappa(BA)$ and $\kappa(A) = \kappa(A^{-1})$,
we get
\begin{eqnarray}
\frac{\SR(\Sigma;\mu, \Sigma)}{\SR(\HS;\mu, \Sigma)}  &\leq& \frac{1}{2}\sqrt{\kappa(\hat{\Sigma}^{-\frac{1}{2}}\Sigma\hat{\Sigma}^{-\frac{1}{2}})+\frac{1}{\kappa(\hat{\Sigma}^{-\frac{1}{2}}\Sigma\hat{\Sigma}^{-\frac{1}{2}})}+2} \nonumber \\ 
& = & \frac{1}{2}\sqrt{\kappa(\Sigma^{-\frac{1}{2}}\hat{\Sigma}\Sigma^{-\frac{1}{2}})+\frac{1}{\kappa(\Sigma^{-\frac{1}{2}}\hat{\Sigma}\Sigma^{-\frac{1}{2}})}+2} \label{eq:worst_case}
\end{eqnarray}

Since there is at least one $x$ for which the denominator of
(\ref{eq:raw_form}) is exactly
equal to (\ref{eq:worst_case}), the reciprocal of (\ref{eq:worst_case})
gives the worst case relative Sharpe Ratio, 
$\min_{\mu}\frac{\SR({\HS};\mu, \Sigma)}{\SR(\Sigma;\mu, \Sigma)}$.

\end{proof}

\subsubsection*{Proof of Lemma \ref{lem:blocks}}
\begin{proof}
\newcommand{\bUo}{{\textbf{U}_1}}
\newcommand{\bWo}{{\textbf{W}_1}}
\newcommand{\bWt}{{\textbf{W}_2}}
\newcommand{\bVo}{{\textbf{V}_1}}
\newcommand{\bVt}{{\textbf{V}_2}}
By definition, 
\begin{eqnarray}
\Delta_n & = & W'\Sigma^{-\frac{1}{2}} \hat{\Sigma}(\eta) \Sigma^{-\frac{1}{2}}W \nonumber \\
& = & W'\Sigma^{-\frac{1}{2}} WW' \hat{\Sigma}(\eta) WW'\Sigma^{-\frac{1}{2}}W. \label{eq:partitioned_delta_1}
\end{eqnarray}
We organize the proof into two parts. First, we establish the asymptotic block diagonal structure of $W'\Sigma^{-\frac{1}{2}} W$. Then, using the first part, we proceed to show the asymptotic block-diagonal structure for $\Delta_n$. 

\noindent
{\bf Part I:} Let's define 
\begin{eqnarray*}
\Upsilon &=& diag(\ell_1^{-\frac{1}{2}}, \cdots, \ell_r^{-\frac{1}{2}}) - I_{r\times r}\\
\textbf{W}_{1} =  [W_{1},\cdots, W_{2r}] ,
&& \textbf{W}_{2} =  [W_{2r+1},\cdots, W_{p_n}] \\
\bUo = [U_{1},\cdots, U_{r}] &&
\textbf{U}_2 = [U_{r+1},\cdots, U_{p_n}] \\
L_1 &=& diag(\ell_1, 1, \ell_2, 1, \dots, \ell_r, 1) .
\end{eqnarray*}

By [SPIKE], 
\begin{eqnarray}
W'\Sigma^{-\frac{1}{2}} W & = & W'(\bUo \Upsilon \bUo'+ I_{p_n\times p_n})W \nonumber \\
&=& W'\bUo \Upsilon \bUo' W + I_{p_n\times p_n} \nonumber \\
&=& [\bWo, \bWt]'\bUo \Upsilon \bUo' [\bWo, \bWt] + I_{p_n\times p_n} \nonumber \\
&=& \left[\begin{array}{cc}
\bWo' \bUo \\
\bWt' \bUo 
\end{array}\right] \Upsilon 
\left[\begin{array}{cc}
\bUo' \bWo & \bUo' \bWt
\end{array}\right] + I_{p_n\times p_n}. \label{eq:term1_block_diagonal_rep}
\end{eqnarray}
By construction, $\bWt' \bUo = 0_{(p_n - 2r)\times r}$. 
Hence, \eqref{eq:term1_block_diagonal_rep} is equal to 
\begin{eqnarray} \label{eq:term1_block_diagonal_rep2}
\left[\begin{array}{cc}
\bWo' \bUo \Upsilon \bUo' \bWo & 0_{r \times (p_n - 2r)} \\
0_{(p_n - 2r) \times r} & 0_{(p_n - 2r) \times (p_n - 2r)}
\end{array}\right] + I_{p_n\times p_n}.
\end{eqnarray}
By Theorem \eqref{lem:spiked}, 
$\bWo' \bUo \xrightarrow{a.s} \Gamma$ (say) in Frobenius norm. 
Here, $\Gamma \in \R ^{2r \times r}$ and 
\[
\Gamma_{i, j} = \left\{ \begin{array}{cc}
1 & i = 2j - 1, \; 1 \leq j \leq r \\
0 & \mbox{otherwise}
\end{array} \right..
\]
Therefore, \eqref{eq:term1_block_diagonal_rep2} can be rewritten
\begin{eqnarray*} 
\left[\begin{array}{cc}
\Gamma \Upsilon \Gamma' + \zeta_n & 0_{r \times (p_n - 2r)} \\
0_{(p_n - 2r) \times r} & 0_{(p_n - 2r) \times (p_n - 2r)}
\end{array}\right] + I_{p_n\times p_n}.
\end{eqnarray*}
where $\zeta_n$ is a symmetric $2r\times 2r$ remainder term that goes to $0_{2r\times 2r}$ almost surely as $n \rightarrow \infty$. Since 
\begin{equation*}
\Gamma \Upsilon \Gamma' = L_1^{-\frac{1}{2}} - I_{2r \times 2r},
\end{equation*}
we conclude that
\begin{equation} \label{eq:partI_results}
W'\Sigma^{-\frac{1}{2}}W = 
\left[\begin{array}{cc}
L_1^{-\frac{1}{2}} + \zeta & 0_{2r\times(p - 2r)}\\
0_{(p - 2r) \times 2r} & I_{(p_n - 2r)\times (p_n - 2r)}
\end{array}\right].
\end{equation}
\noindent
{\bf Part II: } In order to write $\Delta_n$ in block-diagonal form, we define 
\begin{eqnarray}
\bVo = [V_{1},\cdots, V_{r}], && \bVt = [V_{r+1},\cdots, V_{p_n}] \nonumber\\
E_1 &=& diag(\eta_1, \cdots, \eta_r) . \label{eq:Eonedef}
\end{eqnarray}
Applying \eqref{eq:partI_results} and the definitions above, we can write
\[
\Delta_n = \left[\begin{array}{cc}
(\Delta_n)_{1,1} & (\Delta_n)_{1,2}\\
(\Delta_n)_{2,1} & (\Delta_n)_{2\times 2}
\end{array}\right] ,
\]
where 
\begin{eqnarray*} \label{eq:partitioned_delta_2}
(\Delta_n)_{1,1} &=& L_1^{-\frac{1}{2}} \bigg( \bWo' \bVo E_1 \bVo'\bWo + \bWo'\bVt \bVt'\bWo \bigg) L_1^{-\frac{1}{2}} + o_p(1) \\
(\Delta_n)_{1,2} &=& (L_1^{-\frac{1}{2}} +\zeta)\big(\bWo' \bVo E_1 \bVo' \bWt + \bWo' \bVt \bVt'\bWt \big) \\
(\Delta_n)_{2,1} &=& (\Delta_n)_{1,2}'  \\
(\Delta_n)_{2,2} &=& \bWt' \bVo E_1 \bVo' \bWt + \bWt' \bVt  \bVt'\bWt.
\end{eqnarray*}
The terms $\Delta_{1,2}$, $\Delta_{2,1}$ and $\Delta_{2,2}$ are deterministic and simple.
\begin{eqnarray*}
(\Delta_n)_{1,2} &=& (L_1^{-\frac{1}{2}} + \zeta) \big(\bWt' \bVo E_1 \bVo' \bWt + \bWo' \bVt \bVt'\bWt \big) \hfill \quad 
\mbox{[since $\bWt$ are orthogonal to $\bVo$]} \\ 
&=& (L_1^{-\frac{1}{2}} + \zeta) \big( 0 + \bWo' \bVt \bVt'\bWt \big) \hfill \quad 
\mbox{[since $\bWt$ is in the span of $\bVt$]} \\
&=& (L_1^{-\frac{1}{2}} + \zeta)\bWo'\bWt = 0
\end{eqnarray*}
{ Perforce}  $(\Delta_n)_{2,1}=0$ also. We also have
\begin{eqnarray*}
(\Delta_n)_{2,2} &=& \bWt' \bVo E_1 \bVo' \bWt + \bWt' \bVt \bVt'\bWt \\
&=& 0 + \bWt' \bVt  \bVt'\bWt \hfill \quad 
\mbox{[since $\bWt$ is in the span of $\bVt$]}\\
&=& \bWt'\bWt = I_{p_n - 2r}
\end{eqnarray*}
Since the comparable $(2,1)$, $(1,2)$ and $(2,2)$
blocks of the asymptotic pivot are identical, we
have an isometry between Frobenius errors of the full $p_n \times p_n$ matrices 
and those in the upper left $(1,1)$ block:
\begin{equation} \label{eq:frobiso}
  \|\Delta_n - \oplus_{i=1}^r A(\ell_i,\eta_{i}; \gamma) \oplus I_{p_n - 2r}\|_F = \| (\Delta_n)_{1,1} - \oplus_{i=1}^r A(\ell_i,\eta_{i}; \gamma) \|_F .
\end{equation}

It remains to consider the upper left block $(\Delta_n)_{1,1}$. Because $VV' = \bVo \bVo + \bVt \bVt $
and $E_1 = diag(\eta_1,\dots, \eta_r)$:
\begin{eqnarray}
(\Delta_n)_{1,1} &=& L_1^{-\frac{1}{2}} \bigg( \bWo' \bVo E_1 \bVo'\bWo + \bWo'\bVt \bVt'\bWo \bigg) L_1^{-\frac{1}{2}} + o_p(1) \nonumber \\
&=&  L_1^{-\frac{1}{2}} \bigg( \bWo' \bVo (E_1 - I_{r}) \bVo'\bWo + \bWo'V  V'\bWo \bigg) L_1^{-\frac{1}{2}} + o_p(1) \nonumber \\
 &=& L_1^{-\frac{1}{2}} \bigg( \bWo' \bVo (E_1 - I_{r}) \bVo'\bWo + I_r \bigg) L_1^{-\frac{1}{2}} + o_p(1). \label{eq:delta_11_limit}
\end{eqnarray}
By Theorem \ref{lem:spiked}, we have  
\beq \label{eq:QRFrobConvergence}
    \bWo'\bVo\goto_{a.s.}  \Omega, \qquad n \goto \infty,
\eeq
say, with convergence along $2r \times r$ matrices in Frobenius norm.
Here $\Omega$ is a block-diagonal matrix:
\[
\Omega = \oplus_{i=1}^r \left[\begin{array}{c}
c(\ell_i) \\
s(\ell_i) 
\end{array}\right] . 
\] 

Using this block structure, one can see that 
\begin{eqnarray}
\bWo' \bVo (E_1 - I_{r}) \bVo'\bWo + I_r 
&\goto& \oplus_{i=1}^{2r} B(\ell_i,\eta_i) \label{eq:AtildeLim}
\end{eqnarray}
where $B(\ell,\eta)$ denotes the 2-by-2 matrix
 \[
 \left[\begin{array}{cc}
 s^2 + c^2 \eta & c \cdot s \cdot (\eta - 1) \\
 c \cdot s \cdot (\eta - 1) & c^2 + s^2 \eta
 \end{array}\right] , \qquad c = c(\ell), s = s(\ell) .
 \]
Therefore, the almost sure limit of \eqref{eq:delta_11_limit} obeys:
\begin{equation} \label{eq:LimDeltaOneOne}
 (\Delta_n)_{1,1} \goto\oplus_{i=1}^r A(\ell_i, \eta_i).
\end{equation}
Applying  (\ref{eq:LimDeltaOneOne}) to (\ref{eq:frobiso}) completes the argument.
\end{proof}

The results just given showed that,
in the special case $\eta_i=1$, $i >p$,
the pivot $\Delta^e( (\ell_i); (\eta_i), n,p)$
has an asymptotically block-diagonal representation:
\begin{equation} \label{eq:pointwise}
\Vert\Delta^a((\ell_i); \eta;n, p_n) - \Delta^e((\ell_i); \eta;n, p_n) \Vert_F \rightarrow 0 .
\end{equation}
The following  additional 
lemma strengthens this into a result uniform across bounded $\eta$. 

\begin{defn} \label{def:Wnr} By $\cWnr$ we denote the 2r+1-dimensional subspace 
spanned by the first $r$ columns of $U$, the first $r$ columns of $V$
and by one additional unit vector $w_{2r+1}$, say, which is chosen uniformly
at random from the $(p-2r-1)$ dimensional sphere 
situated in the
orthocomplement of the column span of the first $r$ columns
of $U$ and and first $r$ columns of $V$ in $R^p$.

For a matrix $M$ the symbol $M|\cWnr$ denotes the matrix $M$ restricted to
subspace $\cWnr$.
\end{defn}

{\it  Note: In earlier arguments we focused attention on a $2r$-dimensional
subspace, while below, we will have need for a $2r+1$-dimensional subspace.
The rationale is as follows. The additional dimension 
allows us in a sense to `compress' all other dimensions beyond 
those in the first $r$ columns of $U$ and $V$ into one single dimension. This allows
us to use $2r+1$-dimensional arguments below which might otherwise involve
growing dimensions. At the same time, certain bounds we develop show that
attention to this $2r+1$-dimensional subspace will be sufficient; while
a $2r$-dimensional subspace might fail to represent the eigenvalue $1$.}

\newcommand{\rnor}{ \eta^{1,r}}
\begin{lem}
\label{lem:PivotUniformConvergence}
For an $r$-vector $\eta$, let $\rnor$ denote the 
padding of $(\eta_i)_{i=1}^r$ with ones out to a $p$-vector.
Fix $\ell_0 > 1$. Let $\cL_r(\ell_0)$ denote the
collection of all spike configurations in  $[0,\ell_0]^r$. 
Almost surely and in probability:
\[
 \sup_{\eta \in \cL_r(\ell_0)} \Vert\Delta^a((\ell_i); \rnor ;n, p_n) - \Delta^e((\ell_i); \rnor ;n, p_n) \Vert_F \rightarrow 0, \quad n \goto \infty.
\]
Similarly:
\[
 \sup_{\eta \in \cL_r(\ell_0)} \Vert\Delta^a((\ell_i); \rnor;n, p_n)|\cWnr - \Delta^e((\ell_i); \rnor;n, p_n)|\cWnr \Vert_F \rightarrow 0.
\]
\end{lem}

\newcommand{\tworpone}{2r+1}
\begin{proof}
To clarify, the second display concerns the comparison of two $\tworpone$ by $\tworpone$ matrices,
while the first display concerns the comparison of two $p_n$ by $p_n$ matrices.
We  adopt a basis for $\bR^{p_n}$ and $\bR^r$, 
so that the smaller matrices are simply
the upper left block of the corresponding larger  matrices. 
Moreover, the larger matrices
being compared are identical outside the upper block (see the proof of Lemma \ref{lem:blocks}). 
Therefore, the second claim is sufficient to 
establish the first, since the matrices in question 
are identical outside the upper left $\tworpone$ by $\tworpone$ block.
Note that $(\ell_i)$ are fixed, and only the $(\eta_i)_{i=1}^r$ are allowed to vary. Define the function
\[
 F_n(\eta) = \Vert\Delta^a((\ell_i); \eta;n, p_n)|\cWnr - \Delta^e((\ell_i); \eta;n, p_n)|\cWnr \Vert_F .
\]
By  Lemma \ref{lem:blocks}, at each fixed $\eta$ the sequence $(F_n(\eta))$ tends to zero almost surely.
We will demonstrate convergence uniform in $\eta$.\\
{We saw above at (\ref{eq:AtildeLim}) that}
 the discrepancy
\[
G_n(\eta) = \| \hat{\Sigma}(\eta)|\cWnr -  \left( \oplus_{i=1}^{2r} B(\ell_i,\eta_i) \, \oplus w_{2r+1} w'_{2r+1}  \right) \|_F
\]
converges almost surely to zero,
for appropriate $2 \times 2$ blocks $B(\ell_i,\eta_i)$
and appropriate  realization-dependent vectors
$w_{2r+1}$. 
Since all eigenvalues of $\Sigma ^ {-\frac{1}{2}}$ are less than or equal to one, we have 
(using notation from the proof of Lemma \ref{lem:blocks}): 
\begin{eqnarray*}
F_n(\eta) &=& \Vert \Delta^a((\ell_i); \eta;n, p_n)|\cWnr - \Delta^e((\ell_i); \eta;n, p_n)|\cWnr \Vert_F \\
&=& \Vert (\Sigma ^ {-\frac{1}{2}}|\cWnr) \bigg( \hat{\Sigma}(\eta)|\cWnr -  \left( \oplus_{i=1}^{2r} B(\ell_i,\eta_i) \, \oplus w_{2r+1} w'_{2r+1}  \right) \bigg) (\Sigma ^ {-\frac{1}{2}}|\cWnr) \Vert_F \\
&\leq& \Vert \hat{\Sigma}(\eta)|\cWnr -  \left( \oplus_{i=1}^{2r} B(\ell_i,\eta_i) \, \oplus w_{2r+1} w'_{2r+1}  \right) \Vert_F \\
&=& G_n(\eta) .
\end{eqnarray*}

Now letting $\Omega_{i}^{(n)}$ denote the $i$'th column of {  matrix $\Omega^{(n)} = {\bf W}'_1 {\bf V}_1$
and $\Omega_{i}^{(\infty)}$ denote the $i$'th column of its
limit matrix $\Omega =  \lim_{n \goto \infty}{\bf W}'_1 {\bf V}_1$
}, we write:  
\begin{eqnarray*}
G_n &=& \| \Omega^{(n)} E_1 (\Omega^{(n)})' -   \Omega E_1 (\Omega)'  \|_F \\
& = & \| \sum_{i=1}^r  \eta_i [ \Omega_{i}^{(n)}  (\Omega_{i}^{(n)})' -   \Omega_{i}^{(\infty)}  (\Omega_{i}^{(\infty)})'] \|_F \\
& \leq & \sum_{i=1}^r  \eta_i  \|[ \Omega_{i}^{(n)}  (\Omega_{i}^{(n)})' -   \Omega_{i}^{(\infty)}  (\Omega_{i}^{(\infty)})'] \|_F  .
\end{eqnarray*}
{where $E_1 = diag(\eta_1,\dots,\eta_r)$  as earlier in (\ref{eq:Eonedef})}.
Since for unit vectors $x$, $y$,
 $\| xx' - yy'\|_F^2 = 2( 1 - (x'y)^2) = \|x - y\|_2^2$,
\begin{eqnarray*}
G_n &\leq&  \left( \max_i  \eta_i \right) \cdot \sum_{i=1}^r \|\Omega_{i}^{(n)}-\Omega_{i}^{(\infty)}\|_2 \\
    &\leq&   \ell_0 \cdot r  \cdot  \|\Omega^{(n)}-\Omega\|_F .
\end{eqnarray*}
This last expression used the hypothesis
that $\eta \in \cL_r(\ell_0)$, and gives us an upper bound that
does not depend on $\eta$. We know from (\ref{eq:QRFrobConvergence})
that $\|\Omega^{(n)}-\Omega\|_F \goto 0$ both almost surely and in probability, and so
\[
     \sup_{\eta \in \cL_r(\ell_0)} G_n(\eta) \goto 0, \qquad n \goto \infty,
\]
both almost surely and in probability. 
This shows that $F_n(\eta)$ uniformly converges to zero almost surely
and in probability.
\end{proof}

\begin{lem}
\label{lem:MutilationIsBanal}
For a $p$-vector $\eta$ obeying $Ave_{i>r} \eta_i =1$,
let $\eta^{1,r}$ denote the ones-mutilated vector
\[
  \eta^{1,r} = \left \{ \begin{array}{ll}
                            \eta_i & 1 \leq i \leq r\\
                             1       & r < i \leq p
                      \end{array}
               \right . .
\]
Fix $\ell_0 > 1$. Let $\EtaRNPf(\ell_0)$ denote the
collection of all $\eta \in \EtaRNPf \cap [(\ell_0^2+1)^{-1},\ell_0^2+1]^p$.
There is $C = C(\eps,\ell_0,\gamma) > 0$ so that under [SPIKE] and [PGA]:
\[
 \sup_{\eta \in \EtaRNPf(\ell_0)} P\{ \Vert\Delta^e((\ell_i); \eta )|\cWnr - \Delta^e((\ell_i); \eta^{1,r})|\cWnr \Vert_F > \eps \} \leq C \cdot n^{-3}.
\]
\end{lem}
\newcommand{\bUo}{{\bf U}_1}
\newcommand{\bVo}{{\bf V}_1}
\newcommand{\hSigma}{\hat{\Sigma}}
\begin{proof}
We have:
\begin{eqnarray*}
F_n(\eta) &=& \Vert \Delta^e((\ell_i); \eta)|\cWnr - \Delta^e((\ell_i); \eta^{1,r})|\cWnr \Vert_F \\
&=& { \Vert (\Sigma ^ {-\frac{1}{2}}|\cWnr) \bigg( \hat{\Sigma}(\eta)|\cWnr - 
\hat{\Sigma}(\eta^{1,r} )|\cWnr \bigg) (\Sigma ^ {-\frac{1}{2}}|\cWnr) \Vert_F} \\
&\leq& \Vert \hat{\Sigma}(\eta)|\cWnr - \hat{\Sigma}(\eta^{1,r} )|\cWnr \Vert_F \\
&\equiv& G_n(\eta) , \mbox{ say }.
\end{eqnarray*}
Construct now a so-called 
$X$-basis $[x_{1},\dots,x_{2r+1}]$ 
for the subspace $\cWnr$
by applying Gram-Schmidt to $[\bVo \bUo ]$ 
(followed by $w_{2r+1}$)
(rather than to the interlacing of $\bUo$ and $\bVo$ (followed by $w_{2r+1}$) as 
would be done for the
$W$-basis).  
Partition the basis of $2r+1$ vectors
into three blocks
$[x_{1},\dots, x_r]$,
$[x_{r+1},\dots,x_{2r}]$
and $x_{2r+1}$. Here again, by definition \ref{def:Wnr}, the
$x_{2r+1}$ vector is a direction chosen uniformly at random in the
orthocomplement of the column span of $X$.
Each of $\hat{\Sigma}(\eta)|\cWnr$ and
$\hat{\Sigma}(\eta^{1,r} )|\cWnr$ can be represented $2r+1$ by $2r+1$
partitioned matrix 
\[
\hSigma_n = \left[\begin{array}{ccc}
(\hSigma_n)_{1,1} & (\hSigma_n)_{1,2} & (\hSigma_n)_{1,3}\\
(\hSigma_n)_{2,1} & (\hSigma_n)_{2, 2} & (\hSigma_n)_{2,3}\\
(\hSigma_n)_{3,1} & (\hSigma_n)_{3, 2} & (\hSigma_n)_{3,3}\\
\end{array}\right] ,
\]
where the lower right block is $1 \times 1$ and
involves the projection $x'_{2r+1} \hSigma_n x_{2r+1}$,
and the other entries in the third column or row
are either column or row vectors.
Direct calculations show that, in either case, we have the invariance,
\[
 \hat{\Sigma} x_{i} = \eta_i \cdot x_i, \qquad i=1,\dots,r,
\]
so that, in the $X$-basis
\[
(\hSigma_n)_{1,1} = diag(\eta_{1},\eta_{2},\dots,\eta_{r}).
\]
Also, by invariance and Gram-Schmidt,
\[
x'_{i} \hSigma x_{j} =  \eta_i \cdot x'_{i} x_{j} = 0, \qquad 1 \leq i \leq r, \quad r < j \leq 2r+1,
\]
so $(\hSigma_n)_{1,2} = (\hSigma_n)_{2,1} = 0_{r \times r}$ in either case,
while also $(\hSigma_n)_{1,3} = (\hSigma_n)'_{3,1} = 0_{r \times 1}$
in both cases.
Meanwhile, for blocks $(\hSigma_n)_{2, 2}$ and $(\hSigma_n)_{3, 3}$ the two cases differ:
\[
(\hSigma_n(\eta^{1,r}))_{2, 2} = I_{r \times r}, \qquad  (\hSigma_n(\eta))_{2, 2} = I_{r \times r} + \zeta_n,
\]
where $\zeta_n$ is an $r \times r$ matrix-valued random variable, and
\[
(\hSigma_n(\eta^{1,r}))_{3, 3} = 1, \qquad  (\hSigma_n(\eta))_{3, 3} = 1 + \xi_n,
\]
where $\xi_n$ is a real-valued random variable. Finally, let $\rho_n $
denote the $r \times 1$ vector-valued random variable $\rho_n = (\hSigma_n(\eta))_{2,3}$.
(noting that $(\hSigma_n(\eta^{1,r}))_{2,3} = 0$).

We conclude that
\[
   G_n(\eta) = (\| \zeta_n \|_F^2 + 2 \| \rho_n \|_2^2 + \xi_n^2)^{1/2}.
\]
We give the argument in the case $r=1$, such that $\zeta_n$ and $\rho_n$ are
real-valued. Letting $x_2$ denote the second column of the basis matrix $X$,
and taking into account that the first $r$ entries of
$\eta$ are identical to those of $\eta^{1,r}$, we have
\begin{eqnarray*}
   \zeta_n &=&  x'_2 \cdot (\hat{\Sigma}(\eta)|\cWnr - \hat{\Sigma}(\eta^{1,r}) |\cWnr  ) \cdot  x_2 \\
   &=& \sum_{r+1}^p (\eta_i - 1) u_i^2,
\end{eqnarray*}
where $Ave_{i=r+1}^p \eta_i = 1$ and the $u_i$ give the coordinate representation
of the vector $x_2$ in the eigenvector or $V$-basis. The random vector 
$U = (0,u_2,\dots, u_p)$
has norm $\|U\|_2=1$, and is uniformly distributed
on the spherical equator $u_1=0$. We invoke Lemma \ref{lem:TailBoundForWeightedChiSquare}; 
it gives us, for large $n$, this bound:
\[
   \sup_{\eta \in \EtaRNP(\ell_0)}   
      P \{ | \zeta_n |  > 
     \eps\} \leq C(\eps, \ell_0) \cdot p_n^{-3}.
\]
Lemma  \ref{lem:TailBoundForWeightedChiSquare}
also implies
\[
   \sup_{\eta \in \EtaRNP(\ell_0)}   
      P \{ | \xi_n |  > 
     \eps\} \leq C(\eps, \ell_0) \cdot p_n^{-3},
\]
with possibly a different $C$ than in the previous display.
Finally, the same Lemma also implies:
\[
   \sup_{\eta \in \EtaRNP(\ell_0)}   
      P \{ | \rho_n |  > 
     \eps\} \leq C(\eps, \ell_0) \cdot p_n^{-3}.
\]
Combining the last three displays 
completes the proof for the case $r = 1$.

The proof for general $r > 1$ goes similarly, but
involves an analog for Lemma \ref{lem:TailBoundForWeightedChiSquare} \--
involving matrices rather than scalars. We omit  details.
\end{proof}
\newcommand{\cE}{{\cal E}}
\begin{lem} \label{lem:TailBoundForWeightedChiSquare}
Let $U = (0_{1 \times r}, u_{r+1},\dots,u_{p})$
be a random $p$-vector uniformly distributed on the 
$p-r$ dimensional
spherical equator $ \cE_{r,p} = S^{p-1} \cap \{ u : u_1 = \dots = u_r = 0 \}$. 
Then, for $\eps > 0$ and for all sufficiently large $n,p_n$,
\[
   \sup_{\eta \in \EtaRNP(\ell_0)}   
      P \{ |\sum_{i=r+1}^{p} (\eta_i -1) u_i^2 |  > 
     \eps\} \leq C(\eps, \ell_0) \cdot p_n^{-3}.
\]

Let $V= (0_{1 \times r}, v_{r+1},\dots,v_{p})$ be a random $p$-vector uniformly distributed on
$\cE_{r,p} \cap U^\perp$.
Then, for $\eps > 0$ and for all sufficiently large $n,p_n$.
\[
   \sup_{\eta \in \EtaRNP(\ell_0)}   
      P \{ |\sum_{i=r+1}^{p} (\eta_i -1) v_i^2 |  > 
     \eps\} \leq C(\eps, \ell_0) \cdot p_n^{-3},
\]
and
\[
   \sup_{\eta \in \EtaRNP(\ell_0)}   
      P \{ |\sum_{i=r+1}^{p} \eta_i \; u_i v_i |  > 
     \eps\} \leq C(\eps, \ell_0) \cdot p_n^{-3}.
\]

\end{lem}
\begin{proof}

We give the argument for the first displayed inequality only;
similar arguments yield the other inequalities.

Note that $E u_i^2 = \frac{1}{p-r}$, $i=r+1,\dots,p$; hence
for $\eta \in \EtaRNP$,
$E  \sum_{i=r+1}^{p} (\eta_i-1)  u_i^2 = 0$.
The result in question is a uniform bound on the probability of
large fluctuations of a family of zero-mean random variables.

On an appropriate probability space, 
$U =_D Z/R$ where $Z = (0_{1\times r}, Z_{r+1},\dots,Z_p)$ and
the $Z_i$ are iid $N(0,1)$, while $R = \| Z \|_2$. 
By  Laurent and Massart,  \cite{laurent2000} p. 1325, 
we have the exponential bounds
on $Y \equiv   \|Z \|_2^2/m$, where $m = p-r$:
\begin{eqnarray*}
P \{ Y > 1+  2\sqrt{ t/m} + 2 t/m \} &\leq& \exp(-t ), \\
P \{ Y < 1-  2 \sqrt{t/m} \}          &\leq& \exp(-t ).
\end{eqnarray*}
Consequently, we have excellent control on $P \{ |u_i - Z_i/\sqrt{p-r}| > \eps \}$ and similar
quantities. We therefore can work with Gaussian variables
$Z_i/\sqrt{p-r}$ rather than 
uniform spherical coordinates $u_i$,
and it will be sufficient to show:
\beq \label{eq:neededbound}
     P \{ |  \frac{1}{p-r} \sum_{i=r+1}^{p} (\eta_i-1) Z_i^2 | > \eps  \} \leq C(\eps,\ell_0) / (p-r)^3.
\eeq
Let $Y= \sum_{j=1}^m  (\eta_j-1) Z_j^2$
denote a random sum of the type referred 
to in (\ref{eq:neededbound}).
We can obtain the moment bound:
\[
    E \left| Y \right|^6  \leq  { C \cdot (\ell_0+1)^6\cdot  m^3}.
\]
Markov's inequality $\eps^6 \cdot P\{\vert \frac{1}{m} Y \vert > \eps \} \leq E (\frac{Y}{m})^6$, combined with this
moment bound, gives (\ref{eq:neededbound}).

This bound, and the constant $C$,
are obtainable from Lemma 
\ref{lem:cumulantbounds} below
by setting
$c_j = (\eta_j-1)$ and $\zeta_j = Z_j^2$,
and recalling that for $\eta \in \EtaRNP(\ell_0)$, 
$\eta_i \leq M = \ell_0^2+1$.
\end{proof}

No doubt very elegant proofs could be based on 
Paul L\'evy's concentration of measure
for the sphere.

\begin{lem} \label{lem:cumulantbounds}
Suppose $\xi = \sum_{j=1}^m c_j \zeta_j $, 
where $\zeta_j$ are iid,
and with the first six cumulants finite.
Suppose also that $E \xi = 0$.
The sixth moment of $\xi$ obeys
\[
   E \xi^6 \leq C \cdot \|(c_j)\|_\infty^6  \cdot m^3 ,
\]
for some $C$ which can be made explicit in terms of the 
cumulants $\kappa_\ell(\zeta_1)$, $\ell=1,\dots,6$.
\end{lem}
\begin{proof}
The additivity of cumulants tells us that
\[
k_\ell(\xi) = \sum_j c_j^\ell k_\ell(\zeta_j).
\]
For $\zeta_j$ iid we have 
$k_\ell(\xi) = k_\ell(\zeta_1) \sum c_j^\ell$,
and so $|k_\ell(\xi)| \leq |k_\ell(\zeta_1)| \cdot m \cdot \|(c_j)\|_\infty^\ell$. 
When the random variable $\xi$ is centered,
$E \xi = k_1(\xi)=0$, its sixth
moment 
\begin{equation}
\label{eq:MomentFromCumulant}
E \xi^6  = k_6(\xi) + 15 k_4(\xi) k_2(\xi) 
     + 10 k_3^2(\xi) + 15 k_2^3(\xi).
\end{equation}
Suppose that $\|(c_j)\|_\infty \leq M$.
Of the four terms on the RHS on (\ref{eq:MomentFromCumulant}), the strongest dependence on $m$ is contributed by the term $k_2^3 = O(m^3)$,
and the other terms are $O(m^2)$ or smaller. 
\end{proof}

\subsection{Proofs for Section \ref{sec:StudyAsyPiv}}

\begin{lem} \label{lem:boundsOnOptimalEta}
Assume [SPIKE]. 
Let $\Delta(\eta) = \Sigma^{-\frac{1}{2}}\hat{\Sigma}(\eta)\Sigma^{-\frac{1}{2}}$,
where $\eta \in \EtaRNPf$.
Consider the function $K:\EtaRNPf \rightarrow \R^+$ defined so:
\begin{equation}
	K(\eta) \equiv \kappa(\Delta(\eta)).
\end{equation}
Consider the problem of minimizing $K(\cdot)$  over the domain $\EtaRNPf$. 
There exists a solution  $\eta^* \equiv (\eta_i^*)_{i=1}^{p}$ for this minimization problem, with the following properties:
\begin{enumerate}
\item $Ave_{i>r} \eta^*_i = 1$,
\item $\max_i \eta^*_i \leq \ell_1^2 + 1$,
\item $\min_i \eta^*_i \geq (\ell_1^2 + 1)^{-1}$.
\end{enumerate}
\end{lem}

\begin{proof}
Let $ K^* \equiv \inf_{\eta \in \EtaRNPf} K(\eta)$. 
Since $K(1,\dots, 1) = \ell_1$, $K^* \leq \ell_1$. \\
Let $\eta^{(j)} \equiv (\eta^{(j)}_i)_{i=1}^p$ be a sequence 
in $\EtaRNPf$ 
asymptotically achieving the optimal value:
$\lim_{j\rightarrow \infty} K(\eta^{(j)}) = K^*$. 

For an SPD matrix $A$, denote its smallest eigenvalue by $\lambda_{min}(A)$ 
and its largest eigenvalue by $\lambda_{max}(A)$. For a pair $A$, $B$ of conformable SPD matrices,
\[
 \lambda_{max}(ABA) \geq  \lambda_{min}(A) \lambda_{max}(B) \lambda_{min}(A), \quad 
 \lambda_{min}(ABA) \leq  \lambda_{max}(A) \lambda_{min}(B) \lambda_{max}(A).
\]
Note that $\Delta(\eta) = A\cdot B \cdot A$
where $A = \Sigma^{-1/2}$ and $B = \hat{\Sigma}(\eta)$,
and also note that 
$\lambda_{max}(\hat{\Sigma}(\eta)) = \| \eta \|_{\infty}$, while 
$\lambda_{max}^2(\Sigma^{-1/2})  = 1$, while  $\lambda_{min}(\hat{\Sigma}(\eta)) = \min_{i=1}^p \eta_i$ and
$\lambda_{min}^2(\Sigma^{-1/2}) = \frac{1}{\ell_1}$.
We conclude that 
\begin{eqnarray} \label{eq:upperboundOnMinEig}
	\lambda_{min}(\Delta(\eta)) &\leq& \lambda_{max}(\Sigma^{-\frac{1}{2}})\lambda_{min}(\hat{\Sigma}(\eta))
                                       \lambda_{max}(\Sigma^{-\frac{1}{2}}) \\ 
     &=& 1  \cdot \min_{i=1}^p \eta_i \cdot 1 \nonumber  \\
    &=& \min_{i=1}^p\eta_i  . \nonumber
\end{eqnarray}
\begin{eqnarray} \label{eq:lowerboundOnMaxEig}
	\lambda_{max}(\Delta(\eta)) &\geq& \lambda_{min}(\Sigma^{-\frac{1}{2}})\lambda_{max}(\hat{\Sigma}(\eta))
     \lambda_{min}(\Sigma^{-\frac{1}{2}})  \\ 
     &=& \frac{1}{\ell_1^{1/2}}  \cdot \max_i \eta_i \cdot \frac{1}{\ell_1^{1/2}}  
    = \frac{\max_i \eta_i}{\ell_1} .\nonumber
\end{eqnarray}
Hence
\[
   K(\eta) \geq \frac{\max_i \eta_i}{\ell_1 \cdot (\min_{i=1}^p \eta_i)} .
\]
By the convergence $\lim_{j\rightarrow \infty}K(\eta^{(j)}) = K^*$ 
and the upper bound $K^* \leq \ell_1$, 
 $\exists N_0$ s.t. $\forall j >N_0$, $K(\eta^{(j)}) < \ell_1 + \frac{1}{\ell_1}$. 
We conclude that eventually, 
\beq \label{eq:condBND}
\frac{\max_i \eta^{(j)}_i}{  \min_{i=1}^p \eta^{(j)}_i} \leq \ell_1^2 + 1 , \qquad j > N_0.
\eeq

By the hypothesis, the points $\eta^{(j)}$ in our sequence
obey $Ave_{i>r} \eta^{(j)}_i = 1$.
Hence $ \min_{i=1}^p \eta_i^{(j)} \leq 1 \le \max_{i=1}^p \eta_i^{(j)}$ for every $j$.  
Rewriting (\ref{eq:condBND}),
\[
\frac{1}{\ell_1^2+1} \leq  \frac{\max_i \eta^{(j)}_i}{\ell_1^2 + 1} \leq  \min_{i=1}^p \eta^{(j)}_i \leq 1 , \qquad j > N_0.
\]
This shows us that, eventually for all large $j$, 
both  $\max_i \eta^{(j)}_i \leq \ell_1^2+1$
and  $\min_i \eta^{(j)}_i \geq (\ell_1^2+1)^{-1}$.
Hence the subsequence $(\eta^{(j)}: j > N_0)$ lies in the 
compact hypercube $[(\ell_1^2 + 1)^{-1}, \ell_1^2 + 1]^p$.

Since $K( \cdot )$ is continuous on this domain,
 $K( \cdot )$ attains its infimum on
the compact $ \EtaRNPf \cap [(\ell_1^2 + 1)^{-1}, \ell_1^2 + 1]^p$, by a $p$-vector $\eta^*$
with the two claimed properties. 
\end{proof}

\subsubsection*{Proof of Lemma \ref{lem:AsyOptProb}}
\begin{proof}
Consider the case $p>2r$. 
Then, by the block diagonal structure of the asymptotic pivot, the set of eigenvalues of $\Delta^a((\ell_i);(\eta_i);n,p,\gamma)$ is equal to:
\[
\{1\}\cup (\cup_{i=1}^r\{\nu_-(\ell_i, \eta_i, \gamma), \nu_+(\ell_i, \eta_i, \gamma)\}),
\]
where $\nu_-(\ell_i, \eta_i, \gamma)$ and $\nu_+(\ell_i, \eta_i, \gamma)$ represent the smallest and largest eigenvalue of the $A(\ell_i, \eta_i; \gamma)$.
Therefore, for each given set of $\eta_i$'s, $\kappa(\Delta^a((\ell_i);(\eta_i);n,p,\gamma))$ is equal to 
\begin{equation} \label{eq:asymptLoss}
K_r((\ell_i)_{i=1}^r,(\eta_i)_{i=1}^r; \gamma) \equiv \frac{\max(1, \max_i \nu_+(\ell_i, \eta_i, \gamma))}{\min( 1, \min_i \nu_-(\ell_i, \eta_i, \gamma))},
\end{equation}
which is only a function of $\gamma, (\eta_i)$, and $(\ell_i)$ and not a function of $p>2r$. Therefore, for all $p$, (\ref{eq:asymptLoss}) has the same optimal value. 
A compactness argument as in Lemma \ref{lem:boundsOnOptimalEta} 
can be used to show that the infimum is achieved and at least one $\eta^*$ exists. Point 1. follows.

{ Point 2 flows from sublemma \ref{sublemma:lbndone},
showing we are entitled to optimize over the constraint class $\eta_i \geq 1$, $i=1,\dots,r$, without any loss of generality.}
\end{proof}

{
\begin{lem} \label{sublemma:lbndone}
\begin{equation} \label{eq:equality}
  \min_{\eta_i \geq 0} K_r((\ell_i),(\eta_i)) 
  =  \min_{\eta_i \geq 1} K_r((\ell_i),(\eta_i)).
\end{equation}
\end{lem}

\newcommand{\tpeta}{\eta^{1,+}}
\begin{proof}
To prove this, it will be convenient to define 
\[
  K_{+,r}(\eta) \equiv \max(1,  \max_{i=1}^r \nu_+(A(\ell_i,\eta_i))),
\]
and
\[
K_{-,r}(\eta) \equiv \min(1, \min_{i=1}^r \nu_-(A(\ell_i,\eta_i))).
\]
Thus, abbreviating $K_r(\eta) \equiv K_r( (\ell_i),(\eta_i))$,
\[
  K_r(\eta) = \frac{K_{+,r}(\eta)}{K_{-,r}(\eta)}.
\]
For an $r$-vector $\eta$, define $\tpeta  = ( \max(1,\eta_i))_{i=1}^r$.

The next two lemmas 
(\ref{lem:KrMinusLemma})-(\ref{lem:KrPlusLemma}) 
show that
\begin{eqnarray*}
   K_r(\eta) &\equiv& \frac{K_{+,r}(\eta)}{K_{-,r}(\eta)} \\
       &=& \frac{K_{+,r}(\tpeta)}{K_{-,r}(\eta)}  \qquad \hfill [\mbox{Lemma }\ref{lem:KrPlusLemma}]\\
       &\geq&  \frac{K_{+,r}(\tpeta)}{K_{-,r}(\tpeta)}  \qquad \hfill [\mbox{Lemma } \ref{lem:KrMinusLemma}] \\
       &=& K_r(\tpeta).
\end{eqnarray*}
It follows that
\begin{eqnarray*}
  \min_{\eta_i \geq 0} K_r((\ell_i),(\eta_i)) &\geq&   \min_{\eta_i \geq 0} K_r((\ell_i),(\tpeta_i)) \\
 &=&   \min_{\eta_i \geq 1} K_r((\ell_i),(\eta_i)).
\end{eqnarray*}
On the other hand because the condition $\min_{\eta_i \geq 0}$ covers more cases than
$\min_{\eta_i \geq 1}$, we also have
\[
  \min_{\eta_i \geq 0} K_r((\ell_i),(\eta_i)) \leq  \min_{\eta_i \geq 1} K_r((\ell_i),(\eta_i)).
\]
Hence the two sides of (\ref{eq:equality}) are equal.
\end{proof}

\begin{lem}  \label{lem:KrMinusLemma}
For all $\eta \in \bR_+^r$,
\[
   K_{-,r}(\eta) \leq K_{-,r}(\tpeta) .
\]
\end{lem}

\begin{proof}
The proof of Lemma \ref{lem:increasingwrtEta} below
shows that, for each $0 \leq \eta  \leq 1$ 
$A(\ell,1) = A(\ell,\eta) + (1-\eta) \cdot B$ where $B$ is nonnegative definite.
Eigenvalues  of symmetric matrices are nondecreasing under the NND matrix ordering. Hence,
\[
 \nu_-(A(\ell,\eta)) \leq  \nu_-(A(\ell,1) ), \qquad \mbox{ if }  0  < \eta < 1.
\]
Applying this specifically with choices $(\ell_i)$, $(\eta_i)$, we see that
\[
   \nu_-(A(\ell_i,\eta_i))  \leq \nu_-(A(\ell_i,\tpeta_i)) , \qquad i=1,\dots, r,
\]
and so
\[
   K_{-,r}(\eta) = \min(1, \min_{i=1}^r \nu_-(A(\ell_i,\eta_i))) \leq \min(1, \min_{i=1}^r \nu_-(A(\ell_i,\tpeta_i))) =  K_{-,r}(\tpeta).
\]
\end{proof}

\begin{lem} \label{lem:KrPlusLemma}
For all $\eta \in \bR_+^r$,
\[
   K_{+,r}(\tpeta) = K_{+,r}(\eta) .
\]
\end{lem}

\begin{proof}
We first remark that {\it for every $\ell \geq 1$, 
$\eta=1$ marks the boundary separating $\nu_+(\ell,\cdot)  \leq 1 $ from $\nu_+ (\ell,\cdot) \geq 1$.}
To see this, we look ahead to  Lemma 2.6.1, which shows that
\[
  \nu_+(\ell,\eta) =  (T + \sqrt{T^2-4D})/2 ,
\]
where $D = D(\ell,\eta)$ and $T = T(\ell,\eta)$ as in Lemma 2.6.1.
In the special case $\eta=1$,
$T = T(\ell,1) = 1 + 1/\ell$ and $D = D(\ell,1) = 1/\ell$.
We get
\[
  \nu_+(\ell,1) = ((1+1/\ell) + \sqrt{(1+1/\ell)^2 - 4/\ell})/2  = 1.
\]
Again looking ahead to the
proof of Lemma \ref{lem:increasingwrtEta} below,  
for each $\eta \geq 1$ 
$A(\ell,\eta) = A(\ell, 1) + (\eta-1) \cdot B$ where $B$ is nonnegative definite.
Eigenvalues  of symmetric matrices are nondecreasing under the
NND matrix ordering.
Hence for $\eta \geq 1$, 
\[
\nu_+(A(\ell,\tpeta)) = \nu_+(A(\ell,\tpeta)) \geq \nu_+(A(\ell,1)) = 1.
\]
Similarly, for $ 0 < \eta < 1$, arguing as in the previous lemma,
$A(\ell,\eta) = A(\ell, \eta) + (1-\eta) \cdot B$ where $B$ is nonnegative definite,
and so
\[
\nu_+(A(\ell,\eta)) \leq \nu_+(A(\ell,\tpeta)) = \nu_+(A(\ell,1)) = 1.
\]
Hence for $\eta \geq 0$,
\[
\nu_+(A(\ell,\tpeta)) \geq 1.
\]
It follows that 
\begin{equation} \label{eq;maxgeqone}
1 \leq \max_{i=1}^r \nu_+(A(\ell_i,\tpeta_i))),
\end{equation}
and that
\[
\max_{i=1}^r \nu_+(A(\ell_i,\eta_i))) \leq \max_{i=1}^r \nu_+(A(\ell_i,\tpeta_i))).
\]
We have
\begin{eqnarray*}
K_{+,r}(\eta) &=&  \max(1,  \max_{i=1}^r \nu_+(A(\ell_i,\eta_i)))\\
&=&  \max \left (1,  \max \{  \nu_+(A(\ell_i,\eta_i)): \nu_+(A(\ell_i,\eta_i)) \geq 1 \} \right ) \\
&=&  \max \left (1,  \max \{  \nu_+(A(\ell_i,\eta_i)): \eta_i  \geq 1 \} \right ) \\
&=&  \max \left(1,  \max \{  \nu_+(A(\ell_i,\eta_i)): \eta_i = \tpeta_i \} \right ) \\
&=&  \max \left(1,  \max \{  \nu_+(A(\ell_i,\tpeta_i)):  \eta_i = \tpeta_i \} \right ) \\
&=&  \max \left(1,  \max \{  \nu_+(A(\ell_i,\tpeta_i)):   i=1,\dots, r \} \right ) \\
&=& \max_{i=1}^r \nu_+(A(\ell_i,\tpeta_i))) \qquad [\mbox{ by } (\ref{eq;maxgeqone})] \\
&=& K_{+,r}(\tpeta) .
\end{eqnarray*}
\end{proof}
}

\subsubsection{Proofs for Lemma \ref{lem:AsyShrinkOpt}}

\begin{lem} ({\bf Local Lipschitz character of condition number.}) \label{lem:LocalLipschitzCond}
Fix $M > 1$ and let $\Delta$ denote a symmetric 
matrix with eigenvalues lying between $1/M$ and $M$.
There is a constant $C(M)$ so that, whenever $\| \Delta_n - \Delta\|_F < \frac{1}{2M}$,
\begin{equation} \label{eq:LocLipCondNum}
   | \kappa(\Delta_n) - \kappa(\Delta) | \leq C(M) \| \Delta_n - \Delta\|_F;
\end{equation}
for example we may take $C(M) = 4M^3$.
\end{lem}
%
%
\begin{proof}
We state without proof immediately below two elementary sublemmata. 
These lemmas define locally Lipschitz functions $K$ and $\Lambda$ such that
$\kappa = K \circ \Lambda $; combining them
gives the desired result, (\ref{eq:LocLipCondNum}).
\end{proof}
\begin{lem} \label{lem:sub_lemma_local_lipschitz}
Consider the function $K: \bR^+ \times \bR^+ \mapsto \bR^+$ defined by $K(u,v) = u/v$.
At each point $(u,v)$ where $v > 0$ this
function is locally Lipschitz with respect to the $\ell^\infty_2$-norm $|(u,v)|_{\infty} = max(|u|,|v|)$.
Indeed, fix $M > 0$ and suppose $1/M \leq v_0 \leq u_0 \leq M$. Then if 
$|(u,v) - (u_0,v_0)|_{\infty} \leq \epsilon$
\[
     |K(u,v) - K(u_0,v_0)| \leq C(M) \cdot \epsilon,
\]
provided $\epsilon < 1/(2M)$, where $C(M) = 4M^3$.
\end{lem}
\begin{lem}
For an SPD matrix $A$, denote its smallest eigenvalue by $\lambda_{min}(A)$ 
and its largest eigenvalue by $\lambda_{max}(A)$. Let $\Lambda : SPD(p) \mapsto \bR^+ \times \bR^+ $ denote the mapping
$A \mapsto (\lambda_{max}(A),\lambda_{min}(A))$.
This mapping is Lipschitz with respect to Operator and Frobenius norms, 
with Lipschitz constant 1:
\[
|\Lambda(A) - \Lambda(B)|_{\infty} \leq \|A-B\|_{op} \leq  \|A-B\|_{F}.
\]
\end{lem}

We will use a well-known result on eigenvalues of submatrices; 
compare Theorem 4.3.28 of \cite{horn2012matrix}.
\begin{lem} \label{lem:restriction_reduces_extremes}
Let $B$ be a symmetric real matrix of the block form
\[
B = \left[\begin{array}{cc}
B_{1,1} & B_{1,2}\\
B_{2,1} & B_{2,2} 
\end{array}\right] ,
\]
where $B \in \R^{p\times p}$ and $B_{1,1} \in \R^{r \times r}$. Then
\[
\lambda_{max}(B)\geq \lambda_{max}(B_{1,1}), \qquad \lambda_{min}(B)\leq \lambda_{min}(B_{1,1}).
\]
\end{lem}

Armed with this, we can prove:

\begin{lem} \label{lem:LocalLipschitzRestriction}
Let $\cWnr$ denote the $2r+1$-dimensional subspace
of Definition \ref{def:Wnr}, and recall the convention that
for a matrix $\Delta$, $\kappa(\Delta|\cWnr)$ denotes the 
condition number of the restriction of $\Delta$ to the subspace $\cWnr$.
\bitem 
\item 
{ Restriction  to $\cWnr$
can at worst reduce the condition number of the asymptotic pivot:
\[
     \kappa( \DelAsy  | \cWnr ) \leq  \kappa(\DelAsy ).
\]
If all $\eta_i \geq 1$, $i=1,\dots,r$,} restriction
does not change the condition number of  the asymptotic pivot:
\[
      \kappa( \DelAsy | \cWnr ) =   \kappa(\DelAsy),
\]
whenever $p > 2r$. 
\item Restriction
can only reduce the condition number of the empirical pivot:
\[
     \kappa( \DelEmp  | \cWnr ) \leq  \kappa(\DelEmp ).
\]
\item Under the assumptions of 
Lemma \ref{lem:LocalLipschitzCond}, restriction preserves  Lipschitz character. Let $Q_{n,r}$ denote 
orthogonal projection onto $\cWnr$ 
and suppose that the empirical pivot is based on
a factor covariance estimate $\hat{\Sigma}(\eta)$ with at most $r$ factors, i.e. $\eta_{i,p}=1$, $r < i \leq p$.
Then for sufficiently small values of $\| Q_{n,r} (\DelEmp - \DelAsy) Q_{n,r} \|_F$ we have
\[
 | \kappa(\DelEmp | \cWnr) - \kappa(\DelAsy |\cWnr ) | \leq C(M) \| Q_{n,r} (\DelEmp - \DelAsy) Q_{n,r} \|_F ,
\]
where $C(M) = 4M^3$ as in Lemma \ref{lem:LocalLipschitzCond}.
\eitem
\end{lem}
\begin{proof}
Assume $p>2r$ and $\eta_i\geq 1$. By definition, 
\[
\DelAsy((\ell_i); (\eta_i); n,p,\gamma) = \oplus_{i=1}^r A(\ell_i,\eta_{i}; \gamma)  \oplus I_{p-2r}
\]
and, with $w_{2r+1}$ denoting the $2r+1$-th element 
in the construction of the orthobasis $W$,
\[
 \DelAsy  | \cWnr = \oplus_{i=1}^r A(\ell_i,\eta_{i}; \gamma) \;  \oplus  \; w_{2r+1} w'_{2r+1}.
\]
Therefore, the set of eigenvalues  
of either $\DelAsy$ or $\DelAsy  | \cWnr$ is: 
\[
(\cup_{i=1}^r\{\nu_-(\ell_i, \eta_i; \gamma)\})\cup(\cup_{i=1}^r\{\nu_+(\ell_i, \eta_i; \gamma)\})\cup\{1\}.
\]
It follows that
\[
\kappa(\DelAsy) = \kappa(\DelAsy | \cWnr) 
= K_r((\ell_i)_{i=1}^r,(\eta_i)_{i=1}^r ; \gamma)
\]
where $K_r$ was defined in Lemma \ref{sublemma:lbndone} above.
The independence of the RHS from $p > 2r$ 
proves our first assertion.


To prove our second claim, observe that by construction of $\cWnr$, $\DelEmp | \cWnr$ is the upper-left block of $\DelEmp$. 
Our second claim follows from Lemma \ref{lem:restriction_reduces_extremes}.

For our third assertion, observe that
\[
\| Q_{n,r} (\DelEmp - \DelAsy) Q_{n,r} \|_F = \|\DelEmp|\cWnr - \DelAsy|\cWnr\|_F.
\]
Now apply Lemma \ref{lem:LocalLipschitzCond}. 
\end{proof}

\subsubsection*{Proof of Theorem \ref{lem:AsyShrinkOpt}}

\newcommand{\etae}{{ \eta^e}}
\newcommand{\etaer}{{ \eta^{e,r}}}
\renewcommand{\etaa}{{ \eta^a}}
\newcommand{\etaar}{{ \eta^{a,r}}}

Let $\etae$ denote one of the minimizers 
of $\eta \mapsto \kappa(\Delta^e((\ell_i); (\eta_i )_{i=1}^{p_n};n, p_n))$ 
guaranteed by Lemma \ref{lem:boundsOnOptimalEta}.
Then by definition:  
\begin{eqnarray*}
val(K^e((\ell_i), n, p_n)) &=& \kappa(\Delta^e((\ell_i); \etae ))\\
&\geq& \kappa(\Delta^e((\ell_i); \etae )|\cWnr) \qquad 
  \hfill \mbox{[by Lemma \ref{lem:LocalLipschitzRestriction} part 2]}.
\end{eqnarray*}
{ Let $\etaer$ denote the ones-mutilated sequence 
$\etaer \equiv (\etae_1,\dots, \etae_r,1,\dots, 1)$.}
Lemma \ref{lem:boundsOnOptimalEta} gives the boundedness
 $(\ell_1^2 +1)^{-1} \leq \etae_i \leq \ell_1^2 +1$, and   
Lemma \ref{lem:PivotUniformConvergence} yields \-- uniformly across all
$\etae$ obeying these bounds \-- the convergence
\[
\| \Delta^e((\ell_i); \etaer ;n, p_n)|\cWnr  - \Delta^a((\ell_i); \etaer ;n, p_n)|\cWnr \|_F \xrightarrow{a.s} 0.
\]
{Lemma \ref{lem:MutilationIsBanal} yields
the convergence
\[
P\left \{ \| \Delta^e((\ell_i); \etae ;n, p_n)|\cWnr  - \Delta^e((\ell_i); \etaer ;n, p_n)|\cWnr \|_F  > \eps \right \} \leq C(\eps) n^{-3}.
\]
}
By the Lipschitz character of the condition number (of a restriction),
we have for small enough $\eps > 0$, and large enough $min(n,p_n)$,
\[
P \left \{ | \kappa(\Delta^e((\ell_i); \etae ;n, p_n)|\cWnr) -  \kappa(\Delta^a((\ell_i); \etaer ;n, p_n)|\cWnr)  \geq \eps \right\} \leq C'(\eps) n^{-3}.   
\]
\newcommand{\etaarpn}{\eta_{\gamma,p}^{a,r}}
Let $\etaarpn  \equiv ((\eta^a_\gamma)_i: 1 \leq i \leq p)$ 
denote an optimizing configuration 
of 
\[
\eta \mapsto \kappa(\Delta^a((\ell_i); (\eta_i)_{i=1}^p ;n, p, \gamma)),
\]
under the constraints of $(K^a)$
so that $val(K^a) = \kappa(\Delta^a((\ell_i);\etaarpn ;n, p, \gamma))$.
By the convention of the optimization
problem $(K^a)$ this solution has the form of a ones-padded sequence 
$\etaarpn \equiv ((\eta^a_\gamma)_1,\dots, (\eta^a_\gamma)_r,1,\dots, 1)$.
Moreover, by Lemma \ref{lem:AsyOptProb}, the optimized
variables $(\eta^a_\gamma)_1,\dots, (\eta^a_\gamma)_r$ embedded
in slots $1, \dots, r$  of $\etaarpn$ can be chosen to have the
{\it same values} across every instance of $\etaarpn$ where $p > 2r$ and $\gamma$ is the same.

Applying notation from the proof of Lemma \ref{lem:AsyOptProb}, as well as the conclusion of Lemma \ref{sublemma:lbndone}, 
we have
\begin{eqnarray*}
\kappa(\Delta^a((\ell_i); \etaer )|\cWnr) &=& K_r((\ell_i)_{i=1}^r, (\eta^{e}_i)_{i=1}^r; \gamma) \\
&\geq&  \min_{\eta_i \geq 0} K_r((\ell_i),(\eta_i); \gamma) \\
&=&  \min_{\eta_i \geq 1} K_r((\ell_i),(\eta_i); \gamma) \qquad [\mbox{Lemma } \ref{sublemma:lbndone}] \\
 &=&     K_r((\ell_i), ((\eta_\gamma^a)_i); \gamma) \\
 &=& val(K^a((\ell_i))) .
\end{eqnarray*}

We conclude that the event
\[
\big \{ \;  |\kappa(\Delta^e((\ell_i); \etae ;n, p_n)|\cWnr) -  \kappa(\Delta^a((\ell_i); \etaer ;n, p_n)|\cWnr)  | \leq \eps  \; \big \}
\]
implies
\[
\big \{ \; val(K^e((\ell_i), n, p_n)) > val(K^a((\ell_i);n, p_n)) - \epsilon  \; \big \} , 
\]
and hence that for sufficiently large $\min(n,p_n)$,
\[
  P \left \{  val(K^e((\ell_i), n, p_n)) < val(K^a((\ell_i);n, p_n)) - \epsilon \right \} \leq C'(\eps) n^{-3}.
\]

Because $\sum_{n > N} n^{-3} \leq N^{-2} \goto 0$ as $N \goto \infty$, 
we obtain an almost-sure eventual lower bound; namely
that for each $\epsilon > 0$, 
there is an a.s. finite random variable $N_0^-(\eps)$ so
that almost surely 
\begin{equation} \label{eq:lowerbound}
\big \{ \; val(K^e((\ell_i), n, p_n)) > val(K^a((\ell_i);n, p_n)) - \epsilon, \,  n > N_0^-(\eps)  \; \big \}.
\end{equation}

We next develop an asymptotic upper bound complementing this lower bound.
Recall the form $\etaarpn$, whose first $r$ entries are fixed
independently of $p > 2r$ and all the $p-r$ remaining entries are
identically $1$. We may apply Lemma \ref{lem:blocks} in such a case. 
Along some sequence 
$(n,p_n)$ where $\min(n,p_n) \goto \infty$ and $p_n/n \goto \gamma$,
\[
\Vert \Delta^a((\ell_i); \etaarpn ;n, p_n,\gamma) - \Delta^e((\ell_i); \etaarpn ;n, p_n) \Vert_F \xrightarrow{a.s} 0 , \mbox{ as }  n \goto \infty , 
\]
which by Lemma \ref{lem:LocalLipschitzCond} implies:
\[
|\kappa(\Delta^a((\ell_i); \etaarpn ;n, p_n,\gamma)) 
- \kappa(\Delta^e((\ell_i); \etaarpn ;n, p_n)) |\xrightarrow{a.s} 0 , \mbox{ as }  n \goto \infty . 
\]
Combining the above:
\begin{equation} \label{eq:solutionConvergestoOptimum}
|val(K^a((\ell_i);n,p_n, \gamma)) - \kappa(\Delta^e((\ell_i); \etaarpn ;n, p_n))|\xrightarrow{a.s} 0.
\end{equation}
By definition $
K^e((\ell_i);n, p_n) \leq \kappa(\Delta^e((\ell_i); \etaarpn ;n, p_n))$, 
so we conclude that 
for each $\epsilon > 0$, 
there is an a.s. finite random variable $N_0^+(\eps)$  so that almost surely
\begin{equation} \label{eq:upperbound}
\big \{ \;  val(K^a((\ell_i);n,p_n, \gamma)) > val(K^e((\ell_i);n, p_n)) - \epsilon, \, n > N_0^+(\eps) \; \big \} .
\end{equation}
 Lemma \ref{lem:AsyOptProb} shows that $val(K^a((\ell_i);n,p_n, \gamma))$
is a constant independent of $p > 2r$.
\eqref{eq:lowerbound} and \eqref{eq:upperbound}  therefore
prove that the almost sure limit 
$\lim_{n\rightarrow \infty} val( K^e((\ell_i);n, p_n))$ exists and is equal to 
\[
val(K^a((\ell_i);n, p, \gamma)) ,
\]
(which was our first claim). Moreover, \eqref{eq:solutionConvergestoOptimum} proves that our second claim is true almost surely. The same conclusions hold for in-probability convergence. 

\subsubsection{Proofs for Section \ref{sec:OrthoInvar}}

{
Let $m = \min(n, p)$, $S = \frac{1}{n}(X'X)$, 
and $S=V\Lambda V'$ be the spectral decomposition of $S$. 
An orthogonally equivariant procedure obeys $\hat{\Sigma}(U'SU) = U'\hat{\Sigma}(S) U$ for every $U \in O(p)$.

We prove the first claim of the diagonal representation (\ref{eq:OE_characterization}). 
By orthogonal equivariance $\hat{\Sigma}(V'SV) = V'\hat{\Sigma}(S)V$ we have
\beq \label{eq:diagrep}
\hat{\Sigma}(\Lambda) = V'\hat{\Sigma}(S)V.
\eeq
Consider diagonal $p\times p$ matrices $\Xi$ with $\Xi_{i,i} \in \{\pm 1\}$.
Since $\Lambda$ is diagonal, $\Xi'\Lambda \Xi = \Lambda $, so of course 
$\hat{\Sigma}(\Xi'\Lambda \Xi) = \hat{\Sigma}(\Lambda)$. Once more invoking
orthogonal equivariance $\hat{\Sigma}(\Xi'\Lambda \Xi) = \Xi'\hat{\Sigma}(\Lambda)\Xi$ we get: 
\begin{equation} \label{eq:SignInvar}
 \Xi'\hat{\Sigma}(\Lambda)\Xi = \hat{\Sigma}(\Lambda).
\end{equation}
Display (\ref{eq:SignInvar}) holds {\it for all such $\Xi$} iff $\hat{\Sigma}(\Lambda)$ is diagonal. 
Therefore by (\ref{eq:diagrep}), $\hat{\Sigma}(S) = V \hat{\Sigma}(\Lambda) V'$ where $\hat{\Sigma}(\Lambda)$ is diagonal. 
This is precisely the form proposed in the first claim of (\ref{eq:OE_characterization}). 

The second claim about the diagonal representation is that for some $\eta_0 \geq 0$,
$\forall i > m$, $\eta_i = \eta_0$. 
Fix any $i_1, i_0 > m$ such that $i_1, i_0 \leq p$ and $i_1 \neq i_0$.
Let $\Xi_{i_0,i_1}$ be an orthogonal matrix that leaves all but two eigenvectors of 
$S$ unchanged \-- $\Xi_{i_0,i_1} V_i = V_i$, $1 \leq i \leq p$, 
$i \not \in \{i_0,i_1\}$  \--
but permutes $V_{i_0}$ and $V_{i_1}$: $\Xi_{i_0,i_1} V_{i_0} = V_{i_1}$, 
$\Xi_{i_0,i_1} V_{i_1} = V_{i_0}$. Then, from the representation 
\begin{equation*}
\hat{\Sigma}(S) = \sum_{i= 1}^{p} V_i V_i' \eta_i,
\end{equation*}
we derive
\begin{eqnarray*}
\hat{\Sigma}( \Xi'_{i_0,i_1} S  \Xi_{i_0,i_1}) &=& \Xi'_{i_0,i_1}\hat{\Sigma}(S)\Xi_{i_0,i_1} \\ 
&=& \sum_{i \not = i_0, i_1} V_i V_i' \eta_i + (\eta_{i_0}V_{i_1}V_{i_1}') + (\eta_{i_1}V_{i_0}V_{i_0}').
\end{eqnarray*}
However, since $i_0$, $i_1 > m$,
both $V_{i_0}$ and $V_{i_1}$ are in the null-space of $S$, 
$S=\Xi'_{i_0,i_1} S  \Xi_{i_0,i_1}$. 
Therefore,  $\hat{\Sigma}(S) = \hat{\Sigma}( \Xi'_{i_0,i_1} S  \Xi_{i_0,i_1})$.
In short
\begin{eqnarray*}
\sum_{i \not = i_0, i_1} V_i V_i' \eta_i + (\eta_{i_0}V_{i_0}V_{i_0}') + (\eta_{i_1}V_{i_1}V_{i_1}')  &=& \hat{\Sigma}(S) \\ 
&=& \hat{\Sigma}( \Xi'_{i_0,i_1} S  \Xi_{i_0,i_1}) \\
&=& \sum_{i \not = i_0, i_1} V_i V_i' \eta_i + (\eta_{i_0}V_{i_1}V_{i_1}') + (\eta_{i_1}V_{i_0}V_{i_0}')  
\end{eqnarray*}
and so
\[
 (\eta_{i_0}V_{i_0}V_{i_0}') +   (\eta_{i_1}V_{i_1}V_{i_1}')  =
 (\eta_{i_1}V_{i_0}V_{i_0}') +  (\eta_{i_0}V_{i_1}V_{i_1}') .
\]
By orthogonality of the $V_{i}$, $\eta_{i_0} = \eta_{i_1}$. 
The second claim of the lemma is proven.}

\subsection{Proofs for Section \ref{ssec:condNumPivot}}
\subsubsection*{Proof of Lemma \ref{lem:defTD}}
\begin{proof}
Observe that 
\[
A(\ell,\eta):=\left[\begin{array}{cc}
\frac{1}{\sqrt{\ell}} & 0\\
0 & 1
\end{array}\right]\left[\begin{array}{cc}
\eta c^{2}+s^{2} & (\eta-1)cs\\
(\eta-1)cs & c^{2}+\eta s^{2}
\end{array}\right]\left[\begin{array}{cc}
\frac{1}{\sqrt{\ell}} & 0\\
0 & 1
\end{array}\right] .
\]
The determinant of the middle matrix is equal to 
\begin{eqnarray*}
\left(\eta c^{2}+s^{2}\right)\left(c^{2}+\eta s^{2}\right)-\left((\eta-1)cs\right)^{2} &=&\eta c^{4}+\eta s^{4}+2\eta c^{2}s^{2}\\
&=&\eta .
\end{eqnarray*}
By the product rule $det(AB) = det(A) \cdot det(B)$,
$det(A(\ell,\eta))=\frac{\eta}{\ell}$. \\
Moreover, $tr(A(\ell,\eta))$ is simply 
\[
\frac{\eta c^{2}+s^{2}}{\ell}+c^{2}+\eta s^{2}=\frac{(\eta-1)c^{2}+1}{\ell}+1+(\eta-1)s^{2} .
\]
The formulas regarding $\nu_{\pm}(A(\ell,\eta))$ now follow by combining 
the trace and determinant identities with
 standard explicit formulas for eigenvalues of
2-by-2 matrices. 
\end{proof}

\subsection{Proofs for Section \ref{sec:OptShrinkSingleSpike}}
\subsubsection*{Proof of Theorem \ref{thm:spike_1}}
\begin{proof}
We establish equivalence between the formula presented in the 
theorem and the formula of Lemma \ref{lem:spikeFormulaEquivalent}. 
We may equivalently write
\[
  \eta_+ =  \frac{c^2+s^2/\ell}{s^2+c^2/\ell} .
\]
Put $a(\ell; \gamma) \equiv  c^2(\ell;\gamma)/\ell+s^2(\ell;\gamma)$
and $b(\ell; \gamma) \equiv  c^2(\ell;\gamma)+s^2(\ell;\gamma) /\ell$,
so that $\eta_+ = b/a$. With $\elld = \ell-1$ we have:
\beq \label{eq:c2s2ident}
    c^2\cdot \elld \cdot (\elld + \gamma) = (\elld^2 - \gamma), \quad 
    s^2\cdot \elld \cdot (\elld + \gamma) = \gamma  \cdot  \ell,
\eeq
Hence
\begin{eqnarray*}
\eta_+ &=&\frac{b \cdot \ell \cdot \elld \cdot (\elld + \gamma) }{a \cdot \ell \cdot \elld \cdot (\elld + \gamma)}\\
   &=& \frac{\ell \cdot (\elld^2-\gamma) + \gamma \ell}{\gamma \ell^2 + (\elld^2-\gamma)} \\
   &=&  \frac{\ell \cdot \elld^2}{\gamma (\elld^2 + 2 \elld + 1) + (\elld^2-\gamma)}\\
   &=& \frac{\ell}{(1+\gamma)  + 2 \gamma/\elld} .
\end{eqnarray*}
\end{proof}

\subsubsection*{Proof of Corollary \ref{cor:asymp_sol}}
\begin{proof}
Observe that 
\begin{eqnarray*}
\lim_{\ell\rightarrow\infty}\frac{\eta_{1}^{*}(\lambda(\ell))}{\ell} &=&
\lim_{\ell\rightarrow\infty}\frac{\ell c^{2}+s^{2}}{\ell(\ell s^{2}+c^{2})} \\
&=& \lim_{\ell\rightarrow\infty}\frac{c^{2}+\frac{s^{2}}{\ell}}{\ell s^{2}+c^{2}} = \frac{1}{1+\gamma}.
\end{eqnarray*}

The final step follows from 
formulas for Spiked asymptotics \--  Lemma \ref{lem:spiked} \-- 
which show that as $\ell\rightarrow\infty$, we have the limits $c\rightarrow 1$,
$s\rightarrow 0$ and $\ell s^{2}\rightarrow\gamma$. 
In addition, 
\begin{eqnarray*}
\lim_{\ell\rightarrow\infty}tr(A(\ell,\eta_{1}^{*}(\lambda(\ell)))) &=&
 \lim_{\ell\rightarrow\infty}\frac{\eta_{1}^{*}(\lambda(\ell))}{\ell}+c^{2}+\eta_{1}^{*}(\lambda(\ell))s^{2}  \hfill \quad 
\mbox{[using  Lemma \ref{lem:defTD}]} \\
&=& \frac{1}{1+\gamma}+1+\frac{\gamma}{1+\gamma}= 2.
\end{eqnarray*}
Applying Lemma \ref{lem:defTD}, we have
\beq \label{eq:LimitNuEtaSingle}
\lim_{\ell\rightarrow\infty}\nu_{\pm}(A(\ell,\eta_{1}^{*}(\lambda(\ell))))=1\pm\sqrt{1-\frac{1}{1+\gamma}} ;
\eeq
our claims follow immediately.
\end{proof}

\subsubsection*{Proof of Lemma \ref{lem:spikeFormulaEquivalent}}
\begin{proof}
The asymptotic pivot takes the form
\[
\Delta(\eta) = A(\ell,\eta(\lambda(\ell)))\varoplus I_{p-2} ,
\]
where $\eta()$ is used to construct the estimate $\hat{\Sigma}$ and
$\ell$ is the single spike of $\Sigma$.  
Now provided that $\nu_-(A(\ell,\eta)) < 1$ and $\nu_+(A(\ell,\eta)) \geq 1$,
the matrix $\Delta$ has largest eigenvalue $\nu_+(A(\ell,\eta))$
and smallest eigenvalue   $\nu_-(A(\ell,\eta))$. 
In that case, the condition number  
$\kappa(\Delta) = \kappa(A(\ell,\eta) \oplus I_{p-2}) = \kappa(A(\ell,\eta))$.
The condition number 
$ \kappa(A(\ell,\eta))$ can be written as 
\begin{eqnarray*}
\kappa(A(\ell,\eta)) = \frac{\nu_+(A(\ell,\eta))}{\nu_-(A(\ell,\eta))} &=&  \dfrac{\frac{T}{2}+\sqrt{T^{2}/4-D}}{\frac{T}{2}-\sqrt{T^{2}/4-D}}\\
&=&
\dfrac{1+\sqrt{1-\frac{4D}{T^{2}}}}{1-\sqrt{1-\frac{4D}{T^{2}}}} ,
\end{eqnarray*}
where we recall the definitions of  $T = T(\ell,\eta) =  \frac{\eta c^2 + s^2}{\ell} + c^2 + s^2 \eta$ 
and $D = D(\ell,\eta) = \eta/\ell$ from Lemma \ref{lem:defTD}, and the manipulations are permissible
because $\eta > 0$ and hence $T \geq s^2/\ell+c^2 > 0$.

Taking into account monotonicity properties of $x \mapsto (1+x)/(1-x)$ and $y \mapsto \sqrt{1-y}$,
we see that  $\kappa(A(\ell,\eta))$,   viewed as a function of $\eta$ for fixed $\ell$,
is minimized when $T^{2}/D$ is minimized. 
Stationary points of
 $T^{2}/D$ as a function of $\eta$, for fixed $\ell$,
 are solutions to $ 2 TT^\prime/D - T^2/D^2 D^\prime = 0$,   
 where $T^\prime = \frac{\partial}{\partial \eta} T$ and $D^\prime = \frac{\partial}{\partial \eta} D$.
 Noting that over the relevant domain,
 $T \geq (s^2/\ell + c^2) > 0$, $\eta \geq 1$, $D > 0$, this reduces to  $2 D T^\prime = T D^\prime$.
Writing this explicitly,
\[
2\left(  \frac{\eta}{\ell} \right) \left( \frac{c^2 }{\ell} + s^2  \right) = \left( \frac{\eta c^2 + s^2}{\ell} + c^2 + s^2 \eta \right) \left(  \frac{1}{\ell}  \right),
\]
and rearranging, we get that the solution obeys:
\[
  {\eta} \cdot \left( \frac{c^2 }{\ell} + s^2  \right) = \left( \frac{ s^2}{\ell} + c^2  \right).
\]
This establishes formula (\ref{eq:altSingleSpikeShrink}) over the domain where the resulting
$\eta_1^* > 1$.
\end{proof}


\subsection*{Proofs for Section  \ref{sec:ShrinkMultiSpike}}

\begin{proof}
{\bf of Theorem \ref{thm:MultiSingle}}

The previous section shows that for $\ell_1 \geq 1$
and for any $\eta$ whatever,
\[
\frac{\nu_+(\ell_1,\eta)}{\nu_-(\ell_1,\eta)} \geq \kappa_1^*(\ell_1;\gamma).
\]
Hence for any multi-spike configuration $(\ell_i)_{i=1}^r$
and any $(\eta_i)_{i=1}^r$,
\[
\kappa(\Delta^a((\ell_i); (\eta_i)) \geq 
  \frac{\max_i \nu_+(\ell_i,\eta_i)}{\min_i \nu_-(\ell_i,\eta_i)} 
  \geq 
  \frac{  \nu_+(\ell_1,\eta_1)}{
  \nu_-(\ell_1,\eta_1)}  \geq 
\kappa_1^*(\ell_1;\gamma).
\]

Combining Lemmas \ref{lem:increasingwrtELL} and \ref{lem:decreasingwrtELL}, 
we have for $\gamma > \gamma_m^*$, 
that for all spike configurations $(\ell_i)_{i=1}^r$
with a fixed value of $\ell_1$,
\[
   \kappa(\Delta^a((\ell_i); \eta_1^*)) \leq  \frac{ \max_{1 \leq \ell' \leq \ell_1} \nu_+(\ell',\eta_1^*(\ell'))}{
   \min_{1 \leq \ell' \leq \ell_1} \nu_-(\ell',\eta_1^*(\ell'))} = 
   \frac{  \nu_+(\ell_1,\eta_1^*(\ell_1))}{
  \nu_-(\ell_1,\eta_1^*(\ell_1))} = \kappa_1^*(\ell_1;\gamma).
\]
Comparing the previous displays, we see that for all multispike configurations
$(\ell_i)$, we have
\[
  \kappa(\Delta^a((\ell_i); (\eta_1^*(\ell_i))) = \kappa_1^*(\ell_1;\gamma),
\]
and this is the optimal value.
\end{proof}

\begin{lem}\label{lem:increasingwrtEta}
Both $\nu_{+}(A(\ell,\eta))$ and $\nu_{-}(A(\ell,\eta))$
are increasing in $\eta$.\end{lem}
\begin{proof}
Observe that for $\eps > 0$,
$A(\ell,\eta+\eps) = A(\ell,\eta) + a(\eps;\ell)$
where
\[
a(\epsilon;\ell):=\left[\begin{array}{cc}
\frac{\epsilon c^{2}}{\ell} & \frac{\epsilon cs}{\sqrt{\ell}}\\
\frac{\epsilon cs}{\sqrt{\ell}} & \epsilon s^{2}
\end{array}\right] .
\] 
As $a(\epsilon) = \eps \cdot (c/\sqrt{\ell},s)' (c/\sqrt{\ell},s)$ 
is nonnegative semi-definite, adding it to $A$ can only increase eigenvalues:
$\nu_+(A(\ell,\eta+\epsilon))\geq\nu_+(A(\ell,\eta))$
and $\nu_-(A(\ell,\eta+\epsilon))\geq\nu_-(A(\ell,\eta))$. 
\end{proof}

\begin{lem}\label{lem:increasingwrtELL} \label{lem:max_eig}
$\nu_+(A(\ell,\eta_1^{*}(\ell)))$ is nondecreasing in $\ell$
and strictly increasing for $\ell \geq \ell_1^+$. 
Moreover, we have
\[
\nu_+(A(\ell,\eta_1^{*}(\ell)))\leq1+\sqrt{\frac{\gamma}{\gamma+1}} .
\]
\end{lem}

\noindent
{\bf Notational Convention.} {\sl Here and below we let $\Dell \equiv \frac{\partial}{\partial \ell}$
and $\Deta \equiv \frac{\partial}{\partial \eta}$ denote the usual derivative operators.}
\begin{proof}
First, consider the interval $\ell < \ell_1^+(\gamma)$. 
Over this interval $\eta^*_1$ collapses the corresponding eigenvalue to $1$, 
which makes $\nu_+(A(\ell,\eta_1^{*}(\ell))) = 1$. 
Since $\nu_+(A(\ell,\eta_1^{*}(\ell)))$ is constant on this interval, it is nondecreasing. 

Now, consider the case where $\ell \geq \ell_1^+(\gamma) > 1 + \sqrt{\gamma}$. Observe that 
\[
\nu_+(A(\ell,\eta_1^{*}(\ell)))=\frac{1}{2}(T+\sqrt{T^{2}-4D}),
\]
where  $T \equiv (\eta c^2 + s^2)/\ell + c^2 + \eta s^2$ and 
$D \equiv \eta/\ell$, assuming $\eta \equiv \eta_1^*(\lambda(\ell))$.
If we can show that $T' \equiv  \Dell T >0$ and
$(T^2-4D)' \equiv  \Dell (T^2-4D) >0$, the conclusion follows.
We simplify $T$ to $T = \frac{2(\ell - 1)}{\gamma + \ell - 1}$ and obtain 
 $T'  = \frac{2\gamma}{(\gamma + \ell -1)^2} > 0$. 
It remains to show that  $(T^2-4D)' >0$.


Again putting $\elld=\ell-1$, we write $T = 2 \elld/(\gamma+\elld)$, $T' = 2\gamma/(\gamma+\elld)^2$
and $D' = 2\gamma/((1+\gamma) \elld + 2\gamma)^2$.
Then
\[
\Dell (T^{2}-4D)  = 2TT' - 4D' =8\gamma\left[\frac{\elld}{(\gamma+\elld)^{3}}-\frac{1}{((1+\gamma)\elld + 2\gamma)^{2}}\right] .
\]
Since $\ell > 1+\sqrt{\gamma}$, $\elld > \sqrt{\gamma}$ and both denominators are positive. 
We have the equivalent predicates
\[
\Dell (T^{2}-4D) \geq 0  
\Longleftrightarrow {\elld}{((1+\gamma)\elld + 2\gamma)^{2}}  - {(\gamma+\elld)^{3}}  \geq 0.
\]
Define the polynomial in $x$ with $\gamma$-dependent coefficients
\begin{eqnarray*}
  g(x;\gamma) &=&  {x}((1+\gamma)x+ 2\gamma)^{2}  - {(\gamma+x)^{3}} \\
       &=& ((1+\gamma)^2-1) x^3 + (4\gamma(1+\gamma) - 3\gamma) x^2 + (4 \gamma^2 - 3 \gamma^2) x - \gamma^3 \\
       &=& (\gamma^2 + 2\gamma) x^3 + (\gamma+4 \gamma^2) x^2 + \gamma^2 x - \gamma^3.
\end{eqnarray*}
We have the chain of equivalent predicates
\begin{eqnarray*}
 \Dell (T^{2}-4D) > 0 , \quad \ell > 1 + \sqrt{\gamma} 
&\Longleftrightarrow& g(\elld;\gamma) > 0, \qquad \elld \geq \sqrt{\gamma} \\
&\Longleftrightarrow& 0 < \min_{x \geq \sqrt{\gamma}} g(x;\gamma) .
\end{eqnarray*}
Whenever $\gamma > 0$, the coefficients of $x$, $x^2$ and $x^3$ in the polynomial $g$ are all positive.
Hence, whenever $\gamma > 0$, $x \mapsto g(x;\gamma)$ is monotone increasing in $x > 0$.
Hence
\begin{eqnarray*}
    \min_{x \geq \sqrt{\gamma}} g(x;\gamma) &=& g(\sqrt{\gamma}; \gamma) \\
    &=& (\gamma^2+2\gamma) \gamma^{3/2} + (\gamma+4 \gamma^2)\gamma+\gamma^{5/2} - \gamma^3\\
    &=& \gamma^{7/2} + 3 \gamma^3 + 3\gamma^{5/2} + \gamma^2 > 0.
\end{eqnarray*}

We conclude that $ \Dell (T^{2}-4D) > 0$ for $\ell > 1 + \sqrt{\gamma}$, showing that
for $\ell > \ell_1^+(\gamma)$, $\ell \mapsto \nu_+(A(\ell,\eta_1^{*}(\ell)))$ is strictly increasing in $\ell$,
tending to $\lim_{\ell\rightarrow\infty} \nu_+(A(\ell,\eta_1^{*}(\ell)))$.
Arguing as in the proof of Corollary  \ref{cor:asymp_sol}, the limit is $1+\sqrt{\frac{\gamma}{\gamma+1}}$.
\end{proof}
\begin{lem}\label{lem:decreasingRwrtELL} 
$R(\ell) = \frac{4D(\ell)}{T(\ell)^2}$ 
is strictly decreasing for $\ell \geq  1$. 
We have
\[
\min_{\ell \geq 1} R(\ell) = \frac{1}{1+\gamma} .
\]
\end{lem}
\begin{proof}
In case $1 \leq \ell \leq \ell_1^+$, we recall that 
$\eta_1^*(\ell) = 1$ and
use the preceding lemma's formulas for $D=D(\ell,1) = 1/\ell$ and $T = T(\ell,1) = 1 + 1/\ell$. 
We obtain
\[
R(\ell) \equiv \frac{4D}{T^2} =  \frac{4/\ell} {(1+1/\ell)^2} = \frac{4\ell}{(\ell+1)^2 };
\]
and we easily verify that $\partial_\ell \log(R(\ell)) < 0$.

On the other hand,  for $\ell > \ell_1^+$, 
the previous lemma's formulas give, writing now
everything in terms of 
$x \equiv \ell-1$, including the argument of $R$:
\[
   R(x)  = \frac{(x+\gamma)^2}{x \cdot ((1+\gamma) x + 2\gamma)} \equiv \frac{Q(x)}{x \cdot P(x)},
\]
say, where $Q(x) \equiv (x + \gamma)^2$ and $P(x) \equiv ((1+\gamma) x + 2\gamma)$.
\[
  R'(x) =   \frac{Q'(x)}{xP(x)} - \frac{Q(x) \cdot [P(x) + x P'(0) ]}{(x P(x))^2}  .
\]
Now on $\ell \geq  \ell_1^+$, $x > \sqrt{\gamma}$ and $P(x) > 0$. So on $x \geq \sqrt{\gamma}$,
\[
   R'(x) < 0  \Leftrightarrow x P(x) Q'(x) - Q(x) \cdot [P(x) + x P'(0) ] < 0.
\]
Now $xP(x)Q'(x)$ is a cubic polynomial and $Q(x)$ is a quadratic polynomial
and since $4 Q(x) = Q'(x)^2$, we can synthetically divide out $Q'$ as a factor. 
Hence we may write
\[
   R'(x) < 0  \Leftrightarrow x P(x)  - Q'(x)/4 \cdot [P(x) + x P'(0) ]< 0.
\]
The polynomial $g(x) = x P(x)  - Q'(x)\cdot [P(x) + x P'(0) ]/4$ can be simplified to
\[
g(x)= -\gamma^2(1+ x) < 0.
\]
\end{proof}

\begin{lem}\label{lem:decreasingwrtELL} 
   Let $\gamma_m^* = \gammamvals$ denote the constant defined in the statement of
Theorem \ref{thm:MultiSingle}. For $\gamma > \gamma_m^*$,
$\nu_-(A(\ell,\eta_1^{*}(\ell)))$ is nonincreasing in $\ell$.
We have
\begin{equation} \label{eq:lower_bound_on_nu}
\nu_-(A(\ell,\eta_1^{*}(\ell)))\geq 1-\sqrt{\frac{\gamma}{\gamma+1}}, \qquad  \ell \geq 1,  \gamma \geq \gamma_m^* .
\end{equation}
\end{lem}

\begin{proof}
On the interval $1 \leq \ell < \ell_1^+$, we have $\eta_1^*=1$
and $\nu_-(A(\ell,1)) = 1/\ell$. This is a decreasing function of $\ell$.

On the interval $\ell \geq \ell_1^+$,
note that
\[
  \nu_-(A(\ell,\eta_1^{*}(\ell))) = \frac{1}{2} \left(  T - \sqrt{T^2-4D} \right) = \frac{T}{2} \cdot \left( 1 - \sqrt{1 - \frac{4D}{T^2}} \right ),
\]
where
$D=\eta/\ell = \frac{\elldot}{(1+\gamma)\elldot + 2\gamma}$,
and $T = \frac{2\elldot}{(\elldot+\gamma)}$. Put $R = R(\elldot ; \gamma) = \frac{4D}{T^2}$. Then
for $\nu_-(\elldot) \equiv \nu_-(A(\ell,\eta_1^{*}(\ell)))$ 
\[
     \nu_- = \frac{T(\elldot)}{2} \cdot F(R(\elldot)),
\]
where $F(r) \equiv (1 - \sqrt{1-r}) $.
We have
\[
2 \cdot \Dell \nu_- =  T' F(R(\elldot)) + T F'(R(\elldot)) \cdot R'(\elldot).
\]

\newcommand{\xop}{ x_1^+}
It will be convenient below to use $x$ in place of $\elldot$
and to write $\xop = \ell_1^+(\gamma)-1$.
Now $T' > 0$ for $x > 0$. Note that $F'(r) = (2 \sqrt{1-r})^{-1}$ is 
positive for $r \in (0,1)$, and that $R \in [1/(1+\gamma),1)$ for $x \geq \xop$.
Thus $F'(R(x)) > 0$ and we may rewrite the 
inequality $\Dell \nu_- < 0$ as
\[
    \frac{F(R(x))}{{ F'(R(x))R(x)}} < - \frac{T(x)}{T'(x)} \cdot \frac{R'(x)}{R(x)}.
\]
We have, using formulas for $T$,$T'$ from the proof of Lemma \ref{lem:increasingwrtELL},
\[
\frac{T(x)}{T'(x)} = \frac{x \cdot (x+\gamma)}{\gamma}.
\]
Also, using notation and identities from the proof of Lemma \ref{lem:decreasingRwrtELL}
\[
  \frac{R'(x)}{R(x)} =  \left\{  \frac{Q'(x)}{Q(x)} - \frac{1}{x P(x)} \{ P(x) + x P'(x) \}\right\}.
\]
Define 
\[
H(r) \equiv \frac{-F(r)}{{  F'(r) \cdot r}} = { 2 \cdot} ((1-r) - \sqrt{1-r})/r, \qquad 0 < r < 1.
\]
Note that on $0 < r < 1$, $H(r)$ is negative and monotone increasing in $r$.
Define $S(x) \equiv S(x ; \gamma) = \frac{T(x)}{T'(x)} \cdot \frac{R'(x)}{R(x)}$, $x > \xop$.
Then
\[
\Dell \nu_-(\ell)  < 0 \Leftrightarrow H(R(x)) > S(x), \qquad x = \ell-1.
\]
We note that by Lemma \ref{lem:decreasingRwrtELL} and monotonicity of $H$,
\[
    H(R(x)) \geq H(R(\infty)) = H(\frac{1}{1+\gamma}).
\]
Hence we have the sufficient condition
\beq \label{eq:suffcond}
  H(R(\infty)) > \max_{x \geq \xop} S(x)  \qquad \Longrightarrow \quad  \Dell \nu_- < 0 \quad \forall \ell > \ell_1^+ .
\eeq
Also
\begin{eqnarray*}
S(x) &\equiv&  \frac{T(x)}{T'(x)} \cdot \frac{R'(x)}{R(x)} \\
&=&\frac{xQ'(x)}{2\gamma} \left\{ \frac{Q'(x)}{Q(x)} - \frac{P(x) + x P'(0)}{x P(x)} \right\}\\
&=& -2 \cdot { \gamma} \cdot \frac{  x + 1}{ (1+\gamma)x + 2\gamma} .
\end{eqnarray*}
By inspection $x \mapsto \frac{  x + 1}{x + \frac{2\gamma}{1+\gamma}}$ is monotone decreasing on $(0,\infty)$ for
$\gamma < 1$ and monotone increasing for $\gamma > 1$. 

Now $x \mapsto S(x)$ has the opposite isotonicity, and so
\[
    \max_{x \geq \xop} S(x) = \left \{ \begin{array}{ll} S(\infty; \gamma) & \gamma \leq 1, \\
            S(\xop; \gamma) & \gamma > 1 .
            \end{array} \right .
\]
In detail   
\beq \label{eq:twobranches}
      \max_{x \geq \xop} S(x) = 
    \left \{ \begin{array}{ll} \frac{-2{\gamma}}{1+\gamma} & \gamma \leq 1, \\
            \frac{-2{ \gamma} \cdot \ell_1^+}{(1+\gamma) (\ell_1^+-1) + 2\gamma } & \gamma > 1 .
            \end{array} \right .
\eeq 
Choosing the branch of (\ref{eq:twobranches}) 
appropriate to $\gamma < 1$ gives us 
from (\ref{eq:suffcond}) 
the sufficient condition
\beq \label{eq:suffcondA}
  H(\frac{1}{1+\gamma}) > \frac{-2 \gamma}{1+\gamma}.
\eeq
Define $h : [0,1] \mapsto \bR$ by
\[
   { h(\gamma) \equiv (\gamma - \sqrt{\gamma} \sqrt{1+\gamma} ) + \frac{\gamma}{1+\gamma}.}
\]
The condition $h=0$ is equivalent to 
$H(\frac{1}{1+\gamma}) = \frac{-2{ \gamma}}{1+\gamma}$.
The condition $h > 0$ implies the sufficient condition (\ref{eq:suffcondA}).

The quantity $\gamma_m^* = \gammamvals =\gammamval$ defined in 
the statement of Theorem \ref{thm:MultiSingle} solves $h(\gamma_m^*) = 0$. Hence 
\[
H(\frac{1}{1+\gamma_m^*}) = \frac{-2 { \gamma_m^*}}{1+\gamma_m^*}.
\]
One can check that $h(\gamma) > 0$ for $ \gamma \in ( \gamma_m^*, 1]$, so
the Lemma's conclusion  $\Dell \nu_-(\ell) < 0$  for $\ell > 1+\sqrt{\gamma}$
follows as advertised in case $\gamma_m^* < \gamma < 1$.

Finally, for $\gamma \geq 1$,
$H(\frac{1}{1+\gamma}) > S(\xop; \gamma)$ uses
the branch of (\ref{eq:twobranches}) relevant to the  sufficient condition (\ref{eq:suffcond}).
We have
{ $S(\xop ;\gamma) = \frac{-2\gamma \ell_1^+}{(1+\gamma) \ell_1^+ - (1-\gamma)}$ and will show that $S_\gamma \equiv S(\xop(\gamma);\gamma) \leq -1$ on $\gamma \geq 1$. Hence}
\[
H(\frac{1}{1+\gamma}) > \lim_{r \goto 0} H(r) = { -1} \geq S(\xop; \gamma).
\]
So the Lemma's conclusion will also be obtained when $\gamma \geq 1$.
To show $S_\gamma \leq -1$,
we write 
\begin{eqnarray*}
-S_\gamma^{-1} &=& \frac{(1+\gamma) \ell_1^+ -(1-\gamma)}{2\gamma \ell_1^+} \\
&=& \left(  ( \frac{1}{2} + \frac{1}{2\gamma}) - (  \frac{1}{2\gamma} - \frac{1}{2} )\frac{1}{\ell_1^+} \right) .
\end{eqnarray*}
The condition $S_\gamma \leq -1$ 
is the same as $1 \geq -S_\gamma^{-1}$ and hence, equivalent to
$\frac{1}{2} - \frac{1}{2\gamma} \geq ( \frac{1}{2} - \frac{1}{2\gamma})/\ell_1^+$,
which of course is implied by $\ell_1^+ \geq 1$ and which follows from
$\ell_1^+(\gamma) > \ell^+(\gamma) = (1 +\sqrt{\gamma}) > 1$. This proves that $\nu_-(A(\ell,\eta_1^{*}(\ell)))$ is decreasing for all $\gamma > \gamma^*$, $\ell > 1$.

The relation \eqref{eq:lower_bound_on_nu} is the direct consequence of the fact that $\nu_-(A(\ell,\eta_1^{*}(\ell)))$ is decreasing and that
$$ \lim_{\ell\rightarrow\infty}\nu_{-}(A(\ell,\eta_{1}^{*}(\lambda(\ell))))=1-\sqrt{1-\frac{1}{1+\gamma}}, $$
which was proven in Corollary \ref{cor:asymp_sol}.
\end{proof}



\begin{lem} \label{lem:increasingGap}
There is a function $\eta(\ell,\nu_0)$ so that  for each $\nu_0 > 1$,
$\nu_+(A(\ell,\eta(\ell,\nu_0))) = \nu_0$.

Fix $\ell_{1}>\ell_{2}>1$. Fix a constant $\nu_0 > 1$. 
Consequently $\eta_i = \eta(\ell_i,\nu_0)$ obey  $\nu_0=\nu_+(A(\ell_{1},\eta_{1}))=\nu_+(A(\ell_{2},\eta_{2}))$. 

If the function
$\ell \mapsto \eta(\ell,\nu_0)/\ell$ 
is decreasing on $\ell \in [\ell_0,\infty)$, for some $\ell_0 < \ell_2 < \ell_1$, then
for such $\eta_i$, $\nu_-(A(\ell_{1},\eta_{1})))<\nu_-(A(\ell_{2},\eta_{2}))$.
\end{lem}

\begin{proof}
Note that $\nu_+(A(\ell,1)) = 1$.
Lemma \ref{lem:increasingwrtEta} shows that, 
for fixed $\ell$, the mapping $\eta \mapsto \nu_+(A(\ell,\eta))$ is nondecreasing on $\eta > 0$.
Also, for fixed $\ell$, the mappings 
$\eta \mapsto T(\ell,\eta) \equiv (c^2/\ell + s^2) \eta + c^2+s^2/
\ell $ and  $\eta \mapsto D(\ell,\eta) = \eta/\ell$ are linear increasing.
Since $T^2 -4D$ is increasing in $\eta$ for all large enough $\eta$, 
and $\nu_+(A(\ell,\eta)) =  [T(\ell,\eta) + \sqrt{T^2(\ell,\eta) - 4 D(\ell,\eta)}]/2$,
we obtain:
\[
   \lim_{\eta \goto \infty} \nu_+(A(\ell,\eta)) = \infty.
\]
By the intermediate value theorem and continuity of $\eta \mapsto \nu_+(A(\ell,\eta))$ we see that 
for each $\ell > 1$, and each $\nu_0 > 1$, there exists $\eta$ solving $\nu_+(A(\ell,\eta)) = \nu_0$.
Denote the infimum of all such solutions by 
$\eta(\ell, \nu_0) = \inf \{\eta>0: \nu_+(A(\ell,\eta)) = \nu_0 \}$.
Define $\eta_i = \eta(\ell_i,\nu_0)$. For $i \in \{1, 2\}$, we have
\begin{eqnarray}
2 \nu_0 &=& T_i+\sqrt{T_i^{2}-4D_i}  \label{eq:proof_start} \\
2 \nu_0-T_i &=& \sqrt{T_i^{2}-4D_i}   . \label{eq:square_me}  
\end{eqnarray}
Hence
\begin{eqnarray}
\nu_-(A(\ell_i,\eta_i)) &=& (T_i - \sqrt{T_i^{2}-4D_i})/2 \nonumber \\
   &=& (T_i - (2\nu_0 -T_i))/2 \nonumber  \\ & =& (T_i -  \nu_0). \label{eq:AltNuIdent}
\end{eqnarray}
Squaring both sides of (\ref{eq:square_me}) we  have
\begin{eqnarray}
4 \nu_0^{2}+T_i^{2}-4 \nu_0T_i &=& T_i^{2}-4D_i \nonumber  \\
\nu_0 \cdot ( \nu_0-T_i) &=& -D_i . \label{eq:condition} 
\end{eqnarray}
Combining (\ref{eq:AltNuIdent}) and (\ref{eq:condition}), we obtain
\[
\nu_-(A(\ell_i,\eta_i)) = D_i/\nu_0.
\]
Since $D(\ell,\eta) = \eta/\ell$ our conclusion that function
$\ell \mapsto \nu_-(A(\ell,\eta(\ell,\nu_0)))$ 
is decreasing follows from showing that 
$\ell \mapsto \eta(\ell,\nu_0)/\ell$ 
is decreasing.
\end{proof}

\begin{lem}\label{lem:EtaNuRatioNonincreasing}
Let $\gamma \geq 1$ and $\nu_0 > 1$
or else $\gamma < 1$ and $1 < \nu_0 <  \frac{1}{1 - \sqrt{\gamma}}$.
The function
$\ell \mapsto \eta(\ell,\nu_0)/\ell$ 
is decreasing on $\ell \in [\ell_1^+,\infty)$.
\end{lem}

\begin{proof}
We first develop an explicit expression for $\eta(\ell,\nu_0)$.
Define $T \equiv T(\ell,\eta ; \gamma) \equiv a \eta + b$
where $a \equiv c^2/\ell+s^2$ and $b \equiv s^2/\ell + c^2$.
Arguing as in (\ref{eq:AltNuIdent}) and (\ref{eq:condition})
\[
    \nu_0 \cdot (T - \nu_0)  = \eta(\ell,\nu_0)/\ell,
\]
leading to 
\[
    \eta(\ell,\nu_0) =  (\ell \cdot \nu_0) \cdot  (T(\ell,\eta(\ell,\nu_0)) - \nu_0).
\]
 Hence
\[
  \eta(\ell,\nu_0) =  (\ell \cdot \nu_0) \cdot  ( a \cdot \eta(\ell,\nu_0) + b - \nu_0).
\]
Hence
\[
   \eta(\ell,\nu_0)  \cdot ( 1 - a\cdot \ell \cdot \nu_0 ) =   (\ell \cdot \nu_0) \cdot  (b  - \nu_0),
\]
and
\[
   \eta = \frac{ (\nu_0 - b )}{a - 1/(\ell \cdot \nu_0)  } .
\]
We now make this explicit in terms
of $\elld \equiv \ell-1$. 
Applying identities (\ref{eq:c2s2ident}),  we write
\begin{eqnarray*}
  \eta &=& \frac{ \nu_0 - [s^2/\ell + c^2] }{c^2/\ell +s^2  - 1/(\ell \nu_0)}  \cdot \frac{ \ell \cdot \elld \cdot (\elld + \gamma)}{\ell \cdot \elld \cdot (\elld + \gamma)} \\
    &=& \frac{ \ell \cdot [ \nu_0 \cdot \elld \cdot (\elld+\gamma) - \gamma  -    (\elld^2 - \gamma) ]}{    \ell^2 \cdot \gamma  + (\elld^2-\gamma)  - \elld \cdot (\elld + \gamma)/ \nu_0  } \\
    &=&  \frac{\ell \cdot [ \elld^2 \cdot (\nu_0-1) + \elld \cdot ( \gamma \cdot \nu_0) ] }{[  \elld^2  \cdot (1 + \gamma -1/\nu_0)  +  \elld \cdot \gamma \cdot  (2-1/\nu_0)] } .
\end{eqnarray*}
Hence 
\begin{eqnarray*}
\frac{\eta(\ell,\nu_0)}{\ell \cdot \nu_0} 
         &=&    \frac{ \elld \cdot (\nu_0-1) + ( \gamma \cdot \nu_0)  }{\elld  \cdot ((1 + \gamma) \nu_0 -1)  +  \gamma \cdot  (2 \nu_0 -1)} .
\end{eqnarray*}
Defining $w: [0,1] \mapsto \bR$ by
\[
  w(\alpha) = \frac{ \alpha \cdot (\nu_0-1) + (1-\alpha) \cdot ( \gamma \cdot \nu_0)  }{\alpha  \cdot ((1 + \gamma) \nu_0 -1)  +  (1-\alpha) \cdot \gamma \cdot  (2 \nu_0 -1)} ,
\]
we have
\[
    \frac{\eta(\ell,\nu_0)}{\ell \cdot \nu_0}  = w( \frac{\elld}{\ell} ).
\]
Now $\ell \mapsto \alpha(\ell) = \frac{\elld}{\ell}$ 
is monotone increasing on $\ell \geq 1$.
Accordingly we now show that $\alpha \mapsto w(\alpha)$ 
is decreasing in $\alpha \in [0,1]$.

Denote the numerator of $w$ by $y(\alpha)$ and the denominator by $x(\alpha)$.
Then $w(\alpha) = y(\alpha)/x(\alpha)$  
can be viewed as  the slope of the ray $R_\alpha$ in $\bR_+^2$ from the origin  $(x=0,y=0)$ 
through the point $p_\alpha = (x(\alpha),y(\alpha))$. Hence,  $w$ is decreasing 
as $\alpha$ increases iff the slope of $R_\alpha$ 
declines as $\alpha$ increases.

Observe that the point $p_\alpha$
is simply the convex combination of $(x(0),y(0))$ and $(x(1),y(1))$.
Therefore, the slope of $R_\alpha$ 
declines as $\alpha$ increases iff

\begin{equation}\label{eq:two_point_relation}
w(1) < w(0).
\end{equation}
We will now show that \eqref{eq:two_point_relation} holds in the region of our interest. Since $\nu_0 > 1$, we have $x(1) =  ((1+\gamma) \nu_0-1) > 0$. Therefore, \eqref{eq:two_point_relation} is equivalent to
\[
   y(1) < x(1) \cdot  \frac{y(0)}{x(0)}.
\]
Substituting $x(0) = \gamma \cdot (2 \nu_0-1) > 0$, $y(0) = \gamma \cdot \nu_0$,
$y(1) = (\nu_0-1)$,
we get
\[
  (\nu_0-1) < ((1+\gamma) \nu_0-1) \cdot \frac{\nu_0}{(2 \nu_0-1)}.
\]
In terms of the new variable $u \equiv (\nu_0-1) \geq 0$, 
this inequality is equivalent, for $u > 0$, to
\[
   u \cdot ( 2u + 1) < (u + \gamma (u+1)) \cdot (u+1).
\]
After rearrangement this becomes:
\begin{equation} \label{eq:uineq}
    0 < \gamma \cdot (u+1)^2 - u^2.
\end{equation}
If $\gamma \geq 1$ (\ref{eq:uineq}) holds for all $u \geq 0$,
i.e. all $\nu_0 \geq 1$. This proves the lemma for $\gamma \geq 1$.

If $\gamma < 1$, (\ref{eq:uineq}) holds for $0 \leq u < \frac{\sqrt{\gamma}}{1-\sqrt{\gamma}}$.

\end{proof}

\newcommand{\?}{\stackrel{?}{<}} 
\begin{lem}\label{lem:support0}
Let $\ell_1 > 1 + \sqrt{\gamma}$. Let $\nu_-^{1,*}:=\nu_-(\ell_1;\eta_1^*(\lambda(\ell_1),\gamma))$. Then, we have
\begin{equation}
	\nu_-^{1,*} < \frac{1}{1 + \sqrt{\gamma}} .
\end{equation}
\end{lem}
\begin{proof}
Let's consider two cases:\\
\noindent
{\bf Case I:} $\ell_1 \leq \ell_1^+(\gamma)$. In this scenario, $\eta^*_1(\ell) = 1$ and $\nu_-^{1,*} = \frac{1}{\ell_1}$. Since $\ell_1 > 1 + \sqrt{\gamma}$, the conclusion follows. 
\\
{\bf Case II:} $\ell_1 > \ell_1^+(\gamma)$. In terms of $T = T(\ell_1,\eta_1^*(\lambda(\ell_1))$
and $D = D(\ell_1,\eta_1^*(\lambda(\ell_1))$,
we have 
\[
\nu_-^{1,*} = \frac{T - \sqrt{T^2 -4D}}{2} .
\]
The inequality we are trying to establish, $\nu_-^{1,*} \? \frac{1}{1 + \sqrt{\gamma}} $, can 
thus be transformed to:
\begin{eqnarray*}
T - \frac{2}{1 + \sqrt{\gamma}} &\?& \sqrt{T^2 -4D} . \\
\end{eqnarray*}
If the LHS is negative, the inequality holds and our claim is proven. If the LHS is positive, we fall into this sub-case:
\[
0 \leq T - \frac{2}{1 + \sqrt{\gamma}} \? \sqrt{T^2 -4D} ;
\]
squaring both sides of $\?$ and simplifying, the sub-case becomes
\begin{eqnarray}\label{eq:lem:support0:term}
0 &\? &\frac{T}{1 + \sqrt{\gamma}} - D - \frac{1}{(1 + \sqrt{\gamma})^2} .
\end{eqnarray}
Now $T(\ell_1,\eta_1) = a_1 \eta_1 + b_1$ and $D(\ell_1,\eta_1) = \eta_1/\ell_1$,
where $a_1 = a(\ell_1; \gamma)$ and $b_1 = b(\ell_1;\gamma)$.
Multiplying the previous relation by $1 + \sqrt{\gamma}$, we have
\[
0 \? a_1 \eta_1 + b_1 - \frac{\eta_1}{\ell_1} \cdot {(1 + \sqrt{\gamma})} - \frac{1}{(1 + \sqrt{\gamma})} .
\]
Using $\eta_1 = b_1/a_1$  and also  $\eta_1 = \ell_1/(1 + \gamma + 2\gamma/\elld)$
this becomes
\[
0 \? 2b_1 - \frac{1 + \sqrt{\gamma}}{1+\gamma+2\gamma/\elld} - \frac{1}{(1 + \sqrt{\gamma})}.
\]
The identities (\ref{eq:c2s2ident}) give us $b_1 = \elld/(\elld+\gamma)$, and the relation can be rewritten
\beq \label{eq:subsid}
\frac{ 2 \elld}{\elld+\gamma} - \frac{1 + \sqrt{\gamma}}{1+\gamma+2\gamma/\elld}  > \frac{1}{(1 + \sqrt{\gamma})}.
\eeq
On the interval $\elld > \sqrt{\gamma}$ we have $\frac{ 2 \elld}{\elld+\gamma} > \frac{2}{(1 + \sqrt{\gamma})}$.
Also on that interval, $1+\gamma+2\gamma/\elld < (1 + \sqrt{\gamma})^2$, and so 
$\frac{1 + \sqrt{\gamma}}{1+\gamma+2\gamma/\elld} >  \frac{1}{(1 + \sqrt{\gamma})}$.
So inequality (\ref{eq:subsid}) and hence (\ref{eq:lem:support0:term}) holds throughout the interval $\ell > 1+ \sqrt{\gamma}$ .
\end{proof}

\begin{lem}\label{lem:supportone}
Let $\ell_1 > \ell > 1$. Let $\nu_-^{1,*}:=\nu_-(\ell_1;\eta_1^*(\lambda(\ell_1),\gamma))$. 
Put
\[
a\equiv c^2/\ell + s^2 ;
\]
where $c = c(\ell) $ and $s=s(\ell)$ are the cosine and sine induced by $\ell$.
Then, we have 
\begin{equation}
\label{eq:boundOnNuMinus}
	\nu_-^{1,*} < \frac{1}{a\ell} .
\end{equation}
\end{lem}
\begin{proof}
We can write $\frac{1}{a\ell}$ as $\frac{1}{\ell s^2 + c^2}$.  We consider two cases:
\\ \noindent
{\bf Case I:} $\eta^*_1(\ell_1) = 1$. In this case, 
\[
\nu_-^{1,*} = \frac{1}{\ell_1} < \frac{1}{\ell} \leq \frac{1}{\ell s^2 + c^2} = \frac{1}{a\ell} ,
\]
as claimed. \\ \noindent
{\bf Case II:} $\eta^*_1(\ell_1) > 1$. For $\ell \in [1, \ell_1]$, let 
\begin{equation*}
g(\ell) \equiv a\ell = \ell s^2 + c^2.
\end{equation*}
$g(\ell)$ can be written in simplified form as
\begin{equation} \label{eq:g_ell}
g(\ell) = \left\{ \begin{array}{ll}  
\frac{\gamma + \ell +\gamma \ell -1}{\gamma + \ell -1} & \ell_1 \geq \ell > \ell_+(\gamma)\\
\ell & 1 \leq \ell \leq \ell_+(\gamma)
\end{array}   \right. .
\end{equation}
Note that the conclusion of the lemma follows if we show 
\begin{equation} \label{eq:lemma_conclusion_equivalent}
\max_{\ell \in [1, \ell_1]} g(\ell) < \frac{1}{\nu_-^{1,*}}.
\end{equation}
A simple derivative calculation shows that $g(\ell)$ is monotonic on each branch of \eqref{eq:g_ell}; it is always increasing on the lower branch, and, depending on the value of $\gamma$, it is either monotonically increasing or decreasing on the upper branch. Therefore, $g(\ell)$ attains its maximum either at $\ell = \ell_+(\gamma) = 1 + \sqrt{\gamma}$ or $\ell = \ell_1$. 

We show \eqref{eq:lemma_conclusion_equivalent} by verifying that $\nu_-^{1,*}$ is smaller than both $\frac{1}{g(\ell_1)}$ and $\frac{1}{1+\sqrt{\gamma}}$. Let $s_1 = s(\ell_1)$ and $c_1=c(\ell_1)$. By Lemma \ref{lem:defTD}, 
\[
\nu_+(\ell_1,\eta^*_1(\ell_1)) \nu_-(\ell_1,\eta^*_1(\ell_1)) = \frac{\eta^*_1(\ell_1)}{\ell_1} = \frac{c_1^2 + s_1^2 / \ell_1}{\ell_1 s_1^2 + c_1^2} .
\]
Therefore, 
\[
\nu_-(\ell_1,\eta^*_1(\ell_1)) = \frac{c_1^2 + s_1^2 / \ell_1}{\nu_+(\ell_1,\eta^*_1(\ell_1))}  \frac{1}{\ell_1 s_1^2 + c_1^2} .
\]
The proof of Lemma  \ref{lem:KrPlusLemma} 
shows that $\nu_+(\ell;\eta) \geq 1$ for $\eta \geq 1$. Hence 
\begin{eqnarray*}
\nu_+(\ell_1,\eta^*_1(\ell_1)) \geq 1 > c_1^2 + \frac{s_1^2}{\ell_1} 
\leftrightarrow 1 > \frac{c_1^2 + \frac{s_1^2}{\ell_1}}{\nu_+(\ell_1,\eta^*_1(\ell_1))}.
\end{eqnarray*}
Hence, 
\[
\nu_-(\ell_1,\eta^*_1(\ell_1)) = \frac{c_1^2 + s_1^2 / \ell_1}{\nu_+(\ell_1,\eta^*_1(\ell_1))}  \frac{1}{\ell_1 s_1^2 + c_1^2} < \frac{1}{\ell_1 s_1^2 + c_1^2} = \frac{1}{a \ell_1} .
\]
Therefore, $\nu_-^{1,*} <\frac{1}{g(\ell_1)}$. The fact that $\nu_-^{1,*} <\frac{1}{1+\sqrt{\gamma}}$ was proven earlier in Lemma \ref{lem:support0}. 
\end{proof}

\begin{lem}\label{solution}
Let $\ell_1 > \ell > 1$. Let $\nu_-^{1,*}:=\nu_-(\ell_1;\eta_1^*(\lambda(\ell_1),\gamma))$. Set 
\begin{equation} \label{eq:abdef}
a\equiv c^2/\ell + s^2,\;\;b\equiv s^2/\ell + c^2,
\end{equation}
where $c = c(\ell; \gamma)$ and $s= s(\ell;\gamma)$.
For 
\beq \label{eq:OptEtaAlt}
\eta \equiv \frac{(\nu_-^{1,*} - b)}{a-1/(\ell\nu_-^{1,*})} ,
\eeq
we  have $\nu_-(\ell,\eta)=\nu_-^{1,*}$. 
Moreover, $\eta < 1$ iff $\ell < \frac{1}{\nu_-^{1,*}}$.
\end{lem}
\begin{proof}
We directly solve for $\eta$ that satisfies
\begin{eqnarray*}
\nu_-(\ell,\eta)&=&\nu_-^{1,*}.
\end{eqnarray*}
Equivalently we want $T - \sqrt{T^2-4D} = 2\nu_-^{1,*}$; hence
\begin{eqnarray*}
T^2+4(\nu_-^{1,*})^2-4T\nu_-^{1,*}&=&T^2-4D,\\
(\nu_-^{1,*})^2-  (\eta \cdot a + b) \cdot \nu_-^{1,*}&=&-\frac{\eta}{\ell},\\
(\nu_-^{1,*})^2-b \cdot \nu_-^{1,*}&=&\eta \cdot (  a  \cdot\nu_-^{1,*}  -\frac{1}{\ell}).
\end{eqnarray*}
Supposing that $\nu_-^{1,*} \not = \frac{1}{a\ell}$, we obtain (\ref{eq:OptEtaAlt}).
Lemma \ref{lem:supportone} showed that $\nu_-^{1,*} < \frac{1}{a\ell}$, 
so (\ref{eq:OptEtaAlt}) holds under the stated assumptions. 
We have the equivalent conditions:
\begin{eqnarray*}
\eta < 1 &\Leftrightarrow&
(\nu_-^{1,*})(\nu_-^{1,*} - b) > (a\nu_-^{1,*}-\frac{1}{\ell})\hfill \quad \mbox{[since $(a\nu_-^{1,*}-\frac{1}{\ell})<0$]}\\ 
&\Leftrightarrow& (\nu_-^{1,*})^2 - (a + b)(\nu_-^{1,*}) + \frac{1}{\ell}>0 \\ 
&\Leftrightarrow& (\nu_-^{1,*})^2 - (1 + \frac{1}{\ell})(\nu_-^{1,*}) + \frac{1}{\ell} = (\nu_-^{1,*}-1) (\nu_-^{1,*}-1/\ell)>0 \\ 
&\Leftrightarrow&  \ell < \frac{1}{\nu_-^{1,*}}, \hfill \quad \mbox{[since $\nu_-^{1,*} < 1$]} .\\ 
\end{eqnarray*}
This proves our last claim. 
\end{proof}

\begin{lem} \label{lem:continuity}
$\ell \mapsto \eta^*_m(\ell; \ell_1 , \gamma)$ of 
Theorem \ref{thm:OptimalShrinkage} is continuous.
\end{lem}
\begin{proof}
By Lemma \ref{solution}, we can equivalently write 
\[
\eta^*_m = \max \bigg(1, \frac{(\nu_-^{1,*} - b)}{a-1/(\ell \nu_-^{1,*})}\bigg) ,
\]
where 
\[
a\equiv c^2/\ell + s^2,\;\;b\equiv s^2/\ell + c^2 .
\]
Since the max function is continuous, it suffices to show $\frac{(\nu_-^{1,*} - b)}{a-1/(\ell \nu_-^{1,*})}$ is continuous. \\
By Lemma \ref{lem:spiked}, we can see that $\lambda(\ell;\gamma)$, $c(\ell;\gamma)$, and $s(\ell;\gamma)$ are all 
continuous functions on the domain $ \{ 1 \leq \ell < \infty \} \times \{ 0 < \gamma < \infty\}$,
as are $a$ and $b$. 
Also, for $\ell \geq 1$ we have $c^2 + s^2 \ell \geq 1$.
We obtain:
\begin{enumerate}
\item $\eta^*_1(\ell_1, \gamma) \equiv \max(1, \frac{\ell_1c^2(\ell_1) + s^2(\ell_1)}{\ell_1  s^2(\ell_1) + c^2(\ell_1)})$ is a continuous function of $\ell_1$ and $\gamma$.
\item All entries of $A(\ell_1, \eta^*_1(\ell_1);\gamma)$ are continuous functions of $\ell_1$ and $\gamma$. 
Since $\lambda_{min}(\cdot)$ is a continuous function of the matrix entries, 
we conclude that $\nu_-^{1,*}$ is also a continuous function of $\ell_1$ and $\gamma$.
\end{enumerate}

Therefore, both the numerator and the denominator 
are continuous functions of the arguments. 
By Lemma \ref{lem:supportone}, for all valid $(\ell;  \ell_1,\gamma)$ the denominator never vanishes. 
The conclusion follows.
\end{proof}

\subsubsection*{Proof of Theorem \ref{thm:OptimalShrinkage}}
\begin{proof}
Assume the spike configuration $\ell_1 > \ell_2 > \dots > \ell_r > 1$. We
will construct the desired nonlinearity $\eta_m^*$. For large $\ell$ we  
use the one-spike optimal solution $\eta_1^*$.  
\[
\eta^*_m(\ell ; \ell_1, \gamma) = \eta^*_1(\ell;\gamma) , \qquad \ell \geq \ell_1,
\]
For smaller $\ell$, we show below how to construct $\eta_m^*$ so that
$\forall 1<\ell<\ell_1$
\begin{equation} \label{eq:optimalityconditions_low}
\nu_{-}(\ell,\eta^*_m(\ell)) \geq  \nu_{-}(\ell_1,\eta^*_1(\ell_1)),
\end{equation}
\begin{equation} \label{eq:optimalityconditions_up}
\nu_{+}(\ell,\eta^*_m(\ell)) \leq  \nu_{+}(\ell_1,\eta^*_1(\ell_1)).
\end{equation}
These properties entail
\begin{eqnarray}
\kappa(\Delta^a((\ell_i);(\eta^*_m))) &=& \frac{\max_i \nu_{+}(\ell_i,\eta^*_m(\ell_i))}{\min_i\nu_{-}(\ell_i,\eta^*_m(\ell_i))} \hfill \quad 
\mbox{[by the definition of $\eta^*_m$]} \nonumber \\
&=& \frac{\nu_{+}(\ell_1,\eta^*_m(\ell_1))}{\nu_{-}(\ell_1,\eta^*_m(\ell_1)} = \kappa^*_1(\ell_1) . \label{eq:optval}
\end{eqnarray}
Since for {\it any} nonlinearity $\eta$
\begin{eqnarray*}
\kappa(\Delta^a((\ell_i);(\eta))) &\geq& \frac{\nu_{+}(\ell_1,\eta(\ell_1))}{\nu_{-}(\ell_1,\eta(\ell_1))}  \hfill \quad 
\mbox{[by Theorem \ref{thm:spike_1}]} \\
&\geq& \kappa_1^*(\ell_1) ,
\end{eqnarray*}
property (\ref{eq:optval}) ensures that $\eta$ is (asymptotically) optimal. 
Consider the following candidate:
\[
    \eta_m^*(\lambda(\ell))  = \left \{ \begin{array}{ll}
                 \frac{(\nu_-^{1,*} - b)}{a -1/(\ell \cdot \nu_-^{1,*})}  & \ell >  \ell_{\eta_m^*}^+ \\
                 1           &  \ell \leq  \ell_{\eta_m^*}^+
                 \end{array}
                 \right . ,
\]
where $\ell^+_{\eta^*_m} = \frac{1}{\nu_-^{1,*}}$ and $\nu_-^{1,*} = \nu_-(\ell_1, \eta^*_1(\ell_1))$ and $a=a(\ell;\gamma)$ and $b = b(\ell; \gamma)$
are the functions defined e.g. in (\ref{eq:abdef}) and used frequently in earlier lemmas. By Lemma \ref{lem:supportone}, this function is well-defined for all $1<\ell<\ell_1$. By Lemma \ref{solution}, $\eta^*_m \geq 1$.

If $\ell_1 \leq \ell_1^+(\gamma)$ then $\eta^*_1(\ell_1) = 1$. However, $\eta_m^*(\ell_1)=\eta^*_1(\ell_1)=1$ 
which causes in this case $\eta^*_m(\ell) \equiv 1$ for $1 \leq \ell \leq \ell_1$. 
On the other hand, if $\ell_1 > \ell_1^+(\gamma)$, Lemma \ref{lem:support0} 
shows that $\ell_{\eta_m^*}^+ > 1 + \sqrt{\gamma}$.
So, in either case, we can see that $\eta^*_m(\cdot)$ collapses the vicinity of the bulk to 1.   

To finish the proof, we need to show that (\ref{eq:optimalityconditions_low}) 
and (\ref{eq:optimalityconditions_up}) hold. We consider two cases, depending on $\ell <>  \ell^+_{\eta^*_m}(\ell_1; \gamma)$: 
\begin{itemize}
\item[$\eta^*_m(\ell) = 1$ :] In this case, $\nu_{+}(\ell,\eta^*_m(\ell)) = 1 \leq \nu_+^{1,*}$. In addition, $\nu_{-}(\ell,\eta^*_m(\ell)) = \frac{1}{\ell} \geq \frac{1}{\ell^+_{\eta^*_m}} = \nu_-^{1,*}$. Therefore, (\ref{eq:optimalityconditions_low}) and (\ref{eq:optimalityconditions_up}) hold.

\item[$\eta^*_m(\ell) > 1$ :] By Lemma \ref{solution}, $\nu_{-}(\ell,\eta^*_m(\ell)) = \nu_-^{1,*}$
for $\ell^+_{\eta^*_m} \leq \ell \leq \ell_1$. 
We show $\nu_{+}(\ell,\eta^*_m(\ell)) \leq \nu_+^{1,*}$ over the same range of $\ell$,
arguing by contradiction. Suppose 
\begin{equation}
\label{false_assumption}
\nu_{+}(\ell,\eta^*_m(\ell)) > \nu_+^{1,*} ;
\end{equation}
Lemma \ref{lem:increasingGap} provides a function $\eta(\ell,\nu)$ so that
$\eta^+ \equiv \eta(\ell,\nu_+^{1,*})$ solves
$\nu_{+}(\ell,\eta^+) = \nu_+^{1,*}$. Since $ \Deta \nu_+  > 0$ (Lemma \ref{lem:increasingwrtEta}),  
(\ref{false_assumption}) would imply
\begin{equation}\label{false_result}
	\eta^+ < \eta^*_m(\ell).
\end{equation}
Assume either that $\gamma \geq 1$
or else that $0 < \gamma < 1$, in which case note that
by Lemma \ref{lem:increasingwrtELL},
$\nu_+^{1,*} \leq 1 + \sqrt{\frac{\gamma}{1+{\gamma}}} 
             < \frac{1}{1-\sqrt{\gamma}}$. 
 In either case, Lemma \ref{lem:EtaNuRatioNonincreasing} applies
and the ratio  $\ell \mapsto \eta(\ell,\nu_+^{1,*})/\ell$ 
is non-increasing.
Applying the final conclusion of 
Lemma \ref{lem:increasingGap}
with $\ell_2=\ell$ and $\ell_1$ as here,
and with $\nu_0 = \nu_+^{1,*}$,
we have  with $\eta_i = \eta(\ell_i,\nu_0)$ the inequality
\beq   \label{eq:falsestrict}
\nu_-(A(\ell_2,\eta_2)) > \nu_-(A(\ell_1,\eta_1)).
\eeq

Observe that 
\begin{eqnarray*}
\eta_1 &\equiv& \eta(\ell_1,\nu_0)\\ 
       &=& \eta(\ell_1,\nu_+^{1,*})\\
       &=& \eta_1^*(\ell_1),
\end{eqnarray*}
where the last step follows from the fact that
the relation $\nu \mapsto \eta(\ell_1,\nu)$ 
in one-one over the relevant range. 
Also:
\begin{eqnarray*}
\nu_-(A(\ell_1,\eta_1)) &=& \nu_-(A(\ell_1,\eta_1^*(\ell_1)))\\
&=& \nu_-^{1,*}\\
&=& \nu_-(A(\ell,\eta_m^*(\ell))).
\end{eqnarray*}
where the last  step uses
 the constancy of $\ell \mapsto \nu_-(A(\ell,\eta_m^*(\ell)))$
 for $\ell \in [\ell_m^+,\ell_1]$
 (Lemma \ref{solution}).
We may rewrite  (\ref{eq:falsestrict})  as:
\beq \label{eq:falseoutcome}
\nu_-(A(\ell, \eta^+)) > \nu_-(A(\ell,\eta_m^*(\ell))), \qquad \ell \in [\ell_m^+,\ell_1].
\eeq

By Lemma \ref{lem:increasingwrtEta}, $\Deta \nu_-  > 0$. 
Therefore, (\ref{eq:falseoutcome}) implies $\eta^+ > \eta^*_m(\ell)$. 
This contradicts (\ref{false_result});
hence, assumption (\ref{false_assumption}) fails. 
So (\ref{eq:optimalityconditions_up}) 
holds and  $\eta^*_m$ is optimal. 
\end{itemize}
\end{proof}

\subsubsection*{Proof of Lemma \ref{lem:AsympMShrink}}
\begin{proof} 
Put for short $\eta^*_i \equiv \eta^*_m(\lambda_i; \lambda_1, \gamma)$. 
Note that for $i > r$, $\eta^*_i = 1$. Observe that 
\begin{equation} \label{eq:max_decomp}
\footnotesize
\max_{1\leq i\leq \min(n,p_n)} |\eta_{i,n}^e - \eta_{i}^*| 
=  \max \bigg( \max_{1\leq i\leq r} |\eta_{i,n}^e - \eta_{i}^*|,\max_{(r+1)\leq i\leq \min(n,p_n)} |\eta_{i,n}^e - \eta_{i}^*| \bigg)
\end{equation}
Let's consider the left branch on the RHS of (\ref{eq:max_decomp}). 
We have $\eta^e_{i,n} = \eta^*_m(\lambda_{i,n}; \lambda_{1,n}, \frac{p_n}{n})$. 
We know that $\forall 1 \leq i \leq r$ $\lambda_{i,n} \xrightarrow{a.s} \lambda_{i}$ and $\frac{p_n}{n} \rightarrow \gamma$. 
Now observe that $\eta^*_m(\cdot)$ is a continuous function. Hence, we conclude that $\forall 1 \leq i \leq r$ (almost surely)
\[
\lim_{n\rightarrow \infty } \eta^e_{i,n} = \lim_{n\rightarrow \infty } \eta^*_m(\lambda_{i,n}; \lambda_{1,n}, \frac{p_n}{n}) = 
\eta^*_m(\lambda_{i}; \lambda_{1}, \gamma) =
\eta^*_{i}
\]
which proves that 
\[
\max_{1\leq i\leq r} |\eta_{i,n}^e - \eta_{i}^*| \xrightarrow{a.s} 0 ,
\]

Now, let's consider the right branch of (\ref{eq:max_decomp}). By Lemma \ref{lem:max_bulk}, as $n\rightarrow \infty$
\begin{equation} \label{eq:right_branch}
\max_{(r + 1) \leq i \leq \min(n, p_n)} \lambda_{i,n} \xrightarrow{a.s} (1+\sqrt{\gamma})^2 .
\end{equation}
Note that $\forall i > r$, $\eta_{i}^* = 1$. 
We are done if we can show that $\eta^e$ almost surely collapses \textbf{all} $\lambda_{i,n}$ for $i>r$. 
By \eqref{eq:right_branch}, the conclusion follows if we show $\eta^e$ almost surely collapses the vicinity of the bulk to $1$. \\
Consider the following two cases:
\begin{itemize}
\item $\ell_1 \leq 1 + \sqrt{\gamma}: $ In this case, the largest sample eigenvalue converges almost surely to $(1 + \sqrt{\gamma })^2 < \lambda_1^+(\gamma)$. If the largest eigenvalue of the sample covariance is smaller than $\lambda_1^+(\frac{p_n}{n})$, $\eta^e$ collapses all of the eigenvalues to $1$. Observe that as $\frac{p_n}{n} \rightarrow \gamma$, $\lambda_1^+(\frac{p_n}{n}) \rightarrow \lambda_1^+(\gamma)$. Therefore, $\eta^e$ collapses the vicinity of the bulk to one almost surely and
\[
\max_{(r + 1) \leq i \leq \min(n, p_n)} |\eta^e_{i, n} - 1| \xrightarrow{a.s} 0.
\]

\item $\ell_1 > 1 + \sqrt{\gamma}: $ Observe that by continuity of $\nu_-(\cdot)$, $\ell^+_{\eta^*_m}(\lambda_{1,n}) \xrightarrow{a.s} \ell^+_{\eta^*_m}(\lambda_{1})$. By Lemma \ref{lem:support0}, we have $\lambda^+_{\eta^*_m}(\lambda_{1}) > (1+\sqrt{\gamma})^2$. Therefore, $\eta^e$ almost surely collapses the vicinity of the bulk to 1 as $n \rightarrow \infty$.
\end{itemize}
By the analysis above, we conclude that 
\beq \label{eq:EtaConv}
\max_{1\leq i\leq \min(n,p_n)} |\eta_{i,n}^e - \eta_{i}^*|\xrightarrow{a.s} 0,
\eeq
which proves our first claim. 

Lemma \ref{lem:LocalLipschitzCond} showed that $\kappa(\Delta) = K(\Lambda(\Delta))$
where $K$ and $\Lambda$ are both locally Lipschitz functions, $\Lambda$ being Lipschitz with respect to operator norm. 
Hence $\Delta \mapsto \kappa(\Delta)$ is locally Lipschitz with respect to operator norm.
Put for short
\[
\Delta_{\eta_{n}}\equiv \Delta^e((\ell_i),(\eta_{i,n}^e);n,p_n), \qquad \Delta_{\eta^*} \equiv \Delta^e((\ell_i),(\eta_{i}^*);n,p_n),
\]
where
\[
\eta^* = diag(\eta^*_1,\dots, \eta^*_p),\qquad \eta^e = diag(\eta^e_{1,n},\dots, \eta^e_{p,n}).
\]
\begin{eqnarray} \label{eq:convergence_lambda_max}
\| \Delta_{\eta_{n}} - \Delta_{\eta^*} \|_{op} &=& \| \Sigma^{-\frac{1}{2}} \hat{\Sigma}(\eta^e-\eta^*) \Sigma^{-\frac{1}{2}}) \|_{op} \nonumber \\ 
&\leq& \lambda_{max} (\Sigma^{-\frac{1}{2}})^2 \max_{1\leq i\leq \min(n,p_n)} |\eta_{i,n}^e - \eta_{i}^*| \\
& = & \max_{1\leq i\leq \min(n,p_n)} |\eta_{i,n}^e - \eta_{i}^*| . \nonumber
\end{eqnarray}
We conclude from (\ref{eq:EtaConv}) that as $n \goto \infty$, $\kappa(\Delta_{\eta_{n}})  \goto \kappa(\Delta_{\eta^*})$ both in probability and almost surely.

\end{proof}

\subsection{Proofs for Section \ref{sec:MinimaxInterp}}
\begin{proof}

{\it Well-definedness of $\etammnl$.}
We begin with Part b, making some remarks about $\etammnl$. Part b of the theorem statement
claimed in passing that $a(\ell)\ell \neq 1/\nu_-^{\mmbp}$
over the range $\ell > \frac{1}{\nu_-^{\mmbp}}$. Now
(\ref{eq:g_ell}) gave the identity
\[
a(\ell) \ell = \frac{(1+\gamma) \elldot + 2 \gamma}{\elldot + \gamma}.
\]
Hence we must show
\[
 \frac{(1+\gamma) \elldot + 2 \gamma}{\elldot + \gamma} \neq \frac{1}{1 - \sqrt{\frac{\gamma}{1+\gamma}}}.
\]
Observe that 
\begin{eqnarray} \label{eq:denominator_bound}
\frac{1}{1 - \sqrt{\frac{\gamma}{1+\gamma}}} - \frac{(1+\gamma) \elldot + 2 \gamma}{\elldot + \gamma} 
&=& 1 + \gamma + \sqrt{\gamma^2 + \gamma} - \frac{(1+\gamma) \elldot + 2 \gamma}{\elldot + \gamma} \nonumber \\
&=& \frac{\gamma ^ 2 - \gamma + \sqrt{\gamma^2 + \gamma} (\elldot + \gamma)}{\elldot + \gamma} \nonumber\\
&=& \frac{\gamma ^ 2 - \gamma + (\sqrt{\gamma^2 + \gamma} - \sqrt{\gamma}) (\elldot + \gamma) + \sqrt{\gamma} (\elldot + \gamma)}{\elldot + \gamma} \nonumber \\
&\geq& \frac{\gamma ^ 2 - \gamma + (\sqrt{\gamma^2 + \gamma} - \sqrt{\gamma}) (\elldot + \gamma) + \sqrt{\gamma}^2}{\elldot + \gamma} \nonumber \\
&=& \frac{\gamma ^ 2 + (\sqrt{\gamma^2 + \gamma} - \sqrt{\gamma}) (\elldot + \gamma)}{\elldot + \gamma} \nonumber \\
&>& \sqrt{\gamma^2 + \gamma} - \sqrt{\gamma}.
\end{eqnarray}
Therefore, we indeed have the
advertised inequality.

{\it $\etammnl$ as the limit of $\eta^*_m$.} Part 3 of Theorem \ref{thm:OptimalShrinkage} gives an explicit formula
for $\eta_m^*(\ell;\ell_1)$ with two notable features.
First,  $\eta_m^*(\ell) = 1$ below a certain threshold 
$\ell < \ell_{\eta^*_m}^+(\ell_1)$.
Second, above this threshold, $\eta_m^*$ 
has the parametric form $\eta(\ell;v) = \frac{v-b}{a-1/(\ell \cdot v)}$
as a (very simple) ratio involving 
two given fixed functions $a=a(\ell) = c^2/\ell + s^2$ and $b=b(\ell)= s^2/\ell + c^2$
(by now familiar to us) and
involving a real parameter $v$. This  parameter,
controlling both the threshold and the
shape of $\eta$, is  $v= \nu_{-}^{1,*}(\ell_1)$ by the optimal
rule, and according to (\ref{eq:LimitNuEtaSingle}) tends to
a limit $\nu_{-}^{1,*}(\infty)$, say, as $\ell_1 \goto \infty$. 
This limit $\nu_{-}^{1,*}(\infty) \equiv \nu_-^{\mmbp} = 1 - \sqrt{\frac{\gamma}{\gamma+1}}$.

We showed above that for any $\ell \geq \frac{1}{\nu_-^{MM}}$, $a \neq 1/(\ell \nu_-^{\mmbp})$. Therefore, by examining the parametric form of $\eta(\ell; \nu)$ below and above its threshold, it immediately follows that $\eta(\ell; \nu)$ is continuous at $\nu = \nu_-^{MM}$ (for any $\ell \geq 1$). Hence, we conclude that $\forall \ell \geq 1$:
\begin{eqnarray}\label{eq:limit_etamm}
\lim_{\ell_1 \goto \infty} \eta_m^*(\ell; \ell_1) &=&  \lim_{\nu \goto \nu_-^{MM}} \eta(\ell; \nu) \nonumber \\
&=& \eta(\ell; \lim_{\nu \goto \nu_-^{MM}} \nu) \\
&=& \eta(\ell; \nu_-^{MM}) = \eta_{MM}(\ell) \nonumber
\end{eqnarray}

It remains to show the claim (\ref{eq:asyinequality}) in Part b
and also to show Parts a and c. All three tasks will follow 
from matching upper and lower bounds we establish next.

{\it A lower bound on all procedures.} 
Now $\kappa_1^*(\ell_1) \equiv \kappa_{1}^{*}(\ell_1;\gamma) $ is the exact optimal $\kappa$-loss \-- optimal across
all nonlinearities $\eta$ \-- at each specific $1$-spike configuration with
prescribed $\ell_1$. 
Hence no shrinker can guarantee better performance across {\it all} configurations than
$\sup_{\ell_1>1}  \kappa_{1}^{*}(\ell_1)$:
\beq \label{eq:lboundA}
   K_r^*(\eta) \geq \sup_{\ell_1>1}  \kappa_{1}^{*}(\ell_1), \qquad \forall \eta.
\eeq
We next evaluate the right side  lower bound, showing
\begin{eqnarray}
\sup_{\ell_1>1}  \kappa_{1}^{*}(\ell_1) &=& \lim_{\ell_1 \goto \infty }  \kappa_{1}^{*}(\ell_1) \label{eq:supeqlim} \\
&\equiv& \kappa_1^*(\infty) \nonumber\\
&=& \frac{1+ \sqrt{\frac{\gamma}{\gamma+1}}}{1 - \sqrt{\frac{\gamma}{\gamma+1}}} \label{eq:evalkappastar}.
\end{eqnarray}
To obtain evaluation (\ref{eq:supeqlim})-(\ref{eq:evalkappastar}), we first recall that 
\begin{eqnarray}
  \kappa_1^*(\ell_1) &=&  \frac{\nu_+(\ell_{1},\eta_{1}^{*}(\lambda(\ell_1)))}{\nu_-(\ell_{1},\eta_{1}^{*}(\lambda(\ell_1)))}  \nonumber\\
  &=& \frac{1 + \sqrt{1-\frac{4D}{T^2}}}{1 - \sqrt{1-\frac{4D}{T^2}}}\nonumber \\
  &=& \frac{1 + \sqrt{1-R(\ell_1)}}{1 - \sqrt{1-R(\ell_1)}}, \label{eq:evalRatio}
\end{eqnarray}
where in the first step we used the identities
\[
  \nu_\pm(\ell_{1},\eta_{1}^{*}(\lambda(\ell_1)))) = ({T} \pm \sqrt{T^2 - 4D})/2.
\]
with $D \equiv D(\ell_1) = \mbox{det}(A(\ell_1,\eta_1^*(\lambda(\ell_1)))$ 
and $T \equiv T(\ell_1) = \mbox{tr}(A(\ell_1,\eta_1^*(\lambda(\ell_1))))$.
The quantities $D$ and $T$ were studied before,
for example in Lemmas \ref{lem:increasingwrtELL}-\ref{lem:decreasingwrtELL},
and explicit formulas were given in terms of $\ell$ and $\gamma$.
Recall the quantity $R = \frac{4D}{T^2}$ studied in Lemma \ref{lem:decreasingRwrtELL},
where it was shown that $\ell \mapsto R(\ell)$ is decreasing.

The proof of Lemma \ref{lem:spikeFormulaEquivalent} pointed to 
certain monotonicity properties of $y \mapsto (1+y)/(1-y)$ and $r \mapsto \sqrt{1-r}$; 
combining these with monotonicity of $\ell \mapsto R(\ell)$ 
yields at once that  $\ell_1 \mapsto \kappa(A(\ell_1,\eta_1^*(\lambda(\ell_1))))$,  
is an increasing function of $\ell_1$. 
(\ref{eq:supeqlim}) follows. The evaluation (\ref{eq:evalkappastar}) flows
from combining (\ref{eq:evalRatio}) with the evaluation $R(\infty) =\frac{1}{\gamma+1}$ 
provided by Lemma \ref{lem:decreasingRwrtELL}.
Combining (\ref{eq:lboundA}), (\ref{eq:supeqlim}) and
(\ref{eq:evalkappastar}), we obtain the lower bound
\beq \label{eq:lboundB}
   K_r^*(\eta)  \equiv \sup_{\cL_r} \kappa((\ell_i),\eta)  \geq \frac{1+ \sqrt{\frac{\gamma}{\gamma+1}}}{1- \sqrt{\frac{\gamma}{\gamma+1}}}, \qquad \forall \eta.
\eeq

{\it An upper bound for $\etammnl$.} We next show the matching upper bound
\beq \label{eq:ubound}
   K_r^*(\etammnl)  \equiv \sup_{\cL_r} \kappa((\ell_i),\etammnl)  \leq \frac{1+ \sqrt{\frac{\gamma}{\gamma+1}}}{1- \sqrt{\frac{\gamma}{\gamma+1}}}, \qquad \forall \eta.
\eeq
which will complete the proof.

Recall that $\kappa((\ell_i); \eta)$ denotes the $\kappa$-loss with underlying spikes $(\ell_i)$
and shrinker $\eta$, and that $\etammnl = \eta_m^*(\cdot; \infty, \gamma)$. With $\ell_{r+1} \equiv1$,
\begin{eqnarray*}
    \kappa((\ell_{i})_{i=1}^r,\etammnl) &=& \frac{\max_{i=1}^{r+1} \nu_{+}(\ell_{i},\etammnl(\lambda(\ell_i)))}{\min_{i=1}^{r+1} \nu_{-}(\ell_{i},\etammnl(\lambda(\ell_i)))}.
\end{eqnarray*}
It follows that
\beq \label{eq:kappabnd}
\sup_{\cL_r} \kappa((\ell_i),\etammnl)  \leq \frac{\sup_{\ell > 1} \nu_{+}(\ell,\etammnl(\lambda(\ell)))}{\inf_{\ell>1} \nu_{-}(\ell_{1},\etammnl(\lambda(\ell)))} .
\eeq 

By Lemma \ref{lem:increasingwrtELL}, 
$\nu_{+}(\ell,\eta_m^*(\lambda(\ell), \ell_1)) \leq  1 + \sqrt{\frac{\gamma}{1+\gamma}}$.
Hence
\beq \label{eq:nuplusbnd}
   \sup_{\ell > 1} \nu_{+}(\ell,\etammnl(\lambda(\ell))) \leq 1 + \sqrt{\frac{\gamma}{1+\gamma}}.
\eeq
We next work on  the companion lower bound 
\beq \label{eq:numinusbnd}
  \inf_{\ell > 1} \nu_{-}(\ell,\etammnl(\lambda(\ell))) \geq 1 - \sqrt{\frac{\gamma}{1+\gamma}}.
\eeq 

For each pair $(\ell, \ell_1)$ with $\ell  < \ell_1$  we have, by construction of $\eta_m^*$,
\begin{eqnarray*}
\nu_{-}(\ell,\eta_m^*(\lambda(\ell); \ell_1)) &\geq& \nu_{-}(\ell_1,\eta_1^*(\lambda(\ell_1); \ell_1)).
\end{eqnarray*}

We can take the limit of both sides of the equation above and we have
\begin{eqnarray}\label{eq:nu_comp}
\lim_{\ell_1 \goto \infty} \nu_{-}(\ell,\eta_m^*(\lambda(\ell); \ell_1)) &\geq& \lim_{\ell_1 \goto \infty} \nu_{-}(\ell_1,\eta_1^*(\lambda(\ell_1); \ell_1)).
\end{eqnarray}

Now observe that as $\nu_{-}(\ell,\eta_m^*(\lambda(\ell); \ell_1))$ is a continuous function of $\eta_m^*(\lambda(\ell); \ell_1)$, we have:
\begin{eqnarray*}
\lim_{\ell_1 \goto \infty} \nu_{-}(\ell,\eta_m^*(\lambda(\ell); \ell_1)) &=& \nu_{-}(\ell, \lim_{\ell_1 \goto \infty} \eta_m^*(\lambda(\ell); \ell_1)) \\
&=& \nu_{-}(\ell, \eta_{MM}(\lambda(\ell))). \qquad \mbox{[by \eqref{eq:limit_etamm}]}
\end{eqnarray*}
Substituting this in the LHS of \eqref{eq:nu_comp} yields
\begin{eqnarray*}
\nu_{-}(\ell, \eta_{MM}(\lambda(\ell)))  &\geq& \lim_{\ell_1 \goto \infty} \nu_{-}(\ell_1,\eta_1^*(\lambda(\ell_1), \ell_1)) \\
&=& 1 - \sqrt{\frac{\gamma}{1+\gamma}}.
\end{eqnarray*}
Since this inequality holds for all $\ell$, we can take an infimum over all $\ell > 1$ of both sides and have
\begin{eqnarray*}
\inf_{\ell > 1} \nu_{-}(\ell, \eta_{MM}(\lambda(\ell)))  \geq 1 - \sqrt{\frac{\gamma}{1+\gamma}}.
\end{eqnarray*}
This proves \eqref{eq:numinusbnd}. The desired upper bound (\ref{eq:ubound}) follows from combining the last display, and
the earlier (\ref{eq:nuplusbnd}).

The matching upper and lower bounds 
(\ref{eq:lboundB})-(\ref{eq:ubound}) 
complete the proof of parts a and c  of the Theorem. Property (\ref{eq:asyinequality}) of part b 
is established by (\ref{eq:numinusbnd}). 
\end{proof}

\subsection*{Proof of Theorem \ref{thm:worst-case-forecast}}

\begin{lem} \label{lem:blockstilde} 
Under {\bf [Spike]} and {\bf [PGA]},
fix spike values $(\ell_i)_{i=1}^r$ and 
corresponding shrunken values $(\eta_i)$.
With $\hS = \hS(\eta)$,
let  $\tD_n$ denote the $p$-by-$p$ matrix:
\[ 
\tD_n =   \hS_n^{-1/2} \Sigma \hS_n^{-1/2}.
\]
Let $\teta_i = \frac{1}{\sqrt{\eta_i}} - 1$ and 
let $\tD^a( (\ell_i), (\teta_i)) $ denote the
deterministic block-diagonal $p \times p$ matrix  
\[
\tD^a = \left[\begin{array}{ccccc}
\tA(\ell_{1},\teta_1)\\
 & \tA(\ell_{2},\teta_2)\\
 &  & \tA(\ell_{3},\teta_3)\\
 &  &  & ...\\
 &  &  &  & \tA(\ell_{r},\teta_{r})
\end{array}\right]\varoplus I_{p-2r} ,
\]
where the 2-by-2 matrices $\tA(\ell,\teta) \equiv \tA(\ell,\teta; \gamma)$ are defined via:
\[
  \tA(\ell,\teta) = I + a_1 uu' + a_2 (ue' +eu')  + a_3 ee';
\]
here $u = (c,s)'$ as earlier,
with $c=c(\ell;\gamma)$ and $s =\sqrt{1-c^{2}}$. Also
$e = (1,0)'$, and
\[
     a_1(\ell,\eta)  = 2 \teta + \teta^2 + \teta^2 \elld c^2 ; \qquad a_2(\ell,\eta)  = \teta \elld c ; \qquad a_3 = \elld.
\]
We have, as $n \to \infty$, almost surely and in probability,
 the convergence in Frobenius norm:
\begin{equation} \label{eq:TDeltaConv}
   \| \bW'_n \tD_n \bW_n - \tD^a( (\ell_i) , (\eta_i) ) \|_F \goto 0,
\end{equation}
Here $ \bW_n$  is the basis matrix
constructed  in Lemma \ref{lem:blocks}.
\end{lem}

\begin{proof}
Recall that $\hS$ and $\Sigma$ are asymptotically simultaneously block-diagonalized by
basis matrix $\bW_n$. The asymptotic blocks of $\hS$ take the form
\[
B(\ell,\eta):=\left[\begin{array}{cc}
\eta c^{2}+s^{2} & (\eta-1)cs\\
(\eta-1)cs & c^{2}+\eta s^{2}
\end{array}\right] .
\]
Thus  an individual
$2 \times 2$ block $\tA$ has the form 
\begin{eqnarray*}
    \tA(\ell,\teta) &=&  B(\ell,\eta)^{-1/2} \cdot 
diag(\ell,1) 
 \cdot B(\ell,\eta)^{-1/2} , \\
                         &=&  ( I + \teta uu')  \cdot (I + \elld ee') \cdot (I + \teta uu') , 
\end{eqnarray*}
where $u = (c,s)'$ as earlier, is the top eigenvector of $A(\ell,\eta)$,
$e = (1,0)'$ and $\teta = \frac{1}{\sqrt{\eta}} - 1$. We can rewrite this as
\[
    \tA(\ell,\teta) = I + a_1 uu' + a_2 (ue' +eu')  + a_3 ee',
\]
where 
\[
     a_1 = 2 \teta + \teta^2 + \teta^2 \elld c^2 ; \qquad a_2 = \teta \elld c , \qquad a_3 = \elld.
\]
\end{proof}

\begin{proof}
{\bf (of Theorem \ref{thm:worst-case-forecast}).}
By Lemma \ref{lem:ReduceConditionNumber}, 
the ratio of  ideal oracle Sharpe ratio to the achieved Sharpe Ratio 
$\frac{\SR(\Sigma;\mu,\Sigma)}{\SR(\hat{\Sigma};\mu,\Sigma)}$
can be bounded by
\begin{eqnarray*}
\RSRG_n &=& \sup_\mu \frac{ \sqrt{\mu'\Sigma^{-1} \mu} \cdot \sqrt{\mu'\hS^{-1} \Sigma \hS^{-1}  \mu}}{ \sqrt{\mu' \hS^{-1} \mu}}    \\
          &=&  \sup_x \frac{\sqrt{x' \tD_n x} \cdot \sqrt{x'\tD_n^{-1}x}}{\Vert x\Vert_{2}^{2}}, \\
\end{eqnarray*}
where we used the change of variables $x = \hS^{-1/2} \mu$, 
and where $\tD = \hS^{-1/2} \Sigma \hS^{-1/2}$.
Underlying the Kantorovich inequality is the  fact that,
if $\tnu_{\pm,n}$ denote the extremal eigenvalues of $\tD_n$ and
$\tv_{\pm,n}$ the corresponding eigenvectors, then the maximizer $x$ for 
the last display is
\[
   x_{*,n} = (\tv_{+,n} + \tv_{-,n})/\sqrt{2},
\]
and the maximal value is 
\[
   \RSRG_n = \frac{1}{2} (2 + \frac{\tnu_{+,n}}{\tnu_{-,n}} + \frac{\tnu_{-,n}}{\tnu_{+,n}})^{1/2}.
\]
We will show that this tends to the limit:
\beq \label{eq:AsyRSRG}
  \RSRG_n \goto_{a.s.}  \RSRG^*(\ell_1;\gamma) \equiv \frac{1}{2} (2 + \kappa_1^*(\ell_1;\gamma) +  1/\kappa_1^*(\ell_1;\gamma) )^{1/2} .
\eeq
We will also directly construct an asymptotic optimizer $x^a_n$, obeying
\beq \label{eq:AsyRSRGOpt}
\frac{\sqrt{(x_n^{a})' \tD_n x_n^a} \cdot \sqrt{(x_n^a)'\tD_n^{-1}x_n^a}}{\|x_n^a \|_2^2}  \rightarrow_{a.s.}  \RSRG^*(\ell_1;\gamma) .
\eeq
It follows that an asymptotically least-favorable user forecast is given by $\mu_n^a = \hS^{1/2} x_n^a$.
The  construction makes  $x_n^a$ an explicit linear combination of
$U_{1},V_{1,n}$, the top theoretical and empirical eigenvectors.
Applying the matrix $\hS^{1/2}$ approximately preserves this property,
because $\hS$ is asymptotically block diagonal in the basis $\bW_n$.
Hence an asymptotically least-favorable forecast 
has the form $\mu_n^a = \alpha_1 U_{1} + \alpha_2 V_{1,n} + o_P(1)$,
with fixed coefficients $\alpha_1,\alpha_2$,
and the Theorem follows.

We recall two facts about $\Delta^a$. Lemma  \ref{lem:blocks},  
that using the basis matrix $\bW_n$ 
constructed we have
\beq \label{eq:DeltaConv}
   \| \bW'_n \Delta_n \bW_n - \Delta^a \|_F \goto 0, \qquad n \goto \infty ,
\eeq
where $\Delta^a$ is the block-structured asymptotic pivot.
The non-unit eigenvalues of $\Delta_n$ converge to those of $\Delta^a$.
We constructed $\eta^*$ the no-sticking-out condition, for $\ell \in [1,\ell_1]$:
\begin{eqnarray*}
\nu_+(A(\ell,\eta^*(\ell))) &\leq& \nu_+(A(\ell_1,\eta^*(\ell_1)))\\
\nu_-(A(\ell,\eta^*(\ell))) &\geq& \nu_-(A(\ell_1,\eta^*(\ell_1)));
\end{eqnarray*}
in this way the extreme eigenvalues of $\Delta^a$ coincide with those
of the 2-by-2 block $A(\ell_1,\eta^*(\ell_1))$. Letting $\nu_{\pm,n}$ denote the
extreme eigenvalues of $\Delta_n$, (\ref{eq:DeltaConv}) implies that
$\nu_{\pm,n} \goto_{a.s.} \nu_\pm(A(\ell_1,\eta^*(\ell_1))$ as $n \goto \infty$.

Using this we can settle  (\ref{eq:AsyRSRG}).
The eigenvalues of $\Delta_n =  \Sigma^{-1/2} \hS \Sigma^{-1/2}$
are reciprocal to those of $\tD_n$. 
Hence $\tnu_{+,n} \goto 1/\nu_-(A(\ell_1,\eta^*(\ell_1))$, $\tnu_{-,n} \goto 1/\nu_+(A(\ell_1,\eta^*(\ell_1))$.
It follows that $\RSRG_n \goto_{a.s.} \RSRG^*$ as $n \goto \infty$.

To establish (\ref{eq:AsyRSRGOpt}), 
we use Lemma \ref{lem:blockstilde} .
Let  $\tu_\pm \in \bR^2$  denote the top and bottom 
eigenvectors of $ \tA(\ell,\teta)$.
From the convergence (\ref{eq:TDeltaConv}) we have, almost surely as $n \goto \infty$,
\[
\{ \tnu_{+,n}, \tnu_{-,n} \}  \goto  \{ \tu_+' \tA(\ell_1,\teta_1)\tu_+,   (\tu_-)' \tA(\ell_1,\teta_1)\tu_- \},
\]
where $\eta_1 = \eta^*(\ell_1)$ and $\teta_1 = 1/\sqrt{\eta_1} -1$. 
Letting $\bW_{n,2}$ denote the sub matrix of the first two columns  of $\bW$,
we can define the approximate eigenvectors
\[
 \tv_{\pm,n}^a = \bW_{n,2} \tu_\pm,  
\]
and the approximate optimum
\[
    x_n^a = (\tv_{+,n}^a + \tv_{-,n}^a)/\sqrt{2}.
\]
From the convergence (\ref{eq:TDeltaConv}), $x_n^a  \tD_n x_n^a  \goto  y' \tA(\ell_1,\teta_1)y$, where $y=  ( \tu_+ +  \tu_-)/2$.
Since the eigenvalues of $\Delta^a$ are bounded below by 1, (\ref{eq:TDeltaConv}) also implies
 $x_n^a  \tD_n^{-1} x_n^a \goto  y \tA^{-1}(\ell_1,\teta_1)  y$.
(\ref{eq:AsyRSRGOpt}) follows. 
\end{proof}

\section{Details behind Performance Comparisons}

This appendix provides  details behind  our numerical calculations 
of maximum relative regret for Figures \eqref{fig:MaximumRegretVGammaSoftThresh} 
and \eqref{fig:MaximumKappaLossVGammaSoftThresh}. 
In this section, we focus on the $\kappa$-loss regret, 
as $\RSRG$-regret is computed similarly.

Fix $\gamma \in (0,\infty)$. Let $\cL(\ell)$ 
denote the class of all spike configurations \-- of any finite $r\geq 1$ \-- where $\ell_1 = \ell$. 
For any nonlinearity $\eta$, we have the following identity for the maximum $\kappa$-loss regret:
\begin{eqnarray} \label{eq:max_regret_computation}
\mbox{MaxReg}(\eta,\gamma) 
&=& \max_{\ell \in [1, \infty)} \max_{(\ell_i) \in \cL(\ell)} \mbox{Reg}[\eta,(\ell_i)] \nonumber \\
&=& \max_{\ell \in [1, \infty)} \max_{(\ell_i) \in \cL(\ell)} 100 \cdot  \left ( 1- \frac{\kappa[\eta^*,(\ell_i)]}{\kappa[\eta,(\ell_i)]} \right ) \nonumber \\
&=&\max_{\ell \in [1, \infty)} 100 \cdot  \left ( 1- \frac{\kappa_1^*(\ell;\gamma)}{\max_{(\ell_i) \in \cL(\ell)} \kappa[\eta,(\ell_i)]} \right ).
\end{eqnarray}

 Therefore, computation of maximum regret boils down to computing 
\begin{eqnarray} \label{eq:maxRegret_computation}
\max_{(\ell_i) \in \cL(\ell)} \kappa[\eta,(\ell_i)] &=& \frac{\max_{1 \leq \ell_0 \leq \ell} { \nu_{+}}(A(\ell_0, \eta(\ell_0)))}{\min_{1 \leq \ell_0 \leq \ell} { \nu_{-}}(A(\ell_0, \eta(\ell_0)))}.
\end{eqnarray}
The figures use a grid of $\ell$-values: 
$G \equiv \{1, 1.01, 1.02, \cdots, 99.99, 100, 10^5\}$. 
On each point {$ \ell \in G$}, we evaluate \eqref{eq:maxRegret_computation}. This provides us with a numerical value for
\begin{equation} \label{eq:maxRegret_computation2}
\max_{(\ell_i) \in \cL(\ell)} 100 \cdot  \left ( 1- \frac{\kappa_1^*(\ell;\gamma)}{\kappa[\eta,(\ell_i)]} \right )
\end{equation}
for each { $\ell \in G$}. 
We numerically evaluate $\mbox{MaxReg}(\eta,\gamma)$ by taking the maximum of \eqref{eq:maxRegret_computation} over { $\ell \in G$}. 
For Figures \eqref{fig:MaximumRegretVGammaSoftThresh} and \eqref{fig:MaximumKappaLossVGammaSoftThresh} we calculate the maximum regret curves by computing $\mbox{MaxReg}(\eta,\gamma)$ for $\gamma \in G' \equiv \{0.01, 0.03, \cdots, 2\}$.

\bibliographystyle{plain}
\bibliography{SpikedCovar}
\end{document}